\documentclass[reqno,oneside,a4paper]{amsart}

\pdfoutput=1

\usepackage[utf8]{inputenc}
\usepackage[bookmarksdepth=3,
            colorlinks=true,
            urlcolor=blue,
            citecolor=red,
            linkcolor=blue,
            anchorcolor=red,
            linktoc=section]{hyperref} 
\usepackage{amsmath,amsthm,amssymb}
\usepackage{enumitem}
\usepackage{microtype}

\usepackage[all]{xy}

\newtheorem{thm}{Theorem}[section]
\newtheorem{lem}[thm]{Lemma}
\newtheorem{cor}[thm]{Corollary}
\newtheorem{defn}[thm]{Definition}
\newtheorem{prop}[thm]{Proposition}

\newtheorem*{thm*}{Theorem}

\theoremstyle{definition}
\newtheorem{ex}[thm]{Example}
\newtheorem{rmk}[thm]{Remark}

\newcommand{\fR}{\mathbb{R}} 
\newcommand{\ol}[1]{\overline{#1}} 

\let\mbold\mathbf

\makeatletter
\let\L\@undefined
\makeatother
\newcommand{\L}{\mbold L}
\newcommand{\A}{\mbold A}

\newcommand{\U}{\mbold U}
\newcommand{\V}{\mbold V}
\renewcommand{\ll}{L}
\renewcommand{\aa}{A}
\newcommand{\bb}{B}
\newcommand{\uu}{U}

\newcommand{\g}{\mathfrak{g}}
\newcommand{\h}{\mathfrak{h}}

\newcommand{\pb}{!} 
\newcommand{\fb}{\mbold{GL}} 

\newcommand{\so}{\mbold{s}} 
\newcommand{\ta}{\mbold{t}} 
\newcommand{\Xlin}{\mathfrak{X}^{\mathrm{lin}}}
\newcommand{\Tlin}{T^{\mathrm{lin}}}
\newcommand{\TD}{\mathcal{L}} 
\newcommand{\cf}{\omega} 
\newcommand{\bmap}{\omega} 
\newcommand{\amod}{} 

\newcommand{\toto}{\rightrightarrows}
\newcommand{\param}{t}
\newcommand{\unit}{\mbold{1}} 
\newcommand{\expon}[1]{\exp (#1)}
\newcommand{\inv}{\mbold{i}} 
\newcommand{\generic}{\gamma}
\newcommand{\genericU}{\upsilon}
\newcommand{\bisection}{\Sigma}
\newcommand{\Add}{\kappa} 
\newcommand{\di}{D} 
\newcommand{\op}{\text{op}}

\newcommand{\atiyah}{\mathfrak r}
\newcommand{\atiyahTensor}{\mathfrak R}
\newcommand{\atiyahLieAlgebroid}{\alpha}
\newcommand{\atiyahTorsionClass}{\atiyahLieAlgebroid}

\DeclareMathOperator{\Bis}{Bis}
\DeclareMathOperator{\Ad}{Ad}
\DeclareMathOperator{\ad}{ad}

\DeclareMathOperator{\End}{End}
\DeclareMathOperator{\Der}{Der}
\DeclareMathOperator{\id}{id}
\DeclareMathOperator{\Lie}{Lie}
\DeclareMathOperator{\Id}{Id} 
\DeclareMathOperator{\pr}{pr}
\DeclareMathOperator{\m}{m} 
\DeclareMathOperator{\rk}{rk}
\DeclareMathOperator{\Hom}{Hom}

\newcommand{\diff}{{\mathrm d}}

\newcommand{\D}{\mathcal{D}} 

\newcommand{\lvf}[1]{L(#1)}
\newcommand{\rvf}[1]{R(#1)}
\usepackage{amsmath}

\title{Invariant connections and PBW theorem for Lie groupoid pairs}

\thanks{Research partially supported by the Fonds National de la Recherche, Luxembourg, through the AFR Grant PDR 2012-1 (Project Reference 3966341).}

\author{Camille Laurent-Gengoux}
\address{Université de Lorraine, CNRS, IECL, F-57000 Metz, France}
\email{camille.laurent-gengoux@univ-lorraine.fr}
\author{Yannick Voglaire}
\address{Université du Luxembourg, Mathematics Research Unit, Maison du Nombre, 6 avenue de la Fonte, L-4364 Esch-sur-Alzette, Luxembourg}
\email{yannick.voglaire@uni.lu}

\subjclass[2010]{53C05,53C30,53C12} 
\keywords{Atiyah classes, Lie groupoids, homogeneous spaces, linearization, Poincaré-Birkhoff-Witt theorem, foliations, equivariant principal bundles}

\begin{document}

\begin{abstract}
To a closed wide Lie subgroupoid $\A$ of a Lie groupoid $\L$, i.e.\ a Lie groupoid pair, we associate an Atiyah class which we interpret as the obstruction to the existence of $\L$-invariant fibrewise affine connections on the homogeneous space $\L/\A$.
For Lie groupoid pairs with vanishing Atiyah class, we show that the left $\A$-action on the quotient space $\L/\A$ can be linearized.

In addition to giving an alternative proof of a result of Calaque about the Poincar\'e--Birkhoff--Witt map for Lie algebroid pairs with vanishing Atiyah class, this result specializes to a necessary and sufficient condition for the linearization of dressing actions, and gives a clear interpretation of the Molino class as an obstruction to the simultaneous linearization of all the monodromies.

In the course of the paper, a general theory of connections on Lie groupoid equivariant principal bundles is developed.
\end{abstract}

\maketitle

\tableofcontents


\section{Introduction}
\label{sec:introduction}

Invariant connections form an important tool for the study of homogeneous spaces, and their study goes back to the work of É.\ Cartan on Lie groups.
Generalizing the canonical invariant connection on a symmetric space, Nomizu studied the existence of invariant connections on reductive homogeneous spaces \cite{nomizu_invariant_1954}.
Around 1960, Nguyen Van Hai \cite{nguyen-van-hai_conditions_1964}, Vinberg \cite{vinberg_invariant_1960} and Wang \cite{wang_invariant_1958} independently characterized the set of $G$-invariant affine connections on a not necessarily reductive homogeneous space $G/H$, and computed necessary and sufficient conditions for their existence.
In the connected case, the obstruction was given a cohomological meaning in \cite{nguyen-van-hai_relations_1965} as an element in the first Lie algebra cohomology $H^1(\h,\Hom(\g/\h \otimes \g/\h,\g/\h))$.
Recently, this class was rediscovered by Calaque, C\u{a}ld\u{a}raru and Tu \cite{calaque_pbw_2013} and shown to be the obstruction to a Poincaré--Birkhoff--Witt-type theorem for inclusions of Lie algebras.
Bordemann in \cite{bordemann_atiyah_2012}, among other results, realized and explained the link between the two approaches and gave a geometric interpretation of the PBW theorem using invariant connections.

In this paper, we explore some aspects of the scarcely studied class of homogeneous spaces of Lie groupoids, see \cite[Section 8]{liu_dirac_1998} and \cite[Section 3]{moerdijk_integrability_2006}.
This research was triggered by recent papers of Chen, Stiénon and Xu \cite{chen_atiyah_2016}, who generalize Atiyah classes \cite{atiyah_complex_1957} to inclusions of Lie algebroids, and of Calaque \cite{calaque_pbw_2014}, who generalizes the PBW-type theorem to that case as well.
One of our intents is to reprove geometrically the latter theorem. We develop along the way the study of invariant connections on homogeneous spaces of Lie groupoids and on equivariant principal groupoid bundles.

A related study, with a somewhat different point of view, was carried out in \cite{laurent-gengoux_kapranov_2014}: while obviously the authors of op.\ cit.\ mostly focus on the case of non-vanishing Atiyah class, they also consider the vanishing case, where the corresponding Kapranov $dg$-manifold is shown to be \emph{linearizable}.
This means that it can be represented by an $L_\infty$-algebra structure that only admits a $1$-ary bracket while all higher brackets vanish, including  the bilinear bracket. 
As a Kapranov $dg$-manifold, or as an $L_\infty$-algebra, it is therefore of limited interest, but the fact that it is trivial has interesting consequences. The authors interpret it geometrically as meaning that the natural left $\aa$-action on $ \L/\A$ (with $\L$ and $\A $ local Lie groupoids integrating $ \ll$ and $ \aa$, respectively) is formally linearizable. They recover as a corollary an interpretation of this vanishing found in \cite{calaque_pbw_2014} in terms of equivariance of the Poincaré--Birkhoff--Witt map. In this paper, we clarify this geometric interpretation, consider the Lie groupoid $\A$-action instead of the infinitesimal $\aa$-action and replace ``formal'' by ``semi-local''. Semi-local meaning here ``in a neighborhood of the zero section'', i.e.\ in a neighborhood of the common base manifold $M$ of both $\L$  and $\A$.

Our main tool in this article will be Lie groupoid pairs, i.e.\ pairs $(\L,\A) $ with $\L$ a Lie groupoid and $\A \subset \L$ a closed Lie subgroupoid over the same manifold $M$, or their local counterparts. 
As mentioned above, their infinitesimal counterparts have been recently studied \cite{chen_atiyah_2016,laurent-gengoux_kapranov_2014} under the name of Lie algebroid pairs, which are pairs $(\ll,\aa)$ made of a Lie algebroid $\ll$ together with a Lie subalgebroid $\aa$ over the same base. 
In the transitive case, the latter were already studied, although for other reasons, by Kubarski \emph{et al.}\ in \cite{balcerzak_primary_2001,kubarski_algebroid_1998}.
Lie algebroid pairs are an efficient manner to unify various branches of differential geometry where (regular) \emph{transverse} structures appear, as can be seen from the following list of examples:
\begin{enumerate}
  \item \emph{Lie subalgebras}.
    Let $\g$ be a Lie algebra. For any Lie subalgebra $\h$ of $\g$, the pair $(\g,\h) $ is a Lie algebroid pair. 
    
  \item \emph{Foliations}.
    Let $\ll=TM$ be the tangent bundle Lie algebroid of a manifold $M$.
    For any integrable distribution $\aa \subset TM$, i.e.\ any foliation on the manifold $M$, the pair $(\ll,\aa) $ is a Lie algebroid pair.
    
  \item \emph{Non-commutative integrable systems}.
    Let $\ll = T^*M$ be the cotangent algebroid of a Poisson manifold
    $(M,\pi)$, see \cite{crainic_integrability_2004}. 
    Consider a coisotropic foliation on $M$, i.e.\ a foliation whose leaves are all coisotropic submanifolds.
    Covectors vanishing on the tangent space of the coisotropic foliation form a Lie subalgebroid $\aa $ of $ \ll$, and the pair $(\ll,\aa) $ is a Lie algebroid pair.
    In particular, the coisotropic foliation can be chosen to be a regular integrable system (in the sense of \cite[Chapter 12]{laurent-gengoux_poisson_2013}) or a non-commutative integrable system in the sense of \cite{fernandes_global_2015}.

  \item \emph{Manifolds with a Lie algebra action}.
    If $M$ is a manifold with an action of a Lie algebra $\g$, one can build a matched pair \cite{mokri_matched_1997} of Lie algebroids $\ll=TM \times (\g\ltimes M)$ where $\g\ltimes M$ is the action Lie algebroid.
    Taking $\aa$ to be this action Lie algebroid yields a Lie algebroid pair $(\ll,\aa)$.
    
  \item \emph{Poisson manifolds with a Poisson $\g$-action}.
    If $P$ is a Poisson manifold with a Poisson $\g$-action, Lu \cite{lu_poisson_1997} defines a matched pair of Lie algebroids $\ll=T^*P \times (\g\ltimes P)$.
    Taking $\aa$ to be the action Lie algebroid yields a Lie algebroid pair $(\ll,\aa)$.
    
  \item Enlarging the setting from real Lie algebroids to complex Lie algebroids \cite{weinstein_integration_2007}, one could add \emph{complex manifolds} among the examples, since an almost complex structure $ J$ on a smooth manifold $X$ is complex if and only if $(T_{\mathbb C}X, T^{0,1}X)$ is a Lie algebroid pair \cite{laurent-gengoux_holomorphic_2010}. 
\end{enumerate}
There is a canonical, Bott-type, $\aa$-module structure on $\ll/\aa$ which is fundamental for the study of the transverse geometry of $\ll$ with respect to $\aa$. To the best of our knowledge, it was first considered for a general Lie algebroid pair in \cite[Example 4]{crainic_differentiable_2003}, and extensively studied in \cite{chen_atiyah_2016}.
For foliations, this $\aa$-module structure is the Bott connection, while for Lie algebras, it is simply the $\aa$-action on $\ll/\aa$ induced by the adjoint action. 
When $\ll$ is the double of a Lie bialgebra $\aa$, then $\ll/\aa \simeq \aa^*$ with the $\aa$-module structure defined by the coadjoint action.

As in Wang's original work \cite{wang_invariant_1958}, we prove our results first in the context of principal bundles. 
To this end, we introduce a notion of equivariant principal bundle of Lie groupoids (see Definition \ref{defn:epb}).
Such bundles include generalized morphisms as particular examples, but we are mostly interested in those which are not of that kind.
It turns out that, just like generalized morphisms, equivariant principal bundles may be seen as ``morphisms'' between Lie groupoids, whose composition is associative up to bi-equivariant diffeomorphisms (see Proposition~\ref{prop:epb-bicategory}).

For a vector bundle $E$ and an equivariant principal bundle $P$ with structure groupoid the frame groupoid $\fb(E)$, we first extend to this context the equivalence between connection 1-forms on $P$ and connections on the associated bundle $P(E)$ (see Theorem \ref{thm:connection-forms-connections-on-associated-bundles}).
We then establish, for equivariant principal bundles over a homogeneous space $\L/\A$, an equivalence between fibered connection 1-forms and some Lie algebroid connections (see Theorem \ref{thm:general}).
Combining these two results, we arrive at a characterization of invariant connections on vector bundles over homogeneous spaces of Lie groupoids (see Theorems \ref{thm:L-conn-on-E--conn-on-assoc-bundle}).
As a special case, we obtain the following generalization of Wang's theorem (see Theorem \ref{thm:L-conn-on-LoverA--conn-on-tgt-bundle}):
\begin{thm*}
  Let $(\L, \A)$ be a Lie groupoid pair over $M$, with Lie algebroid pair $(\ll,\aa)$.
  There is a bijective correspondence between
  \begin{enumerate}[label=(\arabic*)]
    \item $\A$-compatible $\ll$-connections on $\ll/\aa$, and
    \item $\L$-invariant fibrewise affine connections on $\L/\A \to M$.
  \end{enumerate}
\end{thm*}

Given a Lie groupoid pair $(\L,\A)$ over $M$ and an $\A$-module $E$, the obstruction to the existence of an invariant connection on the associated vector bundle $\frac{\L\times_M E}{\A}$ is a class $\alpha_{(\L,\A),E}$ in the degree 1 Lie groupoid cohomology of $\A$, that we call the \emph{Atiyah class} of the $\A$-module $E$ with respect to the Lie groupoid pair $(\L,\A)$.
We relate it to the Atiyah class of $E$ with respect to the Lie algebroid pair $(\ll,\aa)$ as introduced in \cite{chen_atiyah_2016} (see Proposition \ref{prop:vanest}).
We also show its invariance with respect to Morita equivalence of pairs of Lie groupoids (see Theorem \ref{thm:Morita-equivalence}).
The quotient bundle $\ll/\aa$ is naturally an $\A$-module, and its Atiyah class is called the Atiyah class of the Lie pair $(\L,\A)$.

It should be noted that, although we characterize the connection forms on $\L$-equivariant principal $\U$-bundles $P$ over a homogeneous space $\L/\A$ for all Lie pairs $(\L,\A)$ and all Lie groupoids $\U$ (see Definition~\ref{defn:epb} and Theorem~\ref{thm:general}), we only obtain a cohomological obstruction to their existence for \emph{transitive} $\U$ (see Proposition \ref{prop:atiyah-class-transitive-case}).
Indeed, in that case, the obstruction lies in $H^1(\A,(\ll/\aa)^*\otimes P(\uu_0))$ where $\uu_0$ is the isotropy subalgebroid of the Lie algebroid $\uu$ of $\U$.
In order to make sense of this obstruction in the general case, we would need to work in the context of representations up to homotopy, which we reserve for later work.

Provided that the Atiyah class of $(\L,\A) $ is zero, we construct successively: 
\begin{enumerate}
  \item An $\L$-invariant fibrewise affine connection on the fibered manifold $\L/\A \to M$.

  \item An $\A$-equivariant exponential map from $ \ll/\aa$ to $\L/\A$, which is well defined in a neighborhood of the zero section $\imath :M \to \ll/\aa $ and a diffeomorphism onto its image (see Theorem~\ref{thm:exponential}).

  \item An $\A$-equivariant Poincar\'e--Birkhoff--Witt isomorphism from $ \Gamma(S(\ll/\aa))$ to $ U(\ll)/U(\ll)\cdot\Gamma(\aa)$ (see Theorem \ref{thm:calaque}).
\end{enumerate}
The exponential map above is indeed the exponential of the connection announced in the first item, and its infinitesimal jet is the Poincar\'e--Birkhoff--Witt isomorphism. The Poincar\'e--Birkhoff--Witt  isomorphism that we eventually obtain gives an alternative proof of Theorem 1.1 in \cite{calaque_pbw_2014}, valid under the assumption that the Lie algebroid pair $ (\ll,\aa)$ integrates to a Lie groupoid pair $ (\L,\A)$\footnote{We also have a version with local Lie groupoids which allows to drop the integrability condition, see Theorem~\ref{thm:exponential-local}.}:

\begin{thm*}
	The Atiyah class of a Lie groupoid pair $(\L,\A)$ vanishes if and only if there exists a $\Bis(\A)$-equivariant filtered coalgebra isomorphism from $\Gamma(S(\ll/\aa))$ to $U(\ll) / U(\ll) \cdot \Gamma(\aa)$.
\end{thm*}

The latter theorem specializes to give a cohomological interpretation of the linearization of dressing actions (see Corollary~\ref{cor:dressing-action-linearizable}) and an interesting result about monodromies of foliations (see Theorem~\ref{thm:monodromies-main}) that we quote here.

\begin{thm*}
	Let ${\mathcal F}$ be a regular foliation on a manifold $ M$. The Atiyah class of the Lie algebroid pair $(TM,T{\mathcal F})$ vanishes if and only if all the monodromies are simultaneously linearizable.
\end{thm*}

The paper is structured as follows.
In Section \ref{sec:preliminaries}, we recall basic facts about Lie groupoids and Lie algebroids.
In Section \ref{sec:equivariant-principal-bundles}, we review generalized morphisms of Lie groupoids, before introducing equivariant principal bundles and associated vector bundles.
We describe an action of the tangent groupoid to a Lie groupoid on the anchor map of its Lie algebroid that plays for us the role of an adjoint action.
In Section \ref{sec:global-atiyah-class}, we show that $\ll/\aa$ is an $\A$-module, introduce the Atiyah classes of generalized morphisms and of $\A$-modules, and prove their Morita invariance.
In Section \ref{sec:main-results}, we prove our main results about connections on homogeneous spaces of Lie groupoids.
In Section \ref{sec:PBW-theorem}, we use invariant connections to prove a Poincaré--Birkhoff--Witt theorem.
In Section \ref{sec:local-Lie-groupoids}, we investigate how our results can be adapted to local Lie groupoid pairs in order to drop some integrability conditions.
In Section \ref{sec:examples}, we work out applications to Lie groups and foliations.


\subsection{Acknowledgments}
\label{ssec:acknowledgments}

Both authors would like to thank Penn State University for its hospitality. We are in particular extremely grateful to Professor Ping Xu, who inspired us the present work. We would also like to thank Martin Bordemann, Damien Calaque, Rui Loja Fernandes, Benoît Jubin, Norbert Poncin, Mathieu Stiénon and Matteo Tommasini for useful discussions. The first author would also like to thank the University of Marrakesh where, at a CIMPA conference, part of this work was presented in April 2015. We thank the referees for useful remarks.


\section{Preliminaries}
\label{sec:preliminaries}

This section recalls classical notions for Lie groupoids, and fixes our notation and conventions: most facts are basic, but are quite dispersed in the literature. 
We investigate in particular bisections, and introduce the very convenient operation $\star$ that we insist to be a convenient and pedagogical manner to deal with those objects, in particular when one has to see bisections as the group integrating sections of the Lie algebroids. We also introduce the operator $\Add$ that shall play a crucial rôle in the proofs, and that we invite the reader to understand as a formalization of the adjoint action of bisections on sections of the Lie algebroid.

Let us state some general conventions about vector bundles. 
Projections from vector bundles to their base manifold will be denoted by the letter $q$, with the total space added as a subscript, as in $q_E:E\to M$, if necessary.
For $E$ a vector bundle over a manifold $M$, we shall denote by $\Gamma(E)$ the space of global smooth sections of $E$ and by $\Gamma_{\mathcal U}(E)$ the space of smooth sections over an open subset ${\mathcal U} \subset M$.  
The fiber at a particular point $x \in M$ shall be denoted by $E_x$. For all $e \in \Gamma (E)$, $e_x \in E_x$ stands for the evaluation at $x\in M$ of the section $e$. 
We may also use the notation $\left. e \right|_x$. 
For $\epsilon \in E_x$, a section $e \in \Gamma(E)$ is said to be \emph{through $\epsilon$} if $e_x = \epsilon$. 
For $\phi: N \rightarrow M$ a smooth map, we denote by $\phi^* E$ the pullback of $E$ through $\phi$, i.e.\ the fibered product
\[
  \phi^* E = N \times^{\phi,q}_{M} E = \left\{ (y,\epsilon) \in N\times E \mid \phi(y) = q(\epsilon) \right\}.
\]
It is a vector bundle over $N$, with projection $q(y,\epsilon)=y$.
For every section $e\in \Gamma(E)$, the pullback of $e$ through $\phi$ is the section denoted by $\phi^*e$ and defined by 
$(\phi^*e)_y = (y, e_{\phi(y)})$
for all $y \in N$.


\subsection{Lie groupoids}
\label{ssec:lie-groupoids}

A \emph{groupoid} is a small category in which every morphism is invertible. 
The morphisms of a groupoid are called the arrows, and the objects are called the units.
The \emph{structure maps} of a groupoid are 
the source and target maps associating to an arrow its source and target objects respectively, 
the unit map associating to an object the unit arrow from that object to itself,
the inversion sending each arrow to its inverse,
and the multiplication sending two composable arrows to their composition.
A \emph{Lie groupoid} is a groupoid where the set $\L$ of arrows and the set $M$ of objects are smooth manifolds, all structure maps are smooth and the source and target maps are surjective submersions, see \cite{cannas_da_silva_geometric_1999,mackenzie_general_2005}.
The manifold $M$ of objects shall be referred to as the \emph{unit manifold}.
The unit map can be shown to be a closed embedding, and the unit manifold shall actually be considered as an embedded submanifold of $\L$.
We will often identify the Lie groupoid with its manifold of arrows and talk about ``a Lie groupoid $\L$ over a manifold $M$'', written shortly as $\L\toto M$.
For all the groupoids considered, the source map shall be denoted by $\so$, the target map by $\ta$, the unit map by $\unit$, the inverse map by $\inv$, and the multiplication by a dot. 
The convention in this paper shall be that the product $\generic_1 \cdot \generic_2$ of two elements $\generic_1,\generic_2 \in \L$ is defined if and only if $\so (\generic_1) = \ta (\generic_2)$. 

A \emph{Lie groupoid morphism} from a Lie groupoid $\L$ over $M$ to a Lie groupoid $\U$ over $N$ is a pair of smooth maps $\varphi:\L\to\U$ and $\varphi_0:M\to N$ such that $\so\circ\varphi=\varphi_0\circ\so$, $\ta\circ\varphi=\varphi_0\circ\ta$, and $\varphi(\generic_1\cdot \generic_2)=\varphi(\generic_1)\cdot\varphi(\generic_2)$ for all composable $\generic_1,\generic_2\in\L$.

\begin{ex}\label{ex:frameGroupoid}
The frame groupoid $\fb(E)$ of a vector bundle $E$ over a manifold $M$ is the Lie groupoid whose unit manifold is $M$ and whose set of arrows between two arbitrary points $x,y \in M$ is made of all invertible linear maps from $E_x$ to $E_{y}$. 
The source of such arrows is $x$ and their target is $y$.
\end{ex}

\subsubsection*{Modules}
For a Lie groupoid $\L$ over a manifold $M$, a (left) \emph{$\L$-module} is a vector bundle $E$ over $M$ equipped with a Lie groupoid morphism from $\L$ to $\fb(E)$. 
It is often convenient to see it as an assignment
\begin{equation}
\label{eq:Lie-groupoid-module}
  \L\times_M^{\so,q} E \to E: (\generic, e) \mapsto \generic \cdot e
\end{equation}
satisfying the usual axioms of a left action, see \cite{mackenzie_general_2005}. 

Let us denote by $\L_n$ the manifold of all $n$-tuples $(\generic_1, \dots, \generic_n) \in \L^n$ such that the product of any two successive elements is defined, i.e.\ such that $\so(\generic_i)=\ta(\generic_{i+1})$ for all $i=1,\dots,n-1$. 
The \emph{Lie groupoid cohomology} \cite{crainic_differentiable_2003} of an $\L$-module $E$ is the cohomology $H^\bullet(\L,E)$ of the complex
\begin{equation}
\label{eq:groupoid-cohomology-complex}
  \xymatrix{
    C^0(\L,E)
    \ar[r]^{\partial_0} &
    C^1(\L,E)
    \ar[r]^{\partial_1} &
    C^2(\L,E)
    \ar[r]^{\partial_2} &
    C^3(\L,E)
    \ar[r]^-{\partial_3} &
    \cdots
  }
\end{equation}
where
\begin{enumerate}
  \item $C^0(\L,E)$ is the space $\Gamma (E)$ of sections of $E$;

  \item for all $n \in {\mathbb N_*}$, $C^n (\L,E)$ is the space%
    \footnote{The space $C^n(\L,E)$ can also be described as the space of sections of the vector bundle $\ta^*E \to \L_n$ where $\ta:\L_n \to M$ stands for the map $ (\generic_1, \dots, \generic_n)  \mapsto \ta (\generic_1)$.}
    of smooth functions $F$ from $\L_n$ to $E$ such that 
    $F(\generic_1, \dots, \generic_n) \in E_{\ta (\generic_1)}$
    for all $ (\generic_1, \dots, \generic_n) \in \L_n$;
  
  \item for all $e\in \Gamma(E)$, $\partial_0 e$ is the element of $C^1(\L,E)$ defined by
    $\partial_0 e (\generic) = \generic \cdot e_{\so(\generic)} - e_{\ta(\generic)}$
    for all $\generic \in \L_1 = \L $;

  \item for all $n\in \mathbb N$ and all $F\in C^n(\L,E)$, 
    $\partial_n F$ is the element of $C^{n+1}(\L,E)$ defined by
    \begin{align*}
      \partial_n F (\generic_0, \dots, \generic_n) 
      ={} & \generic_0 \cdot F( \generic_1, \dots, \generic_{n-1} ) \\
       & - \sum_{i=0}^{n-1} (-1)^i 
           F(\generic_0, \dots, \generic_{i} \cdot \generic_{i+1}, \dots, \generic_{n}) \\
       & - (-1)^n F(\generic_0, \dots, \generic_{n-1}) 
    \end{align*} 
    for all $(\generic_0, \dots, \generic_n) \in \L_{n+1}$.
\end{enumerate}

Since we will mostly be interested in the first cohomology space $H^1(\L,E)$, it is worth describing it 
more explicitly. 
On the one hand, $1$-cocycles are functions $F$ from $\L$ to $E$
such that $F(\generic) \in E_{\ta (\generic)}$ for all $\generic \in \L$ and satisfying the 
\emph{cocycle identity}
\begin{equation}
\label{eq:1cocycle}
  F( \generic_1 \cdot \generic_2) 
  = \generic_1 \cdot F(\generic_2) + F(\generic_1)
\end{equation} 
for all $\generic_1,\generic_2$ in $\L$.
On the other hand, $1$-coboundaries are $E$-valued functions on $\L$ of the form
\begin{equation}
\label{eq:1cobord}
  F(\generic) = \generic \cdot e_{\so(\generic)} - e_{\ta(\generic)} 
\end{equation} 
 for some section $e \in \Gamma(E)$.

\subsubsection*{Subgroupoids}

A \emph{Lie subgroupoid} of a Lie groupoid $\L\toto M$ is a Lie groupoid $\A\toto N$ together with a Lie groupoid morphism $(\varphi,\varphi_0)$ from $\A$ to $\L$ such that both $\varphi$ and $\varphi_0$ are injective immersions.
A Lie subgroupoid $\A$ is said to be \emph{wide} if its unit manifold is $M$ and $\varphi_0=\id$.
A wide Lie subgroupoid is said to be \emph{closed} if the inclusion $\varphi:\A\to\L$ is a closed embedding.

\begin{defn}
A \emph{Lie groupoid pair} is a pair $(\L,\A)$ with $\L$ a Lie groupoid and $\A$ a closed wide Lie subgroupoid of $\L$.
\end{defn}

For $\A$ a closed wide subgroupoid of $\L$, the quotient space $\L/\A $ is a (Haussdorf) manifold that fibers over $M$ through a surjective submersion $\underline{\ta}$ (see \cite[Section 3]{moerdijk_integrability_2006}) defined as the unique map making the following diagram commutative: 
\begin{equation}
\label{eq:t_bar}
  \vcenter{\xymatrix{
    &
    \L \ar[d]^{\pi} \ar[dl]_{\ta}
    \\
    M
    &
    \L/\A \ar@{.>}[l]^{\underline{\ta}}
  }} .
\end{equation}
Projections to quotients by a group(oid) action will generally be denoted by the letter $\pi$. 
The source will be added as a subscript, as in $\pi_P:P\to P/\U$, when a risk of confusion exists.

The quotient $\L/\A$ is a \emph{homogeneous space} of $\L$ in the following sense \cite{liu_dirac_1998}: it is a smooth manifold $X$ with a map to $M$ such that there exists a smooth section $\sigma:M\to X$ with $\L\cdot\sigma(M)=X$. Given such data, the \emph{stabilizer} of $\sigma$ is the closed subgroupoid $\A$ of $\L$ that sends $\sigma(M)$ to itself, and this yields an equivariant diffeomorphism $X\cong \L/\A$.

\subsubsection*{Pullbacks}
Let $\L\toto M$ be a groupoid and $\phi:N\to M$ a map.
Consider the \emph{anchor} $\rho:\L\to M\times M:\generic\mapsto(\ta(\generic),\so(\generic))$ of $\L$.
The \emph{pullback} $\phi^\pb\L$ of $\L$ by $\phi$ is, as a set, the pullback of the anchor by $\phi\times\phi$:
\begin{equation}
\label{eq:groupoid-pullback}
  \vcenter{\xymatrix{
    \phi^\pb\L \ar[r] \ar[d] & \L \ar[d]^\rho \\
    N\times N \ar[r]_{\phi\times\phi} & M\times M
  }} .
\end{equation}
Explicitly, we set $\phi^\pb\L=N\times_{M}^{\phi,\ta} \L \times_{M}^{\so,\phi} N$.
The structure maps $\so(y,\generic,x)=x$, $\ta(y,\generic,x)=y$ and multiplication
\[
  (z,\generic,y) \cdot (y,\generic',x) = (z,\generic\cdot\generic',x)
\]
on $\phi^\pb\L$ make \eqref{eq:groupoid-pullback} a commutative diagram of groupoid morphisms, with left arrow the anchor of $\phi^\pb\L$ and top arrow the projection on the second component.

When $\L$ is a Lie groupoid, we require $\phi$ to be a smooth map such that
\begin{equation}
\label{eq:target-circ-pr1}
  \ta\circ\pr_1 : \L\times_M^{\so,\phi} N \to M
\end{equation}
is a surjective submersion.
In that case, $\phi^\pb\L$ becomes a Lie groupoid such that \eqref{eq:groupoid-pullback} is a commutative diagram of Lie groupoid morphisms with the appropriate universal property.

\subsubsection*{Tangent groupoid}

The tangent bundle $T\L$ of a Lie groupoid $\L$ over $M$ canonically becomes a Lie groupoid over $TM$ by applying the tangent functor to all the structure maps (see  \cite{mackenzie_general_2005}, or \cite{courant_tangent_1994} for its infinitesimal counterpart).
In what follows, we will denote the groupoid multiplication in $T\L$ by
\begin{equation}
\label{eq:def-composition-in-TL}
  u \bullet u'
  = T_{(\generic,\generic')}\m(u,u') ,
\end{equation}
for all $u\in T_\generic\L$, $u'\in T_{\generic'}\L$, where $m$ denotes the multiplication in $\L$.
We will denote by $u^{-1}$ the inverse $T\inv(u)$, and by $0_\generic$ the zero vector at $\generic$.


\subsection{Lie algebroids}
\label{ssec:lie-algebroids}

A \emph{Lie algebroid} is a vector bundle $\ll$ over a smooth manifold $M$ together with a Lie bracket $[\cdot,\cdot]$ on the space $\Gamma(\ll)$ of global sections of $L$ and a bundle map $\rho:\ll\to TM$ called the \emph{anchor map}, related by the following Leibniz rule:
\begin{equation}
  [l_1,fl_2] = f[l_1,l_2] + \rho(l_1)(f) l_2
\end{equation}
for all $l_1,l_2\in\Gamma(\ll)$ and $f\in C^\infty(M)$.
All vector bundles will be real in this paper, and we shall use the letter $\rho$, with the Lie algebroid as subscript if necessary, for the anchor map and $[\cdot,\cdot]$ for the Lie bracket of all the Lie algebroids that we consider.
We refer to \cite{mackenzie_general_2005} for the general theory of Lie algebroids.

A \emph{base-preserving Lie algebroid morphism} from $\ll\to M$ to $\uu\to M$ is a bundle map $\phi:\ll\to\uu$ covering the identity of $M$, such that $\rho_\ll=\rho_\uu\circ\phi$ and $[\phi(l_1),\phi(l_2)]-\phi[l_1,l_2]=0$ for all $l_1,l_2\in\Gamma(\ll)$ (see below for more on the first condition).
A \emph{(wide) Lie subalgebroid} of a Lie algebroid $\ll\to M$ is an injective base-preserving morphism $\aa\to\ll$.

We recall \cite{moerdijk_introduction_2003,huebschmann_lie-rinehart_2004} that the universal enveloping algebra of a Lie algebroid $\ll$ is constructed by taking the quotient of the augmentation ideal of the universal enveloping algebra of the semi-direct product Lie algebra $\Gamma(\ll) \ltimes_\rho C^\infty(M)$ by the relations $ f \cdot l = fl $,  and $ f \cdot g =  fg $ for all $ f,g \in C^\infty(M)$ and $ l\in \Gamma (\ll)$. It is a coalgebra, and, for $\aa \subset \ll$ a Lie subalgebroid, $ U(\ll)/U(\ll) \cdot \Gamma(\aa)$
inherits a coalgebra structure \cite{calaque_pbw_2014}.

Every Lie groupoid $\L$ with unit manifold $M$ admits a Lie algebroid 
$\ll \to M$.
In the present article, for all $x \in M$, $\ll_x$ shall be defined as the kernel of 
$T_x \so : T_x \L \to T_x M$
(i.e.\ the tangent space at $x$ to the $\so$-fiber over $x$), while the anchor map $\rho:\ll\to TM$ at $x$ is the restriction to $\ll_x$ of 
$T_x \ta : T_x \L \to T_x M$.
To every section $l \in \Gamma(\ll)$ there correspond two vector fields on $\L$, namely the \emph{left-invariant vector field} $\lvf{l}$, and the \emph{right-invariant vector field} $\rvf{l}$.
The values of these vector fields at a generic element $\generic \in\L$ with source $x$ and target $y$ are given by
\begin{equation}
\label{eq:left-and-right-invariant-vector-fields} 
  \lvf{l}\big|_\generic 
  = 0_\generic \bullet (-l_x^{-1})
  \qquad \textrm{and} \qquad
  \rvf{l}\big|_\generic = l_y\bullet 0_\generic .
\end{equation}
The sign convention is chosen so that, evaluated at a unit $x\in M$, the left- and right-invariant vector fields $\lvf{l}|_x=-T_x\inv(l_x)$ and $\rvf{l}|_x=l_x$ project to the same element in the normal bundle $T_x\L/T_xM$.
The commutator of two right-invariant vector fields is again right-invariant.
Hence the sections of $L$ acquire a Lie bracket by transporting the commutator of (right-invariant) vector fields through the isomorphism $l\mapsto \rvf l$, completing the description of the Lie algebroid of a Lie groupoid.

\begin{ex}
\label{ex:derivative-endomorphisms}
Let us describe the Lie algebroid of the frame groupoid $\fb(E)$ of a vector bundle $E$ (Example \ref{ex:frameGroupoid}).
The $\so$-fiber at a point $x\in M$ is the manifold of all linear isomorphisms from $E_x$ to $E_y$ for all $y\in M$.
Its tangent space at $\Id_{E_x}$ is the vector space of linear maps $X:E_x\to TE|_{E_x}$ which are sections of the canonical projection $TE|_{E_x}\to E_x$.
In particular, for each such $X$ there exists an element $X_M\in T_xM$ making the diagram
\[
  \vcenter{\xymatrix{
    E_x \ar[r]^X \ar[d]_q & TE|_{E_x} \ar[d]^{Tq} \\
    \{x\} \ar[r] & \{X_M\}
  }}
\]
commute.
The collection of these tangent spaces forms a vector bundle $\Tlin E$ whose sections are the \emph{linear vector fields} $\Xlin(E)$, i.e.\ the bundle maps $E\to TE$ that are sections of the canonical projection bundle map $TE\to E$ over $TM\to M$.
These vector fields are closed under the Lie bracket of vector fields and, together with the anchor map $X\mapsto X_M$, this gives $\Tlin E$ the structure of a Lie algebroid.

A common description of $\Tlin E=\Lie(\fb(E))$ is as the Lie algebroid $\D(E)$ whose sections are the \emph{derivative endomorphisms} of $E$, i.e.\ the $\fR$-linear endomorphisms $D$ of $\Gamma(E)$ such that there exists a vector field $D_M$ on $M$ with
\begin{equation}
\label{eq:derivative-endomorphism}
  D(fe) = D_M(f)e + f D(e)
\end{equation}
for all $e\in\Gamma(E)$ and $f\in C^\infty(M)$.
The explicit correspondence between $\D(E)$ and $\Tlin E$ is as follows (see \cite[Section 3.4]{mackenzie_general_2005} for more detail).

Any $D\in \Gamma(\D(E))$ yields a dual $D^*\in \Gamma(\D(E^*))$ such that
\[ \left< D^*(\epsilon),e\right> = D_M (\left<\epsilon,e\right>) - \left< \epsilon, D(e)\right> \]
for all $\epsilon\in\Gamma(E^*)$ and $e\in\Gamma(E)$.
This $D^*$ in turn induces a linear vector field $X\in\Xlin(E)$ defined by
\[
  X(l_\epsilon) = l_{D^*(\epsilon)} ,
  \qquad
  X(q^*f) = q^*(D_M(f))
\]
for all $\epsilon\in\Gamma(E^*)$ and $f\in C^\infty(M)$.
Here, $l_\epsilon$ is the fibrewise-linear function on $E$ corresponding to the section $\epsilon$.
We used the fact that linear vector fields are determined by their action on linear functions and on pullbacks of functions on the base, and that they preserve the latter two subspaces of functions.
The association $D\mapsto X$ is $C^\infty(M)$-linear and induces a Lie algebroid isomorphism 
\begin{align}
\label{eq:isomorphism-derivative-endomorphisms-derivations}
  \TD &: \D(E) \to \Tlin(E) .
\end{align}
\end{ex}


\subsubsection*{Pullbacks}

The pullback $\phi^\pb\uu$ of a Lie algebroid $\uu\to N$ by a smooth map $\phi:M\to N$ is defined in a similar way to the Lie groupoid case, see e.g.\ \cite{higgins_algebraic_1990}.
As a set, it is the pullback of its anchor by $T\phi$:
\[
  \vcenter{\xymatrix{
    \phi^\pb\uu \ar[r] \ar[d] & \uu \ar[d]^\rho \\
    TM \ar[r]_{T\phi} & TN
  }} .    
\]
To see when it is a vector bundle, it is best to replace the right-hand column in the above diagram by its pullback $\phi^*\uu\to\phi^*TN$ by $\phi$, to get a diagram of vector bundles over $M$.
Then $\phi^\pb\uu$ is a vector subbundle of $TM\oplus\phi^*\uu$ if and only if the bundle map $\psi:TM\oplus\phi^*\uu\to\phi^*TN:(X,u)\mapsto T\phi(X)-\rho(u)$ has constant rank.
For simplicity, we require $\psi$ to have maximal rank, i.e.\ we require $T\phi$ and $\rho$ to be transverse, in which case the kernel of $\psi$,
\[ \phi^\pb\uu = TM\times_{\phi^*TN}\phi^*\uu , \] 
has rank
\[
  \rk(\phi^\pb\uu) = \rk(\uu) + \dim M - \dim N .
\]
The vector bundle $\phi^ \pb\uu$ then becomes a Lie algebroid with anchor map given by the first projection and Lie bracket defined by
\begin{multline}
\label{eq:bracket-pullback}
  [(X,\sum_i f_i\phi^*u_i),(X',\sum_j f'_j\phi^*u'_j)] 
  \\
  = ( [X,X'], \sum_{i,j} \left( f_if'_j\phi^*[u_i,u'_j] + f X(f')\phi^*u' - f' X'(f)\phi^*u \right) ) 
\end{multline}
for all $X,X'\in\Gamma(TM)$, $f_i,f'_j\in C^\infty(M)$ and $u_i,u'_j\in\Gamma(\uu)$ such that $T\phi(X)=\sum_i f_i \phi^*(\rho(u_i))$ and $T\phi(X')=\sum_j f'_j \phi^*(\rho(u'_j))$.

When $\U\toto N$ is a Lie groupoid and $\phi:M\to N$ a smooth map such that \eqref{eq:target-circ-pr1} is a surjective submersion, we have that $T\phi$ is transverse to $\rho$ and there is a natural isomorphism $\phi^\pb(\Lie(\U)) \cong \Lie(\phi^\pb\U)$.


\subsubsection*{Morphisms not preserving the base}

Consider two Lie algebroids $L\to M$ and $U\to N$ and a bundle map $\Phi:L\to U$ over $\phi:M\to N$.
The pullback $\phi^\pb U$ is not a vector bundle for general $\phi$, but the Lie bracket \eqref{eq:bracket-pullback} on 
\[  \phi^\pb \Gamma(U) := \Gamma(TM)\times_{\Gamma(\phi^*TN)}\Gamma(\phi^*U) \]
always makes sense.
Moreover, $\Phi$ always induces a base-preserving map $\Phi^*:L\to\phi^*U$ and thus a map of sections $\Phi^\pb=(\rho,\Phi^*):\Gamma(L)\to\phi^\pb\Gamma(U)$.
Hence, $\Phi$ said to be a \emph{Lie algebroid morphism} if
\begin{enumerate}
  \item it is \emph{anchored}, i.e.\ it commutes with the anchors:
		$\rho_U \circ \Phi = T\phi \circ \rho_L$,
	\item the induced map $\Gamma(L) \to \phi^\pb \Gamma(U)$ is a Lie algebra morphism.
\end{enumerate}


\subsubsection*{Connections}

Let $\ll\to M$ be a Lie algebroid, and $E\to M$ a vector bundle.
An \emph{$\ll$-connection on $E$} is an $\fR$-bilinear assignment
\[
  \Gamma (\ll) \times \Gamma (E) \to \Gamma(E)
  : (l,e) \mapsto \nabla_l e
\]
which satisfies
\[
  \nabla_{fl} e = f \nabla_l e 
  \qquad \text{and} \qquad
  \nabla_l  (fe) = f \nabla_l e + \rho(l)(f) e 
\]
for all $f\in C^\infty(M)$, $l \in \Gamma(\ll)$, and $e \in \Gamma(E)$. 
An $\ll$-connection $\nabla$ on $E$ is said to be \emph{flat} when
\[
  \nabla_{l_1} \nabla_{l_2} e - \nabla_{l_1} \nabla_{l_2}e  = \nabla_{[l_1,l_2]} e
\]
for all $l_1,l_2 \in \Gamma (\ll)$, $e \in \Gamma (E)$.
A vector bundle $E \to M$ equipped with a flat $\ll$-connection is said to be an \emph{$\ll$-module}. 

For each $l\in\Gamma(\ll)$, the map $\nabla_l:\Gamma(E)\to\Gamma(E)$ is a derivative endomorphism as in Example~\ref{ex:derivative-endomorphisms} (with associated vector field $\rho(l)$), and the assignment $l\mapsto \nabla_l$ is $C^\infty(M)$-linear.
Hence, a connection $\nabla$ can be recast (see \cite{kosmann-schwarzbach_differential_2002}) as an anchored map $\ll\to \D(E)$.
A connection is flat if and only if the corresponding anchored map is a Lie algebroid morphism.

The above is the ``covariant derivative'' picture.
The corresponding horizontal lift is obtained by composing with the Lie algebroid isomorphism \eqref{eq:isomorphism-derivative-endomorphisms-derivations}, to get an anchored map $\ll\to \Tlin(E)$.
This horizontal lift point of view of anchored maps $\ll\to \Tlin(E)$, seen as maps $q_L^*E\to TE$, was extensively studied in \cite{fernandes_lie_2002}.

Replacing $\mathcal D(E)$ by any Lie algebroid $U$, not necessarily over the same base, we may consider anchored maps from $\ll$ to $\uu$ as generalized connections.
Anchored maps have been used by many authors, and their use in Lie algebroid theory goes back at least to~\cite{balcerzak_primary_2001}. 
The \emph{curvature} of an anchored map $\nabla$ from $\ll$ to $\uu$ is the bundle map $R^\nabla:\ll\wedge \ll\to r^*\uu$ defined on sections by
\[  R^\nabla(l,l')=[\nabla^\pb(l),\nabla^\pb(l')]-\nabla^\pb([l,l']) . \]
When $\uu$ is regular, the curvature actually lands in $r^*\uu_0$, where $\uu_0$ is the isotropy subalgebroid of $\uu$, i.e.\ the kernel of its anchor map.

A \emph{connection on a vector bundle $E\to M$} is a $TM$-connection on $E$, i.e.\ an anchored map $TM\to \D(E)$.
If the manifold $M$ is fibered over another manifold $N$ through a surjective submersion $f:M\to N$, a \emph{fibered connection on $E$} is by definition a $T^fM$-connection on $E$, where $T^fM=\ker Tf\subset TM$.
It is thus a smooth family of connections on the vector bundles $i_x^*E$, for $x\in N$, where $i_x:f^{-1}(x)\to M$ is the inclusion of the fiber over $x$. 

An \emph{affine connection} on a manifold $M$ is a connection on the vector bundle $TM$, i.e.\ an anchored map $TM\to \D(TM)$.
If the manifold $M$ is fibered over $N$ through a surjective submersion $f:M\to N$, a \emph{fibrewise affine connection} is a $T^fM$-connection on $T^fM$.
It is thus a smooth family of affine connections on the fibers of $f$.


\subsection{Bisections}
\label{ssec:bisections}

An open submanifold $\bisection$ of $\L$ to which the restrictions of $\so$ and of $\ta$ are both diffeomorphisms onto their respective images is called a \emph{local bisection} of $\L$. The open subsets
$\ta(\bisection)$ and $\so(\bisection)$ are called the \emph{target} and \emph{source} of $\bisection$ respectively.
When $\ta(\bisection) = \so(\bisection) = M$, we speak of a \emph{global bisection}. 
A local bisection is said to be \emph{through an element $\generic \in \A$} when $\generic \in \bisection$, and to be  \emph{through an element $u \in T_\generic \A$} when $\generic \in \bisection$ and $u \in T_{\generic} \bisection$. 

It is often convenient to see a local bisection $\bisection$ as a section of the target map, that we then denote by $\bisection_{\ta} : \ta(\bisection) \to \L$, or as a section of the source map, that we then denote by $\bisection_{\so} : \so(\bisection) \to \L$. 
A global bisection $\bisection$ induces a diffeomorphism of $M$ denoted by $\underline{\bisection}$ and defined by
\begin{equation}\label{eq:bisection_induces_on_base}
  \underline{\bisection} := \ta \circ \bisection_{\so} .
\end{equation}
For $\bisection$ a local bisection, \eqref{eq:bisection_induces_on_base} still makes sense as a diffeomorphism from the source to the target of $\bisection$.

The composition of two local bisections $\bisection', \bisection \in \L$ is defined by
\begin{equation}\label{def:star_subsets} 
  \bisection' \star \bisection 
  := \left\{ \generic' \cdot \generic \mid \generic' \in \bisection',\generic \in \bisection ,
  \so(\generic') = \ta (\generic) \right\}. 
\end{equation}
Global bisections form a group under $ \star$, and the product of two local bisections is a local bisection.
In the whole text, we simply write \emph{bisection} when referring to a local bisection, since all constructions considered in this paper are local by nature. Local bisections only form a pseudogroup, that we denote by $\Bis(\L)$. We shall often speak of inverses and products without mentioning that we only have a pseudogroup.

For $ E$ an $\L$-module, $e \in \Gamma(E)$ and $\bisection$ a local bisection, we again use the notation $\star$ and denote by $\bisection \star e$ the local section of $E$ of $\bisection$ given by
\[  \left. \bisection \star e \right|_x  := \bisection_\ta (x) \cdot e_{\underline{\bisection}^{-1} (x)}  \]
for all $x\in M$ where the right-hand side is defined. 
For any pair $\bisection',\bisection$ of bisections, the relation
$(\bisection' \star \bisection) \star e =\bisection' \star (\bisection \star e)$ holds, 
allowing us to erase the parentheses and to write simply $\bisection' \star \bisection \star e$ for such expressions.

The Lie algebra $ \Gamma(\ll)$ is ``the Lie algebra of the group of bisections''.
For example,
for any smooth%
\footnote{A $1$-parameter family of bisections is said to be smooth if the map $(x,t) \mapsto \bisection(\param)_{\so}(x)$ is smooth.} 
$1$-parameter family $\bisection(\param)$ of 
bisections, defined for $\param$ in a neighborhood of $0 \in {\mathbb R}$, and such that $\bisection (0)$ is the unit manifold of $\L$, the map defined, for all $x \in M$, by
\begin{equation}\label{derivation}
  x \mapsto \left. \frac{\diff}{\diff \param} \bisection(\param)_\so(x) \right|_{\param =0}
\end{equation}
takes values in the kernel of the source map in $T_x \L $, i.e.\ is a section of the Lie algebroid $\ll$. 

For any $\L$-module $E$, an $\ll$-module structure (i.e.\ a flat $\ll$-connection) on $E$ is induced by
\begin{equation}\label{eq:relation-algebroid-groupoid}
   \nabla_{ l }  e  = \left. \frac{\diff}{\diff \param} \bisection (\param)^{-1} \star e \right|_{\param =0}   ,
\end{equation}
for any smooth $1$-parameter family $\bisection(\param)$ of local bisections with 
$ l = \left. \frac{\diff}{\diff \param} \bisection(\param)_\so \right|_{\param =0}.$

The group of bisections of a Lie groupoid $\L$ naturally acts by the adjoint action on the Lie algebra of sections of $\ll$,
\begin{equation}\label{eq:def_adjoint_bisections} 
  \left. \Ad_\bisection l \right|_x
  := \left. \frac{\diff}{\diff \param} \left( \bisection 
            \star \bisection(\param)
            \star \bisection^{-1} \right)_\so (x) \right|_{\param =0},
\end{equation}
where $\bisection \in \Bis(\L)$, $l \in\Gamma(\ll)$, and $\bisection(\param)$ is as in the previous paragraph. 

\begin{ex}
\label{ex:bisections-of-frame-bundle-are-automorphisms}
  The bisections of the frame groupoid of a vector bundle $E$ are the \emph{automorphisms} of $E$, i.e.\ the diffeomorphisms $E\to E$ that send fibers to fibers linearly.
  The Lie algebra $\Der(E)$ of this infinite dimensional Lie group is composed of the vector fields on $E$ whose flows are by automorphisms of $E$, i.e.\ the linear vector fields $\Xlin(E)=\Gamma(\Tlin(E))$ of Example \ref{ex:derivative-endomorphisms}.
  
  The action of $\Bis(\fb(E))$ on $\Xlin(E)$ is given by
  \begin{align}
  \label{eq:action-bisections-frame-bundle-on-linear-vector-fields}
    \Ad_\phi X &= T\phi \circ X \circ \phi^{-1}
  \end{align}
  for all $\phi\in\Bis(\fb(E))$ and $X\in\Xlin(E)$.
  When the base manifold is a point, this boils down to the action of $GL(V)$ on $\mathfrak{gl}(V)$ by matrix conjugation, $\Ad_gX=gXg^{-1}$, for a vector space $V$.
\end{ex}

Since the group of bisections also acts on the smooth functions on $M$ by $\Sigma \cdot f = \underline{\Sigma}^* f$ for all functions $f$ and all bisections $\Sigma$, this action extends to $U(\ll) $ by coalgebra morphisms.
Indeed, for all $\Sigma \in \Bis(\L)$ and $u = l_1 \cdot \dots \cdot l_k\in U(\ll)$, with $l_1, \dots,l_k  $ sections of $\ll$, we have an action
   \[ \Sigma \cdot u = \Ad_\Sigma l_1  \cdot \dots \cdot   \Ad_\Sigma l_k  \]
Moreover, for  $(\L,\A)$ a Lie groupoid pair, the group of bisections of $\A$, seen as a subgroup of the group of bisections of $\L$, acts on $ U(\ll)/U(\ll) \cdot \Gamma(\aa)$ by coalgebra morphisms.
The following proposition is easy to prove.

\begin{prop}
\label{bisection_action}
Let $(\L,\A)$ be a Lie groupoid pair. 
Then the (pseudo-)group $\Bis(\A)$ of bisections of $\A$ acts on $U(\ll)/U(\ll) \cdot \Gamma(\aa)$. 
The infinitesimal action of this action is simply left-multiplication by 
sections of $\aa$,
\[  a \cdot \overline u = \overline{ a \cdot u}  \]
for all $a \in \Gamma (\aa)$ and all $u \in U(\ll)$.
Here, $ u \mapsto \overline{u} $ is the projection from $ U(\ll)$ to $U(\ll)/U(\ll) \cdot \Gamma(\aa) $.
\end{prop}


\subsubsection*{Exponential map}

Given a section $l\in\Gamma(\ll)$, for each $\param \in {\mathbb R}$ and $\generic \in \L$
for which is it defined, we denote by $\Phi_\param (\generic)$ the flow of the right-invariant vector field $\rvf{l}$ starting from $\generic$ and evaluated at time $\param$.
Every point $m \in M$ admits a neighborhood ${\mathcal U}$ such that the submanifold $\Phi_{\param} (\unit(\mathcal{U}))$ is a local bisection of $\L$ for all $\param$ small enough.
We denote by $\param \to \expon{\param l}$ this map, about which we recall three important properties.

\begin{prop}[{\cite{mackenzie_general_2005}}]
\label{prop:bisections}
  Let $\L$ be a Lie groupoid with Lie algebroid $\ll$.
  \begin{enumerate}
    \item For all $l \in \Gamma (\ll)$, and all $\param_1,\param_2$ such that $\expon{\param_1 l}, \expon{\param_2 l}$, and $\expon{(\param_1 + \param_2 )l}$ exist, the identity 
			\[ \expon{\param_1 l} \star \expon{\param_2 l} = \expon{(\param_1 + \param_2 )l} \]
			holds.

    \item For all $l \in \Gamma (\ll)$ and all local bisections $\bisection$,
      \begin{equation}
      \label{expon-and-bisections}
        \expon{ \Ad_\bisection l } = \bisection \star \expon{l} \star \bisection^{-1}.
      \end{equation}
		
		\item For all $l,l'\in\Gamma(\ll)$,
			\begin{equation}
			\label{expon-lie-bracket}
				[l,l'] = \left. \frac{d}{dt} \Ad_{\expon{-tl}}l' \right|_{t=0} .
			\end{equation}
  \end{enumerate}
\end{prop}


\section{Equivariant principal bundles}
\label{sec:equivariant-principal-bundles}


\subsection{Generalized morphisms}
\label{ssec:lie-groupoid-generalized-morphisms}

In this section, we collect some facts about Lie groupoid generalized morphism (see \cite{blohmann_stacky_2008} and references 15, 18, 19, 25 and 26 therein for more details).
We then go on by introducing a Lie algebroid morphism $\Add$ related to the adjoint action that will be useful in the next sections. 

We adopt the point of view of bibundles: a generalized morphism from $\L$ to $\U$ is a manifold with two commuting actions of $\L$ and $\U$ such that the $\U$-action is free and proper (i.e.\ principal) with orbit space the base of $\L$.

\begin{defn}
\label{def:generalized-morphism}
  Let $\L\toto M$ and $\U\toto N$ be two Lie groupoids.
  A \emph{generalized morphism from $\L$ to $\U$} is a smooth manifold $P$ with a left $\L$-action on a map $l:P\to M$ and a right $\U$-action on a map $r:P\to N$ such that:
  \begin{enumerate}[label=(\arabic*)]
    \item \label{def:generalized-morphism-1}
      the left and right actions commute,
    \item \label{def:generalized-morphism-2}
      $l$ is a surjective submersion,
    \item \label{def:generalized-morphism-3}
      the map
      \begin{equation}
      \label{eq:gen-mor-inverse-of-right-action-on-P}
        P \times_N \U \to P\times_M P : (p,\genericU) \mapsto (p,p\cdot \genericU)
      \end{equation}
      is a diffeomorphism.
  \end{enumerate}
  The manifold $P$ is called the manifold of arrows, and the maps $l$ and $r$ are called the left and right moment maps, respectively.
\end{defn}

Note that both sides of the map \eqref{eq:gen-mor-inverse-of-right-action-on-P} have a structure of Lie groupoid over $P$: the action groupoid for the right $\U$-action on $P$ on the left-hand side (with source $\so(p,\genericU)=p\cdot\genericU$ and target $\ta(p,\genericU)=p$) and the Lie groupoid induced by the submersion $l:P\to M$ on the right-hand side (with source $\so(p,p')=p'$ and target $\ta(p,p')=p$).
With these structures, the map is actually a Lie groupoid morphism.

The second component
\begin{equation}
\label{eq:gen-mor-division-map}
  \di_P : P\times_M P \to \U
\end{equation}
of the inverse of \eqref{eq:gen-mor-inverse-of-right-action-on-P} is called the \emph{division map} of $P$.
Since \eqref{eq:gen-mor-inverse-of-right-action-on-P} and $\pr_2:P\rtimes \U\to \U$ are Lie groupoid morphisms, the division map is also a Lie groupoid morphism.
Moreover, it is \emph{$\L$-invariant} and \emph{$\U$-equivariant}, in the sense that
\begin{align}
\label{eq:gen-mor-invariance-equivariance-division-map-1}
  \di_P(\generic \cdot p,\generic\cdot p') &= \di_P(p,p') \\
\label{eq:gen-mor-invariance-equivariance-division-map-2}
  \di_P(p\cdot\genericU,p'\cdot\genericU') &= \genericU^{-1}\cdot\di_P(p,p')\cdot\genericU'
\end{align}
for all $p,p'\in P$ in the same $l$-fiber, $\generic\in\L$ such that $\so(\generic)=l(p)$, and $\genericU,\genericU'\in\U$ such that $r(p)=\ta(\genericU)$ and $r(p')=\ta(\genericU')$.

Generalized morphisms can be composed: if $P$ is a generalized morphism from $\L$ to $\U$, and $Q$ is a generalized morphism from $\U$ to $\V$, then $P\times_N Q$ is a smooth manifold thanks to condition~\ref{def:generalized-morphism-2} (here, $N$ is the manifold of units of $\U$).
Moreover, by condition~\ref{def:generalized-morphism-3} it has a proper and free (right) $\U$-action $((p,q),\genericU)\mapsto (p\cdot\genericU,\genericU^{-1}\cdot q)$, so that the quotient $\frac{P\times_{N}Q}{\U}$ is a smooth manifold as well.
The latter still has a free and fiber-transitive $\V$-action and is in fact a generalized morphism from $\L$ to $\V$, which we call the \emph{composition} $P\circ Q$ of $P$ and $Q$.

There is a natural notion of equivalence between generalized morphisms with same source and target Lie groupoids.
An equivalence between two generalized morphisms $P$ and $P'$ from $\L$ to $\U$ is a smooth map $\phi:P\to P'$ which is \emph{bi-equivariant}: it commutes with the left and right moment maps and with the left and right actions.
Since the left moment maps are surjective submersions and the right actions are transitive on the $l$-fibers, such bi-equivariant maps are necessarily diffeomorphisms.

\begin{ex}[Units]
\label{ex:identity-generalized-morphism}
  Any Lie groupoid $\L\toto M$ defines a generalized morphism, the \emph{unit generalized morphism} $\Id_\L$ from $\L$ to $\L$.
  Its manifold of arrows is $\L$, with the left and right multiplications as left and right actions, respectively.
  It is a unit for the composition of generalized morphisms, in the sense that if $P$ is a generalized morphism from $\L$ to $\U$, then there are natural bi-equivariant diffeomorphisms ${\Id_\L} \circ P \simeq P$ and $P\circ {\Id_\U} \simeq P$ induced by the left and right actions, respectively.
\end{ex}

Similarly, the composition of generalized morphisms is associative up to coherent bi-equivariant diffeomorphisms.
More precisely, the above structures fit together to form a bicategory (see \cite[Proposition 2.12]{blohmann_stacky_2008} for example).

\begin{prop}
\label{prop:generalized-morphisms-bicategory}
  The Lie groupoids with generalized morphisms as 1-morphisms and bi-equivariant maps as 2-morphisms form a bicategory.
\end{prop}

\begin{ex}[Morphisms]
\label{ex:bundlization-generalized-morphism}
  Any Lie groupoid morphism $\varphi$ from $\L\toto M$ to $\U\toto N$ defines a generalized morphism from $\L$ to $\U$ which, following \cite{blohmann_stacky_2008}, will be called its \emph{bundlization}.
  Its manifold of arrows is $P_\varphi=M\times_N^{\varphi_0,\ta}\U$, the moment maps are $l(x,\genericU)=x$ and $r(x,\genericU)=\so(\genericU)$, the left action is $\generic\cdot(x,\genericU)=(\generic\cdot x,\varphi(\generic)\cdot\genericU)$, and the right action is $(x,\genericU)\cdot\genericU'=(x,\genericU\cdot\genericU')$.
\end{ex}

Any manifold $P$ with a left $\L$-action on a map $l$ and a right $\U$-action on a map $r$, has an \emph{opposite} manifold $P^\op$ with a left $\U$-action and a right $\L$-action.
It is defined by $P^\op=P$, $l^\op=r$, $r^\op=l$, with left action $\genericU\cdot_\op p = p\cdot \genericU^{-1}$ and right action $p\cdot_\op \generic = \generic^{-1} \cdot p$.
If $P$ and $P^\op$ are both generalized morphisms, $P$ is then \emph{weakly invertible} in the sense that 
\begin{align}
\label{eq:P-Pop}
  P \circ P^\op \to \Id_\L    & : [(p,p')]\mapsto \di_{P^\op}(p,p') \\
\label{eq:Pop-P}
  {P^\op} \circ P \to \Id_\U  & : [(p,p')]\mapsto \di_{P}(p,p')
\end{align}
are bi-equivariant diffeomorphisms.
Such weakly invertible generalized morphisms are called \emph{Morita morphisms}.
Two Lie groupoids are called \emph{Morita equivalent} if there exists a Morita morphism between them.

Some Lie groupoid morphisms, although not invertible themselves, have a weakly invertible bundlization, as shows the following basic example.
\begin{ex}[Pullbacks]
\label{ex:pullback-generalized-morphism}
  Consider the pullback $\varphi_0^\pb\L$ of a Lie groupoid $\L\toto M$ by a surjective submersion $\varphi_0:N\to M$, and consider the corresponding morphism $\varphi:\varphi_0^\pb\L\to\L:(x,\generic,x')\mapsto\generic$.
  Its bundlization is $P_\varphi=N\times_M\L$ with the actions described in Example~\ref{ex:bundlization-generalized-morphism}.
  Now it is easy to see that $(P_\varphi)^\op$ is also a generalized morphism, with division map $D_{(P_\varphi)^\op}((x,\generic),(x',\generic')) = (x,\generic\cdot{\generic'}^{-1},x')$.
  
  More generally, it is sufficient that $\varphi_0$ be a smooth map such that $\so\circ\pr_2:N\times_M\L\to M$ is a surjective submersion for the pullback groupoid to be a Lie groupoid and for the associated generalized morphism to be a Morita morphism from $\varphi_0^\pb\L$ to $\L$.
\end{ex}

A generalized morphism $P$ from $\L$ to $\U$ with moment maps $l$ and $r$ induces a morphism $\Phi_P$ from the pullback $l^\pb\L$ to $\U$ over $r$.
This \emph{induced morphism} is defined by 
\begin{align}
\label{eq:gen-mor-induced-morphism}
  \Phi_P(p,\generic,p')=\di_P(p,\generic\cdot p')
\end{align}
where $\di_P$ is the division map of $P$, i.e.\ $\Phi_P(p,\generic,p')$ is the unique $\genericU\in\U$ such that $\generic\cdot p' = p\cdot \genericU$.
Hence, a generalized morphism may be pictured as
\begin{equation}
\label{eq:diagram-generalized-morphism-pullback-and-morphism}
  \vcenter{\xymatrix{
      & l^\pb\L \ar[dl]_{\pr_2} \ar[dr]^{\Phi_P}  \ar@<.5ex>[d] \ar@<-.5ex>[d]
      & \\
    \L \ar@<.5ex>[d] \ar@<-.5ex>[d] 
      & P \ar@{>>}[dl]_l \ar[dr]^r
      & \U \ar@<.5ex>[d] \ar@<-.5ex>[d] \\
    M 
      & 
      & N
  }} .
\end{equation}
When $P$ is a Morita morphism, $r$ is also a surjective submersion and we may also pull back $\U$ to $P$. The morphism $\Phi_P$ then induces a base-preserving morphism 
\begin{equation}
\label{eq:isomorphism-of-pullback-groupoids-from-Morita-equivalence}
  \tilde\Phi_P:l^\pb\L\to r^\pb\U
\end{equation}
which is readily seen to be a diffeomorphism, with inverse $\tilde\Phi_{P^\op}$.

\begin{ex}
\label{ex:gen-mor-global-section}
  Any global section $\sigma:M\to P$ of the left moment map $l$ yields a Lie groupoid morphism $\varphi_\sigma:\L\to \U$ defined by $\varphi_\sigma(\generic) = \Phi_P(\sigma(\ta(\generic)), \generic, \sigma(\so(\generic)))$, and $P$ is then bi-equivariantly diffeomorphic to the bundlization of $\varphi_\sigma$ through the map $M\times_N\U\to P:(m,\genericU)\mapsto \sigma(m)\cdot\genericU$.
  Conversely, any bundlization has an obvious global section.
  Hence, $l$ has a global section if and only if $P$ is isomorphic to the bundlization of a Lie groupoid morphism.
\end{ex}

\begin{ex}[Induced morphism of units]
\label{ex:induced-morphism-identity-gen-mor}
  The induced morphism $\Phi_{\Id_\L}:\ta^\pb\L\to\L$ (over $\so:\L\to M$) of a unit generalized morphism $\Id_\L$ is 
  \begin{align}
  \label{eq:induced-morphism-identity-gen-mor}
    \Phi_{\Id_\L} 
    : \L \times_M^{\ta,\ta} \L \times_M^{\so,\ta} \L \to \L
    : (\generic_1, \generic_2, \generic_3) \mapsto \generic_1^{-1} \cdot \generic_2 \cdot \generic_3.
  \end{align}
\end{ex}

\begin{ex}[Induced morphism of bundlizations]
\label{ex:induced-morphism-bundlization-of-morphism}
  The induced morphism $\Phi_{P_\varphi}:l^\pb\L\to\U$ (over $r:P_\varphi\to N$) of the bundlization $P_\varphi=M\times_N\U$ of a morphism $\varphi:\L\to\U$ as in Example \ref{ex:bundlization-generalized-morphism} is 
  \begin{align}
  \label{eq:induced-morphism-bundlization-of-morphism}
    \Phi_{P_\varphi} = \Phi_{\Id_\U} \circ \tilde\varphi
  \end{align}
  where $\tilde \varphi$ is the Lie groupoid morphism from $l^\pb\L\toto P_\varphi$ to $\ta^\pb\U\toto\U$ given by $((m,\genericU),\generic,(m',\genericU'))\mapsto(\genericU,\varphi(\generic),\genericU')$.
\end{ex}


\subsection{An adjoint action}
\label{ssec:adjoint-action}

It is well-known that representations of Lie groupoids in the sense of Eq.~\eqref{eq:Lie-groupoid-module} are not general enough to include a natural notion of adjoint representation. Instead, representations up to homotopy \cite{arias_representations_2013} or VB-groupoids \cite{gracia-saz_vb_2017} are necessary.
Here, we do not need a full-blown adjoint representation as, in the end, our Atiyah class lies in the cohomology with values in a usual module. As a computational replacement, the following $T\L$-action on the anchor $L\to TM$ will be sufficient.
Although very much related, it should not be confused with the representation of the first jet bundle $J^1\L$ on the vector bundle $L\to M$ \cite{crainic_secondary_2005}.
 
Consider the Lie algebroid morphism 
\begin{equation}
\label{eq:def-kappa}
  \Add = \Lie(\Phi_{\Id_\L}):\ta^\pb\ll\to\ll ,
\end{equation}
over $\so:\L\to M$ corresponding to the Lie groupoid morphism of Example \ref{ex:induced-morphism-identity-gen-mor}.

\begin{lem}
\label{lem:kappa}
  The Lie algebroid morphism $\Add$ reads
  \begin{align}
  \label{eq:def-adjoint-gen-mor}
    \Add(u,\lambda)   
    &= u^{-1} \bullet \lambda \bullet 0_\generic
  \end{align}
  for all $(u,\lambda)\in (\ta^\pb \ll)_\generic$, $\generic\in\L$.
\end{lem}

\begin{proof}
  Let $(\generic^u(t), \generic^\lambda(t), \generic)$
  be a path through $1_\generic=(\generic,1_{\ta(\generic)},\generic)$ in $\ta^\pb\L$, with $\generic^u(t)$ (respectively, $\generic^\lambda(t)$) tangent to $u$ (respectively, $\lambda$) at 0.
  Then
  \begin{align*}
  \label{eq:}
    \Add(u,\lambda)
    &= \left. \frac{d}{dt} \generic^u(t)^{-1} \cdot \generic^\lambda(t) \cdot \generic \right|_{t=0} \\
    &= u^{-1} \bullet \lambda \bullet 0_\generic .
  \end{align*}
\end{proof}

As expected from its definition, $\Add$ is closely related to the (right) adjoint action of bisections on sections of the Lie algebroid.
\begin{lem}
\label{lem:adjoint-formula}
	Let $\bisection$ be a local bisection of 
  the groupoid	$\L$ and $l\in \Gamma(\ll)$ a section of its Lie algebroid  $\ll$.
	For all $\generic \in \bisection$ with source $x$ and target $y$,  the value at $x \in M$ of the adjoint action of $\bisection$ on $l$ depends only on $l_y$ and on the unique tangent vector in $T_\generic \bisection$
	that $T_\generic \ta $ maps to $ \rho (l_y)$. Explicitly:
  \begin{equation}
  \label{eq:adjoint_explicit}
    \left( \Ad_{\bisection^{-1}} l \right)_x
    = \Add(u,l_y)
  \end{equation}
where $u = T\bisection_\ta \big( \rho (l_y) \big)$.
\end{lem}

\begin{proof}
  For all $\generic \in \bisection$ with source $x$ and target $y$, we have
  \begin{align*}
    \left. \Ad_{\bisection^{-1}} l \right|_x 
    &= \left. \frac{\diff}{\diff \param} \left( \bisection^{-1}
              \star \exp(\param l) 
              \star \bisection \right)_\so(x) \right|_{\param =0} \\
    &= \left. \frac{\diff}{\diff \param} 
              \left( \bisection_\ta {\left( \ta(\exp(\param l)_\so(y)) \right)} \right)^{-1}
              \cdot \exp(\param l)_\so(y)
              \cdot \generic \right|_{\param =0} \\
    &= \left( T\bisection_\ta \left( \rho (l_y) \right) \right)^{-1} \bullet l_y \bullet 0_\generic \\
	&= \Add(u, l_y)
  \end{align*}
  where $u = T\bisection_\ta \big( \rho (l_y) \big)$.
\end{proof}

We list here some properties of $\Add$ that follow in a straightforward manner from the definitions. Below, $\lvf{\cdot}$ and $\rvf{\cdot}$ are the left- and right-invariant vector fields defined in \eqref{eq:left-and-right-invariant-vector-fields}.

\begin{prop}
\label{prop:properties-kappa}
  For every $\generic\in \L$ with source $x$ and target $y$, we have
  \begin{enumerate}
    \item \label{prop:properties-kappa-lvf}
      $\Add\big( \lvf{\lambda}|_\generic, \lambda \big) = 0$ for all $\lambda\in \ll_x$,
    \item \label{prop:properties-kappa-rvf}
      $\Add\big( \rvf{\lambda}|_\generic, 0_x \big) = -\lambda$ for all $\lambda\in \ll_y$,
    \item \label{prop:properties-kappa-action-of-units}
      $\Add\big( 0_x, \lambda) = \lambda$ for all $\lambda\in \ll_x$.
  \end{enumerate}
  For every $(u,\lambda) \in (\ta^\pb \ll)|_\generic$ and every $u'\in T\L$, we have
  \begin{enumerate}[resume]
    \item \label{prop:properties-kappa-composition}
      $\Add(u\bullet u',\lambda) = \Add(u',\Add(u,\lambda))$ whenever $u\bullet u'$ is defined.
  \end{enumerate}
\end{prop}

By items \ref{prop:properties-kappa-action-of-units} and \ref{prop:properties-kappa-composition} of Proposition \ref{prop:properties-kappa}, we may see the map
\begin{align*}
	\Add &: T\L \times_{TM}^{T\ta,\rho} \ll \to \ll
\end{align*}
as a \emph{right $T\L$-action on the anchor map} $\rho:L\to TM$.
In that context, we will use the notation 
\begin{equation}
\label{eq:definition-Add-u}
	\Add_u=\Add(u,\cdot) : \rho^{-1}(T\ta(u)) \to \rho^{-1}(T\so(u))
\end{equation}
for all $u\in T\L$.


\subsection{Equivariant principal bundles}
\label{ssec:equivariant-principal-bundles}

In this section, we slightly extend the notion of generalized morphism.
Specifically, we consider (right) principal bundles with a compatible left action of a Lie groupoid, without requiring the base manifolds of the principal bundle and of the Lie groupoid acting on the left to agree.

In the case of groups instead of groupoids, these objects are already very natural and were considered in \cite{wang_invariant_1958} and \cite{bordemann_atiyah_2012}, for example. 
Apart from the passage to groupoids, what seems to be new here is the fact that these objects, as well as the connection forms on them, can be composed.

\subsubsection*{Definition}
We start with the definition of our bundles and their composition.

\begin{defn}
  \label{defn:epb}
  Let $\L\toto M$ and $\U\toto N$ be two Lie groupoids.
  An \emph{$\L$-equivariant principal $\U$-bundle} over a manifold $X$ is a surjective submersion $\pi:P\to X$ from a manifold $P$ with a left $\L$-action on a map $l:P\to M$ and a right $\U$-action on a map $r:P\to N$ such that:
  \begin{enumerate}[label=(\arabic*)]
    \item \label{def:equivariant-principal-bundle-1}
      the left and right actions commute,
    \item \label{def:equivariant-principal-bundle-2}
      $l$ is a surjective submersion,
    \item \label{def:equivariant-principal-bundle-3}
      the map
      \begin{equation}
      \label{eq:epb-inverse-of-right-action-on-P}
        P \times_N \U \to P\times_X P : (p,\genericU) \mapsto (p,p\cdot \genericU)
      \end{equation}
      is a diffeomorphism.
  \end{enumerate}
  The maps $l$ and $r$ are called the left and right moment maps, respectively.
\end{defn}

As before, the diffeomorphism \eqref{eq:epb-inverse-of-right-action-on-P} is a Lie groupoid morphism, which yields a \emph{division map} $D_P:P\times_X P\to \U$ with the same invariance and equivariance properties as those of generalized morphisms (see Eqs. \eqref{eq:gen-mor-division-map}--\eqref{eq:gen-mor-invariance-equivariance-division-map-2}).

The axioms imply that $l$ descends to $\ol l:X\to M$, and $P$ will often be represented by the diagram
\[
  \vcenter{\xymatrix{
    \L \ar@<.5ex>[d] \ar@<-.5ex>[d] 
      & P \ar@{>>}[dl]_l \ar@{>>}[d]^\pi \ar[dr]^r 
      & \U \ar@<.5ex>[d] \ar@<-.5ex>[d] \\
    M 
      & X \ar@{>>}[l]_{\ol l} 
      & N
  }} .
\]
When there can be no confusion about the actions, we will use the notation $_\L P_\U$ as shorthand for the above diagram.

When $X=M$ and $l=\pi$, we recover the notion of \emph{generalized morphism} between Lie groupoids.
Just as generalized morphisms, equivariant principal bundles may be composed: if $P$ is an $\L$-equivariant principal $\U$-bundle, and $Q$ is a $\U$-equivariant principal $\V$-bundle, then $P\times_N Q$ is a smooth manifold thanks to condition~\ref{def:equivariant-principal-bundle-2}.
Moreover, by condition~\ref{def:equivariant-principal-bundle-3} it has a proper and free $\U$-action $(u,(p,q))\mapsto (pu^{-1},uq)$, so that the quotient $P\circ Q=\frac{P\times_{N}Q}{\U}$ is a smooth manifold as well.
The latter still has a free and proper $\V$-action and is in fact an $\L$-equivariant principal $\V$-bundle over the manifold $(P\circ Q)/\V$, which we call the \emph{composition} of $P$ and $Q$.

The equivalences between two $\L$-equivariant principal $\U$-bundles $P$ and $P'$ (over $X$ and $X'$, respectively) are again the bi-equivariant diffeomorphisms $\phi:P\to P'$.
Note that these induce $\L$-equivariant maps $\psi:X\to X'$ on the bases such that $\pi'\circ \phi = \psi\circ \pi$.

The proof of the following result is a straightforward adaptation of the corresponding proof for generalized morphisms in, e.g., \cite{blohmann_stacky_2008}.

\begin{prop}
  \label{prop:epb-bicategory}
  The Lie groupoids with equivariant principal bundles as 1-morphisms and bi-equivariant diffeomorphisms as 2-morphisms form a bicategory. 
\end{prop}

The equivariant principal bundles form a generalization of the generalized morphisms which only includes more ``degenerate'' morphisms.
Indeed note that, keeping our previous notation, the base $(P\circ Q)/\V$ of a composition $P\circ Q$ fibers over $X$ through a surjective submersion whose fibers are those of $\overline{l_Q}:Y\to N$, where $Y$ is the base of $Q$.
As a result, the dimension of the fibers of $\overline{l_{P\circ Q}}$ is the sum of the dimensions of the fibers of $\overline{l_{P}}$ and of $\overline{l_{Q}}$.
This implies that a weakly invertible equivariant principal bundle is, up to bi-equivariant diffeomorphism, a weakly invertible generalized morphism (i.e.\ a Morita morphism).

\subsubsection*{Examples}

The main example of ``degenerate'' morphism that we will consider is induced by a Lie groupoid pair $(\L,\A)$.
Given such a pair, $\L$ can be considered in two ways as an equivariant principal bundle: either as an $\A$-equivariant principal $\L$-bundle, or as an $\L$-equivariant principal $\A$-bundle. 
Those are denoted respectively by ${}_\A\L_\L$ and ${}_\L\L_\A$.
In both cases, the left and right moment maps are just the target and source maps, respectively, and the actions are by left and right translations.
The first version, ${}_\A\L_\L$, has $X=M$ and is isomorphic to the bundlization $P_i$ of the inclusion morphism $i:\A\to \L$, since $P_i=M\times_M \L$ is bi-equivariantly diffeomorphic to $\L$.
It is thus a generalized morphism.
The second version, ${}_\L\L_\A$, sits over the homogeneous space $X=\L/\A$ and is not a generalized morphism as soon as $X\neq M$.

The composition of the $\L$-equivariant principal $\A$-bundle ${}_\L\L_\A$ with a generalized morphism $P$ from $\A$ to a Lie groupoid $\U$ yields an $\L$-equivariant principal $\U$-bundle $Q={}_\L\L_\A \circ P$ over $\L/\A$. This kind of composition will be studied extensively.

\begin{ex}
\label{ex:LLA-morphism-phi}
  Let $\varphi: \A \to \U$ be a Lie groupoid morphism over $\varphi_0:M\to N$.
  Then the fibered product $\L \times_{N} \U = \L \times_N^{{\varphi_0}\circ{\so},\ta} \U$ is a smooth manifold which admits a free and proper right $\A$-action defined by
  \[  (\L\times_N \U) \times_M \A \to \L\times_N \U : ((\generic,\genericU),\generic') \mapsto (\generic\cdot \generic',\varphi(\generic'^{-1})\cdot\genericU) .  \]
  The corresponding quotient 
  \begin{equation}
  \label{eq:defn-Q}
    Q=\frac{\L \times_{N} \U}{\A} 
  \end{equation}
  is then an $\L$-equivariant principal $\U$-bundle over $\L/\A$ when endowed with the left and right moment maps $l:Q\to M:[(g,u)]\mapsto \ta(g)$ and $r:Q\to N:(g,u)\mapsto \so(u)$, with the projection $\pi:Q\to \L/\A:[(g,u)]\mapsto [g]$, and with the actions
  \begin{align*}
    \L \times_M Q \to Q &: (\generic,[(\generic',\genericU)]) \mapsto [(\generic\cdot\generic',\genericU)] \\
    Q \times_N \U \to Q &: ([(\generic,\genericU)],\genericU') \mapsto [(\generic,\genericU\cdot\genericU')] .
  \end{align*}
  It is canonically isomorphic to the composition ${}_\L\L_\A \circ P_\varphi$.
\end{ex}

The bundles constructed in Example\ \ref{ex:LLA-morphism-phi} exhaust all $\L$-equivariant principal $\U$-bundles $Q$ over $X$ such that
\begin{enumerate}
  \item the $\L$-action on $X$ is transitive, and
  \item the left moment map has a smooth global section.
\end{enumerate}
Indeed, let $\sigma$ be such a section, let $\ol\sigma=\pi\circ\sigma$ be the corresponding section of $\ol l$, and denote by $\A$ the closed subgroupoid of $\L$ that sends $\ol\sigma(M)$ to itself.
Then, $X$ is equivariantly diffeomorphic to $\L/\A$.
Moreover, there is a Lie groupoid morphism $\phi:\A\to\U$ defined by the relation $\generic\cdot \sigma(\so(\generic)) = \sigma(\ta(\generic))\cdot\phi(\generic)$.
Now, we have a map $r\circ\sigma:M\to N$ and we may consider $\L\times_N\U\to Q:(\generic,\genericU)\mapsto \generic\cdot\sigma(\so(\gamma))\cdot\genericU$, which descends to an isomophism (i.e.\ bi-equivariant diffeomorphism) $\frac{\L \times_{N} \U}{\A}  \to Q$ as promised.
  
Dropping the existence of a section of the left moment map, we still have that $X$ is a homogeneous space of $\L$ and thus has a section $\ol\sigma:M\to X$. Defining $\A$ as above and $P=\pi^{-1}(\sigma(M))$, we get that $P$ is a generalized morphism from $\A$ to $\U$ and that there is a map $\L\times_M P\to Q:(\generic,p)\mapsto \generic\cdot p$, which descends to an isomorphism $\frac{\L \times_{M} P}{\A}  \to Q$.
  
We thus proved the following result about principal bundles over homogeneous spaces.
\begin{prop}
\label{prop:exhaust}
  Any $\L$-equivariant principal bundle $Q$ over $X$ for which the $\L$-action on $X$ is transitive is isomorphic to a composition ${}_\L\L_\A \circ P$ where $\A$ is the stabilizer of some section $\sigma:M\to X$ and $P$ is a generalized morphism. The latter generalized morphism is isomorphic to (the bundlization of) a morphism if and only if the left moment map of $Q$ admits a section.
\end{prop}

\begin{rmk}
\label{rmk:epb-seen-as-principal-bibundles}
  If $P$ is an $\L$-equivariant principal $\U$-bundle over $X$, there is an induced $\L$-action on $X$.
  Hence, we may consider the action groupoid $\L\ltimes X$ over $X$.
  Then $P$ is actually a generalized morphism from $\L\ltimes X$ to $\U$ with left action defined by $(\generic,x)\cdot p = \generic\cdot p$ whenever $\pi(p)=x$.
  \[
    \vcenter{\xymatrix{
      \L \ar@<.5ex>[d] \ar@<-.5ex>[d] 
        & P \ar@{>>}[dl]_l \ar@{>>}[d]^\pi \ar[dr]^r 
        & \U \ar@<.5ex>[d] \ar@<-.5ex>[d] \\
      M 
        & X \ar@{>>}[l]_{\ol l} 
        & N
    }}
    \qquad\leadsto\qquad
    \vcenter{\xymatrix{
      \L\ltimes X \ar@<.5ex>[d] \ar@<-.5ex>[d] 
        & P \ar@{>>}[dl]_{\pi} \ar[dr]^r 
        & \U \ar@<.5ex>[d] \ar@<-.5ex>[d] \\
      X 
        & 
        & N
    }}
  \]
  
  So in some sense, equivariant principal bundles are just a special kind of generalized morphisms, but we do not want to see them in this way. We really want to see them as ``morphisms'' from $\L$ to $\U$.
\end{rmk}

\begin{ex}
\label{ex:manifolds-not-full-subcategory-of-epb}
  While the association $X\mapsto \underline X=(X\toto X)$ realizes the category of manifolds as a full subcategory of the bicategory of Lie groupoids defined in Proposition \ref{prop:generalized-morphisms-bicategory}, the analogous result is not true for the bicategory where generalized morphisms as 1-morphisms are replaced by equivariant principal bundles as in Proposition \ref{prop:epb-bicategory}.
  Indeed, the equivariant principal bundles between $\underline X$ and $\underline Y$ for two smooth manifolds $X$ and $Y$ are the \emph{multivalued functions from $X$ to $Y$}.
  Here, a multivalued function from $X$ to $Y$ is a smooth manifold $Z$ together with a surjective submersion to $X$ and a smooth map to $Y$.
\end{ex}

\subsubsection*{Connections}

The connection forms defined below are natural extensions of those on usual principal bundles with structure group, and on principal bundles with structure groupoid.
To define them, let us first explain some notation.

If $\L\times_M P\to P$ is a left action on a map $P\to M$, there is an induced $T\L$-action on $TP$ simply obtained by applying the tangent functor. We will write $\generic\cdot p$ for the action of $\generic\in \L$ on $p\in P$, and $u\cdot X$ for the action of $u\in T\L$ on $X\in TP$, whenever defined.
We use the same notation for a right action $P\times_N\U\to P$, and write $p\cdot \genericU$ for the action of $\genericU\in \U$ on $p\in P$, and $X\cdot v$ for the action of $v\in T\U$ on $X\in TP$, whenever defined.
The zero vector at a point $p\in P$ is denoted by $0_p$.
Recall that $T^lP$ is the subbundle of $TP$ of vectors tangent to the $l$-fibers, $T^lP=\ker Tl$, and similarly for $T^\pi P$. 

\begin{defn}
  \label{defn:infinitesimal-v-f-right-action}
  The \emph{infinitesimal vector fields} of a right action $R:P\times_{N}\U\to P$ of a Lie groupoid $\U\toto N$ on a manifold $P$ are the vector fields $\hat v\in \mathfrak X(P)$ defined by 
  \begin{align*}
    \hat v|_p = \left. \frac{d}{dt} R_{\exp(tv)^{-1}}(p) \right|_{t=0} = 0_p \cdot (v|_p)^{-1}
  \end{align*}
  for all $v\in\Gamma(\uu)$ and $p\in P$.
\end{defn}

Let us consider now an equivariant principal bundle $P$ and momentarily denote by $\psi$ the Lie groupoid diffeomorphism \eqref{eq:epb-inverse-of-right-action-on-P}.
Then the infinitesimal vector fields map coincides with the Lie algebroid isomorphism
\begin{equation}
\label{eq:Lie-PxU-to-PxP}
  r^*U \cong \Lie(P\rtimes \U) \xrightarrow{Lie(\psi)} \Lie (P\times_X P) \cong T^\pi P .
\end{equation}
A connection form is a bundle map from $TP$ to $r^*U$ that extends the inverse of the above map, in an equivariant way.
\begin{defn}
  \label{defn:connection-epb}
  A \emph{connection form} (resp., \emph{fibered connection form}) on an $\L$-equivariant principal $\U$-bundle $P$ is an $r^*\uu$-valued 1-form $\omega\in\Omega^1(P,r^*\uu)$ (resp., an $r^*\uu$-valued 1-form on the $l$-fibers $\omega\in\Omega^1_l(P,r^*\uu)$) such that
  \begin{enumerate}[label=(F\arabic*)]
    \item \label{defn:connection-epb-1}
      $\omega(\hat v) = r^*v$, for all $v\in\Gamma(\uu)$,
    \item \label{defn:connection-epb-2}
      ${\Ad_{\Sigma^{-1}}}\circ{\omega} = R_\Sigma^*\omega$, for all $\Sigma\in\Bis(\U)$,
    \item \label{defn:connection-epb-3}
      $L_\Sigma^*\omega = \omega$, for all $\Sigma\in\Bis(\L)$.
  \end{enumerate}
\end{defn}

\begin{rmk}
\label{rmk:fibered-connection-forms-unique-on-generalized-morphisms}
Fibered connection forms are specifically designed for equivariant principal bundles which are \emph{not} generalized morphisms. 
Indeed, if $P$ is a generalized morphism then $l=\pi$, the infinitesimal vector fields of the right $\U$-action span the tangent space to the $l$-fibers at all points, and condition \ref{defn:connection-epb-1} in Definition \ref{defn:connection-epb} entirely determines the values of $\omega$ at all points.
On the other hand, since $R_\bisection(\hat v)=\left(\Ad_{\bisection^{-1}}(v)\right)^{\wedge}$ and since the left and right actions commute, the map defined by \ref{defn:connection-epb-1} also satisfies \ref{defn:connection-epb-2} and \ref{defn:connection-epb-3}.
Hence, there exists one and only one fibered connection form on a generalized morphism, given by the inverse of \eqref{eq:Lie-PxU-to-PxP}.
We will call it the \emph{Maurer--Cartan form} of a generalized morphism.

Connection forms, on the other hand, reduce to the usual notion of connection form on a principal bundle when $\U$ is a Lie group and $\L$ is the (trivial) groupoid $M\toto M$ of a manifold.
Moreover, when $\U$ is a Lie group and $\L$ is any Lie groupoid, our connection forms coincide with the connection forms on a principal bundle over a groupoid defined in \cite{laurent-gengoux_chern_2007}.

We stress that both kinds of connections need not exist in general, due to the left-invariance condition, just like $G$-invariant connections on a homogeneous space $G/H$ need not exist in general.
It is the very purpose of this paper to study the obstruction to their existence and some constructions that can be made when such connections exist.
\end{rmk}

In this paper, we will only use \emph{fibered} connection forms $\cf\in\Omega^1_l(P,r^*\uu)$ and, more precisely, we will only use the corresponding bundle maps
$$
  \vcenter{\xymatrix{
    T^lP \ar[r]^{\bmap} \ar[d]
      & \uu \ar[d] \\
    P \ar[r]^r
      & N
  }} .
$$
We will use the same letter $\bmap$ for fibered connection forms and the corresponding bundle maps, and call both \emph{fibered connection forms}.
These bundle maps enjoy similar properties to \ref{defn:connection-epb-1}--\ref{defn:connection-epb-3} above, however, in case we want to avoid using bisections, we have two simpler axioms: $T\U$-equivariance and $\L$-invariance.

\begin{prop}
\label{prop:characterization-bundle-maps-of-connection-forms}
  A vector bundle morphism $\bmap:T^lP\to\uu$ over $r: P \to N$ is (the bundle map of) a fibered connection form if and only if it is an anchored map such that, for all $Y\in T^lP$, $v\in T\U$ and $\generic\in\L$,
  \begin{enumerate}[label=(B\arabic*)]
    \item \label{prop:characterization-bundle-maps-of-connection-forms-2}
      $\Add_v \bmap(Y) = \bmap(Y\cdot v)$ whenever $Tr(Y)=T\ta(v)$,
    \item \label{prop:characterization-bundle-maps-of-connection-forms-3}
      $\bmap(0_\generic\cdot Y) = \bmap(Y)$ whenever $Y\in T^l_pP$ with $l(p)=\so(\generic)$.
  \end{enumerate}
\end{prop}

\begin{proof}
We will show that \ref{prop:characterization-bundle-maps-of-connection-forms-2} is equivalent to \ref{defn:connection-epb-1} + \ref{defn:connection-epb-2}, and that \ref{prop:characterization-bundle-maps-of-connection-forms-3} is equivalent to \ref{defn:connection-epb-3}.

Assume $\bmap:T^lP\to \uu$ satisfies \ref{prop:characterization-bundle-maps-of-connection-forms-2}.
Taking $Y=0_p$ and $v=u^{-1}$ with $u\in \uu_{r(p)}$, we get
\[
  u 
  = u \bullet 0_{1_{r(p)}} \bullet 0_{(1_{r(p)})^{-1}} 
  = \Add(u^{-1},0_{1_{r(p)}})
  = \bmap (0_p \cdot u^{-1}) 
  = \bmap (\hat u|_p), 
\]
which is \ref{defn:connection-epb-1}.
By Lemma \ref{lem:adjoint-formula}, \ref{prop:characterization-bundle-maps-of-connection-forms-2} also implies \ref{defn:connection-epb-2}.
Conversely, assume $\bmap:T^lP\to \uu$ is a bundle map satisfying \ref{defn:connection-epb-1} and \ref{defn:connection-epb-2}.
Notice that any pair $(Y,v)\in T^l_{(p,\genericU)}(P\times_N \U)$ can be written as
\[  (Y,v) = (0_p,0_\genericU \bullet u^{-1}) + (Y,T\bisection_\ta(Tr(Y)))  \]
for some bisection $\bisection\in\Bis(\U)$ and some $u\in\uu_{\so(\genericU)}$.
The same arguments as for the converse show that \ref{prop:characterization-bundle-maps-of-connection-forms-2} holds for each term in this sum. Hence, it holds in general.

The equivalence between \ref{prop:characterization-bundle-maps-of-connection-forms-3} and \ref{defn:connection-epb-3} follows from the fact that, for a bisection $\bisection\in\Bis(\L)$ and a vector $Y\in T^l_pP$, we have
\begin{align*}
  \bisection \star Y &= T\bisection_\so(Tl(Y)) \cdot Y \\
  &= T\bisection_\so(0_{l(p)}) \cdot Y \\
  &= 0_{\bisection_\so(l(p))} \cdot Y .
\end{align*}
\end{proof}

\begin{ex}[Units]
\label{ex:maurer-cartan-form}
  The Maurer--Cartan form (see Remark \ref{rmk:fibered-connection-forms-unique-on-generalized-morphisms}) of a unit generalized morphism $\Id_\L$ is the bundle map $\tau=\tau_\L:T^\ta\L\to\ll$ defined by
  \begin{align}
  \label{eq:maurer-cartan-form}
    \tau(u) = u^{-1} \bullet 0_{\generic} = \Add(u,0_{\ta(\generic)})
  \end{align}
  for all $u\in (T^\ta\L)_\generic$.
  Hence, it is the usual Maurer--Cartan form of the Lie groupoid $\L$.
  It is natural in $\L$: if $\phi:\L\to\U$ is a Lie groupoid morphism, then $\tau_\U\circ T\phi = \Lie(\phi)\circ \tau_\L$. 
\end{ex}

\begin{ex}[Bundlizations]
\label{ex:fibered-connection-form-on-bundlization}
  The Maurer--Cartan form of the bundlization of a morphism $\varphi:\L\to\U$ is induced from the Maurer--Cartan form of $\U$.
  Indeed, the fibered tangent bundle $T^lP_\varphi$ is the set of tangent vectors $(X,v)$ in $TM\times_{TN}T\U$ that project to zero through the projection on the first component, hence the set of vectors $(0,v)$ with $v\in T^\ta\U$.
  The connection form on $P_\varphi$ is then the bundle map $\omega_\varphi:T^lP_\varphi\to\uu$ defined by
  \[
    \bmap_\varphi(0,v) = \tau_\U(v)
  \]
  for all $(0,v)\in T^lP_\varphi$.
\end{ex}

Just like equivariant principal bundles can be composed, connections on such bundles can be composed as well (see Section~\ref{sec:reductive-homogeneous-spaces} for an application of that result).

\begin{prop}
\label{prop:composition-of-connections}
  Let $P$ be an $\L$-equivariant principal $\A$-bundle, and $Q$ an $\A$-equivariant principal $\U$-bundle, for some Lie groupoids $\L$, $\A$, and $\U$.
  Let $\bmap_P$ and $\bmap_Q$ be fibered connection forms on $P$ and $Q$, respectively.
  Then there is a fibered connection form $\bmap$ on $P\circ Q$ defined by
  \begin{equation*}
    \bmap([(X,Y)]) = \bmap_Q \left( Y- \bmap_P(X)\cdot 0_q \right)
  \end{equation*}
  for all $(X,Y)\in T^l(P\times_M^{r,l} Q)$, where $M$ is the base manifold of $\A$, and $q$ is the base point of $Y$.
\end{prop}
\begin{proof}
  Property \ref{prop:characterization-bundle-maps-of-connection-forms-2} for $\bmap_P$ and property \ref{prop:characterization-bundle-maps-of-connection-forms-3} for $\bmap_Q$ imply that $(X,Y) \mapsto \bmap_Q \left( Y- \bmap_P(X)\cdot 0_q \right)$ is $T\A$-invariant, so $\bmap$ is well-defined.
  Then \ref{prop:characterization-bundle-maps-of-connection-forms-2} follows from \ref{prop:characterization-bundle-maps-of-connection-forms-2} for $\bmap_Q$ and from the commutativity of the left and right actions on $P$.
  And \ref{prop:characterization-bundle-maps-of-connection-forms-3} follows from \ref{prop:characterization-bundle-maps-of-connection-forms-3} for $\bmap_P$.
\end{proof}

It is straightforward to check that this composition is compatible with the Maurer--Cartan form of generalized morphisms, and that equivariant principal bundles with fibered connection also form a bicategory, which contains generalized morphisms.


\subsection{Associated vector bundles}
\label{ssec:associated-vector-bundles}

In this section, for a Lie groupoid $\L\toto M$ and a vector bundle $E\to N$, we exhibit an explicit correspondence between fibered connection forms on an $\L$-equivariant principal $\fb(E)$-bundle $P$ and $\L$-invariant fibered connections on the associated vector bundle $P(E)$. 

The results in this section are natural analogues of the classical notions. 

\begin{defn}
  \label{defn:associated-vector-bundle}
  Let $\L$ and $\U$ be Lie groupoids over $M$ and $N$, respectively, let $E\to N$ be a $\U$-module, and let $P$ be an $\L$-equivariant principal $\U$-bundle over $X$.
  The \emph{associated $\L$-module} to $P$ and $E$ is the associated vector bundle
  \[ P(E) = \frac{P \times_N E}{\U} \to X  \]
  with the $\L$-module structure $\L \times_M P(E) \to P(E)$ defined by
  \[  \generic \cdot [(p,e)] = [(\generic\cdot p,e)] .  \]
\end{defn}

We will denote by $\mu$ the (left) $\U$-action on $P\times_N E$ that defines $P(E)$:
\begin{align*}
  \mu_\genericU(p,e) = (p\cdot \genericU^{-1},\genericU\cdot e)
\end{align*}
where $\genericU\in \U$, $p\in P$ and $e\in E$ satisfy $r(p)=q(e)=\so(\genericU)$.
Also, we will denote by $\pi_{P(E)}$ the projection $(p,e)\mapsto [(p,e)]$ from $P\times_N E$ to $P(E)$.

Any vector bundle $E$ is canonically a $\fb(E)$-module.
\emph{In the rest of this section, we will only consider the case where $P$ is an $\L$-equivariant principal $\fb(E)$-bundle}.

\begin{ex}\label{ex:avecE}
  Let $(\L,\A)$ be a Lie groupoid pair over $M$ and $E$ be an $\A$-module.
  The Lie groupoid pair defines an $\L$-equivariant principal $\A$-bundle over $\L/\A$, and the associated $\L$-module to ${}_\L\L_\A$ and $E$ is $\frac{\L\times_M E}{\A}$.
  
  Considering the $\A$-module structure on $E$ as a Lie groupoid morphism from $\A$ to the frame groupoid $\fb(E)$ we get, as in Example \ref{ex:LLA-morphism-phi}, an $\L$-equivariant principal $\fb(E)$-bundle $P=\frac{\L \times_M \fb(E)}{\A}$ over $X=\L/\A$.
  The associated $\L$-module $P(E)=\frac{P \times_M E}{\fb(E)}$ is isomorphic to $\frac{\L\times_M E}{\A}$.
\end{ex}

The following result will be needed in the proof of Proposition \ref{prop:sections-equivariant-maps}.
\begin{lem}
  \label{lem:pistarPofE-isom-fiberproductPandE}
  There is a canonical isomorphism $\Psi:\pi^*P(E)\to P\times_N E$ of vector bundles over $P$.
  It is $\L$- and $\fb(E)$-equivariant, in the sense that 
  \begin{align*}
    \Psi(\generic\cdot p,\generic\cdot \epsilon) &= \generic \cdot \Psi(p,\epsilon) \\
    \Psi(p\cdot \phi^{-1},\epsilon) &= \mu_\phi \Psi(p,\epsilon) 
  \end{align*}
  for all $(p,\epsilon)\in \pi^*P(E)$, $\generic\in\L$ such that $\so(\generic)=l(p)$, and $\phi\in\fb(E)$ such that $\so(\phi)=r(p)$.
\end{lem}
\begin{proof}
  This follows directly from the isomorphism \eqref{eq:epb-inverse-of-right-action-on-P}.
  Explicitly, consider the map
  \begin{align}
    \label{eq:definitionPsi}
    P\times_X(P\times_N E) \to P\times_N E : (p',(p,e)) \mapsto (p',D_P(p',p)\cdot e) .
  \end{align}
  It is smooth and invariant under the $\fb(E)$-action on $P\times_X(P\times_N E)$ given by $(\phi,(p',(p,e)))\mapsto (p',\mu_\phi(p,e))$.
  Hence it descends to a smooth map $\Psi:P\times_X P(E) = \pi^*P(E) \to P\times_N E$ which is moreover an inverse of the map $P\times_N E\to \pi^*P(E):(p,e)\mapsto (p,[(p,e)])$.
  
  The equivariance is now obvious from \eqref{eq:definitionPsi} and the equivariance of the division map $D_P$.
\end{proof}

We will use the notation
\[
  \vcenter{\xymatrix{
    \L \ar@<.5ex>[d] \ar@<-.5ex>[d] 
      & P \ar@{>>}[dl]_l \ar@{>>}[d]^\pi \ar[dr]^r 
      & \fb(E) \ar@<.5ex>[d] \ar@<-.5ex>[d] \\
    M 
      & X \ar@{>>}[l]_{\ol l} 
      & N
  }} .
\]
Note that $X$ is still fibered over $M$, so that we may consider the subbundle $T^{\ol l}X = \ker T\ol l \subset TX$, and correspondingly consider \emph{fibered} differential $k$-forms on $X$.
A \emph{$P(E)$-valued fibered $k$-form on $X$} is a section of $\Lambda^k (T^{\ol l}X)^* \otimes P(E) \to X$.

On the other hand, we can consider the \emph{$E$-valued fibered $k$-forms on $P$}, i.e.\ the bundle maps $\alpha:\Lambda^k(T^lP)\to E$ over $r:P\to N$. 
We make the following definition.

\begin{defn}
  \label{defn:horizontal-forms}
  An $E$-valued fibered $k$-form $\alpha$ on $P$ is \emph{horizontal} if it vanishes whenever one of its arguments is in $\ker T\pi \subset T^lP$, where $\pi:P\to X$ is the projection:
  \begin{align*}
    \alpha_p(Y_1,\dots,Y_k) = 0 \quad \text{if\ $Y_i\in\ker T\pi$ for some $i\in\{1,\dots,k\}$},
  \end{align*}
  for all $p\in P$ and $Y_1,\dots,Y_k\in T^l_pP$.
  It is \emph{equivariant} if it is equivariant for the $\fb(E)$-actions on $P$ and $E$:
  \begin{align*}
    R_{\phi^{-1}}^*\alpha = \phi \circ \alpha \quad \text{for all $\phi\in \Bis(\fb(E))$.}
  \end{align*}
  On the right-hand side, $\phi$ is considered as in Example \ref{ex:bisections-of-frame-bundle-are-automorphisms} as a map $E\to E$.
\end{defn}

\begin{prop}
  \label{prop:sections-equivariant-maps}
  There is a $\Bis(\L)$-equivariant isomorphism of $C^\infty(X)$-modules 
  \begin{align*}
    \Omega^k_{\ol l}(X,P(E)) &\xrightarrow{\ \sim\ } \Omega^k_{\text{hor}}(P,E)^{\fb(E)}
  \end{align*}
  between the space $\Omega^k_{\ol l}(X,P(E))$ of $P(E)$-valued fibered $k$-forms on $X$ and the space $\Omega^k_{\text{hor}}(P,E)^{\fb(E)}$ of equivariant horizontal $E$-valued fibered $k$-forms on $P$.
\end{prop}

\begin{proof}
  Let $\alpha$ be an equivariant horizontal $E$-valued fibered $k$-form on $P$.
  Define a $P\times_N E$-valued fibered $k$-form $\tilde\eta$ on $P$ by
  \begin{align*}
    \tilde\eta_p(Y_1,\dots,Y_k) = (p,\alpha_p(Y_1,\dots,Y_k))
  \end{align*}
  for all $p\in P$ and $Y_i\in T^l_pP$.
  Since $\alpha$ is horizontal, $\tilde\eta$ only depends on $X_i=T\pi(Y_i)$, $i=1,\dots,k$.
  Moreover, since $\alpha$ is equivariant, we have 
  \begin{align*}
    (R_{\phi^{-1}}^*\tilde\eta)_p(Y_1,\dots,Y_k) 
    &= (p\phi^{-1},(R_{\phi^{-1}}^*\alpha)_p(Y_1,\dots,Y_k)) \\
    &= (p\phi^{-1},\phi \circ \alpha_p(Y_1,\dots,Y_k)) \\
    &= \mu_\phi(\tilde\eta_p(Y_1,\dots,Y_k)) 
  \end{align*}
  for all $\phi\in \Bis(\fb(E))$, $p\in P$, and $Y_i\in T^l_pP$.
  Hence the composition $\pi_{P(E)} \circ \tilde\eta$ only depends on $x=\pi(p)$ and on $X_i=T\pi(Y_i)$, $i=1,\dots,k$, so that it descends to a $P(E)$-valued fibered $k$-form $\eta$ on $X$.

  Conversely, let $\eta$ be a $P(E)$-valued fibered $k$-form on $X$.
  Define an $E$-valued fibered $k$-form on $P$ by
  \begin{align*}
    \alpha_p(Y_1,\dots,Y_k) = ({\pr_2} \circ {\Psi})(p,(\pi^*\eta)_p(Y_1,\dots,Y_k)) ,
  \end{align*}
  for all $p\in P$ and $Y_i\in T^l_pP$, where $\Psi$ is defined in Lemma~\ref{lem:pistarPofE-isom-fiberproductPandE}.
  By construction, $\alpha$ is horizontal and equivariant.

  An easy check now shows that the two assignments $\alpha\mapsto \eta$ and $\eta\mapsto \alpha$ 
  \begin{enumerate}[label=(\arabic*)]
    \item are linear inverses of each other, 
    \item are $C^\infty(X)$-linear for the multiplications $(f\alpha)_p=f(\pi(p))\alpha_p$ and $(f\eta)_x=f(x)\eta_x$, where $f\in C^\infty(X)$, and
    \item are $\Bis(\L)$-equivariant for the actions 
    \begin{align}
      \label{eq:left-action-on-equivariant-horizontal-forms}
      (g\cdot\alpha)_p(Y_1,\dots,Y_k)
      &= \alpha_{g^{-1}p}(TL_{g^{-1}}(Y_1),\dots,TL_{g^{-1}}(Y_k))
      \\
      \nonumber
      (g\cdot\eta)_x(X_1,\dots,X_k)
      &= g\cdot\left(\eta_{g^{-1}x}(TL_{g^{-1}}(X_1),\dots,TL_{g^{-1}}(X_k))\right) ,
    \end{align}
    where $g\in \Bis(\L)$.
    \qedhere
  \end{enumerate}
\end{proof}

A fibered connection on $P(E)$ can be seen as an $\fR$-linear map 
\[ \nabla:\Omega^0_{\ol l}(X,P(E))\to \Omega^1_{\ol l}(X,P(E)) \]
satisfying the Leibniz rule 
$\nabla(f\epsilon)(Y)=Y(f)\epsilon + f \nabla(\epsilon)(Y)$
for all $\epsilon\in\Omega^0_l(X,P(E))=\Gamma(P(E))$, $f\in C^\infty(X)$ and $Y\in\Gamma(T^{\ol l}X)$.
Through the correspondence of Proposition \ref{prop:sections-equivariant-maps}, this becomes an $\fR$-linear map 
\begin{align}
  \label{eq:connection-eq-hor-forms}
  \nabla:\Omega^0_{\text{hor}}(P,E)^{\fb(E)} \to \Omega^1_{\text{hor}}(P,E)^{\fb(E)}
\end{align}
satisfying the Leibniz rule 
\begin{align}
  \label{eq:connection-eq-hor-forms-leibniz-rule}
  \nabla(f\alpha)(Y) = Y(f)\alpha + f \nabla(\alpha)(Y)
\end{align}
for all $\alpha\in\Omega^0_{\text{hor}}(P,E)^{\fb(E)}$, $f\in C^\infty(P)^{\fb(E)}$ and $Y\in\Gamma(T^lP)$.

\begin{thm}
  \label{thm:connection-forms-connections-on-associated-bundles}
	Let $\L$ be a Lie groupoid, $E$ a vector bundle, and $P$ an $\L$-equivariant principal $\fb(E)$-bundle.
  There is a bijective correspondence between 
	\begin{enumerate}
		\item fibered connection forms on $P$, and
		\item $\L$-equivariant fibered connections on the vector bundle $P(E)$. 
	\end{enumerate}
\end{thm}

\begin{proof}
  As explained above, by Proposition~\ref{prop:sections-equivariant-maps}, we can consider $\L$-equivariant fibered connections on the vector bundle $P(E)$ as $\L$-equivariant maps \eqref{eq:connection-eq-hor-forms} satisfying the Leibniz rule \eqref{eq:connection-eq-hor-forms-leibniz-rule}.

  \medskip\noindent
  \textbf{Step 1.} ({From connection forms to connections})
  Let $\bmap:T^lP \to \Tlin E$ be a fibered connection form on $P$.
  For an equivariant map $\alpha\in \Omega^0_{\text{hor}}(P,E)^{\fb(E)}$, we will denote by $T^l\alpha$ its differential $T\alpha:TP\to TE$ restricted to $T^lP$.
  There is a canonical map $\ker Tq_E\to E$ obtained by considering each vertical vector $v$ at $e\in E_x$ as a vector in $E_x$. We will denote this map by $I$.

  Define a map $\nabla : \Omega^0_{\text{hor}}(P,E)^{\fb(E)} \to \Omega^1_{\text{hor}}(P,E)^{\fb(E)}$ by
  \begin{align}
  \label{eq:definition-nabla-from-omega-sharp}
    \nabla\alpha = I ( T^l\alpha - \bmap(\alpha) ) ,
  \end{align}
  for all $\alpha\in\Omega^0_{\text{hor}}(P,E)^{\fb(E)}$,
  where $\bmap(\alpha)$ is the map $T^lP\to TE$ defined by $\bmap(\alpha)(Y)=\bmap(Y)(\alpha(p))$ for all $Y\in T^l_pP$. 
  
  Since $\bmap$ is an anchored map, it follows easily that $Tq_E\circ (\bmap(\alpha))=Tr$.
  Hence, we have $Tq_E \circ (T^l\alpha - \bmap(\alpha)) = Tr - Tr = 0$.
  As a result, $\nabla\alpha$ is well-defined.
  
  The 1-form $\nabla\alpha$ is horizontal since if $Y\in\ker (T\pi)_p \subset T^l_pP$, then $Y=\hat D_p$ for some $D\in (\Tlin E)_{r(p)}$, and we have
  \begin{align*}
    T^l\alpha(\hat D) - \bmap(\hat D)(\alpha(p))
    &= D(\alpha(p)) - D(\alpha(p)) = 0
  \end{align*}
  by the equivariance of $\alpha$ and the first property of a fibered connection form.
  
  Let us show that $\nabla\alpha$ is equivariant.
  Let $\phi\in\Bis(\fb(E))$. As in Example \ref{ex:bisections-of-frame-bundle-are-automorphisms}, we consider $\phi$ as a diffeomorphism $E\to E$ that sends fibers to fibers linearly.
  Since $\alpha$ satisfies $\alpha\circ R_{\phi^{-1}}=\phi\circ\alpha$, we have 
  \begin{equation}
  \label{eq:nabla-alpha-equivariant-1}
    T^l\alpha \circ TR_{\phi^{-1}}=T\phi\circ T^l\alpha .
  \end{equation}
  The second term satisfies
  \begin{align}
  \label{eq:nabla-alpha-equivariant-2}
    \left( R^*_{\phi^{-1}}(\bmap(\alpha)) \right)(Y) 
    &= \bmap(\alpha)_{p\cdot \phi^{-1}}(TR_{\phi^{-1}}(Y)) \nonumber \\
    &= \bmap(TR_{\phi^{-1}}(Y))(\alpha(p\cdot\phi^{-1})) \nonumber \\
    &= \left( {\Ad_\phi}\circ{\bmap}(Y) \right) (\phi\circ\alpha(p)) \\
    &= T\phi \circ \bmap(Y)(\alpha(p)) \nonumber \\
    &= T\phi \circ \bmap(\alpha)(Y) \nonumber
  \end{align}
  for all $p\in P$, $Y\in T^l_pP$, and $\phi\in \Bis(\fb(E))$, where we used the equivariance of $\alpha$ and $\bmap$, and \eqref{eq:action-bisections-frame-bundle-on-linear-vector-fields}.
  But $\phi$ is linear on the fibers so we have $I\circ T\phi|_{\ker Tq_E} = \phi \circ I$.
  With \eqref{eq:nabla-alpha-equivariant-1} and \eqref{eq:nabla-alpha-equivariant-2}, this shows that $\nabla\alpha$ is equivariant.
  Hence, $\nabla\alpha \in \Omega^1_{\text{hor}}(P,E)^{\fb(E)}$.
  
  The map $\nabla$ is $\L$-equivariant since both parts $\alpha\mapsto T^l\alpha$ and $\alpha\mapsto \bmap(\alpha)$ are.
  
  Finally, it satisfies the Leibniz rule since the first term does and the second term is $C^\infty(P)^{\fb(E)}$-linear.
  
  \medskip\noindent
  \textbf{Step 2.} ({From connections to connection forms}) 
  Let $\nabla: \Omega^0_{\text{hor}}(P,E)^{\fb(E)} \to \Omega^1_{\text{hor}}(P,E)^{\fb(E)}$ be an $\L$-equivariant map satisfying the Leibniz rule \eqref{eq:connection-eq-hor-forms-leibniz-rule}.
  Similarly to $I$ above, for each $e\in E_x$ there is a canonical map $V_e:E_x\to T_e^{\mathrm{vert}}E=\ker (Tq_E)_e$ such that $I\circ V_e = \Id$.
  Define a map $\bmap:T^lP\to \Tlin E$ by
  \begin{align}
  \label{eq:definition-omega-sharp-from-nabla}
    \bmap(Y)(\alpha(p)) = (T^l\alpha)(Y) - V_{\alpha(p)}(\nabla\alpha)(Y)
  \end{align}
  for all $\alpha\in \Omega^0_{\text{hor}}(P,E)^{\fb(E)}$, $Y\in T^l_pP$ and $p\in P$.
  
  By Proposition~\ref{prop:sections-equivariant-maps}, equivariant maps $P\to E$ correspond to sections of $P(E)\to X$.
  Since there exist sections through any point, this shows that, given $p\in P$, any $e\in E_{r(p)}\subset E$ can be written as $\alpha(p)$ for some $\alpha\in\Omega^0_{\text{hor}}(P,E)^{\fb(E)}$.
  Hence, \eqref{eq:definition-omega-sharp-from-nabla} defines $\bmap(Y):E_{r(p)}\to TE$ on all of $E_{r(p)}$.
  
  Let us show that for every $Y\in T^l_pP$ and $e\in E$, $\bmap(Y)(e)$ is well-defined by \eqref{eq:definition-omega-sharp-from-nabla}.
  Let $e\in E$ and let $\alpha,\alpha'\in \Omega^0_{\text{hor}}(P,E)^{\fb(E)}$ taking the value $e$ at $p$.
  Since $\alpha-\alpha'$ vanishes at $p$ and is equivariant, it vanishes on the whole $l$-fiber through $p$.
  Hence $T^l(\alpha-\alpha')(Y)=0$.
  On the other hand, the Leibniz rule \eqref{eq:connection-eq-hor-forms-leibniz-rule} shows that $\nabla(\alpha-\alpha')(Y)$ only depends on the derivative of $\alpha-\alpha'$ in the direction of $Y$, which also vanishes.
  As a result, $\bmap(Y)(e)$ is well-defined by \eqref{eq:definition-omega-sharp-from-nabla}.
  
  Let us show that $Y\mapsto \bmap(Y)$ defines an anchored map $T^lP\to\Tlin E$ over $r:P\to N$.
  We only need to see that $\bmap(Y)(e)\in T_eE$ and that $Tq(\bmap(Y)(e)) = Tr(Y)$ for all $e\in E$ and $Y\in T^l_pP$ with $r(p)=q(e)$, and these statements are easily checked.
  
  We now proceed to show that $\bmap$ satisfies the three conditions of Definition \ref{defn:connection-epb}.
  
  Let $D\in \Gamma(\Tlin E)$ and $p\in P$.
  We have 
  $$T^l\alpha(\hat D_p) 
  = \left.\frac{d}{dt} \alpha(p\cdot (\exp tD)^{-1}) \right|_{t=0} 
  = \left.\frac{d}{dt} \exp tD \cdot \alpha(p) \right|_{t=0} 
  = D(\alpha(p)).$$
  On the other hand, $V_{\alpha(p)}(\nabla\alpha)(\hat D_p) = 0$ since $\hat D_p\in \ker T\pi$.
  Hence, $\bmap(\hat D)=D$.
  
  Now, let $\phi\in \Bis(\fb(E))$. We have
  \begin{align*}
    \bmap ( TR_\phi Y)(\alpha(p\cdot\phi))
    &= T^l\alpha(TR_\phi Y) - V_{\alpha(p\cdot \phi)} (\nabla\alpha)(TR_\phi Y) \\
    &= T\phi^{-1} \circ T^l\alpha (Y) - V_{\phi\circ \alpha(p)} (\phi^{-1} \circ (\nabla\alpha)(Y)) \\
    &= T\phi^{-1} \circ \left( \bmap(Y)(\alpha(p)) \right),
  \end{align*}
  for all $\alpha\in \Omega^0_{\text{hor}}(P,E)^{\fb(E)}$, $Y\in T^l_pP$ and $p\in P$.
  Hence, $\left(R_\phi^*\bmap \right) (Y) = T\phi^{-1} \circ \bmap(Y) \circ \phi = \left( {\Ad_{\phi^{-1}}} \circ {\bmap} \right) (Y)$, i.e.\ $\bmap$ is $\fb(E)$-equivariant.
  
  Let $g\in\Bis(\L)$. 
  Recall that $\nabla$ is equivariant for the $\Bis(\L)$-action \eqref{eq:left-action-on-equivariant-horizontal-forms} on $\Omega^k_{\text{hor}}(P,E)^{\fb(E)}$.
  We have
  \begin{align*}
    \bmap ( TL_g Y)(\alpha(g\cdot p))
    &= T^l\alpha(TL_g Y) - V_{\alpha(g\cdot p)} (\nabla\alpha)(TL_g Y) \\
    &= T^l(g^{-1}\cdot\alpha) (Y) - V_{g^{-1}\cdot\alpha(p)} (\nabla(g^{-1}\cdot\alpha))(Y) \\
    &= \bmap(Y)((g^{-1}\cdot\alpha)(p)) \\
    &= \bmap(Y)(\alpha(g\cdot p))
  \end{align*}
  for all $\alpha\in \Omega^0_{\text{hor}}(P,E)^{\fb(E)}$, $Y\in T^l_pP$ and $p\in P$.
  Hence, $\bmap\circ TL_g = \bmap$.
  
  So $\bmap$ is a fibered connection form on $P$.
  
  \medskip\noindent
  \textbf{Step 3.} ({Bijection}) 
  Let us denote the associations \eqref{eq:definition-nabla-from-omega-sharp} and \eqref{eq:definition-omega-sharp-from-nabla} respectively by $\bmap \mapsto \nabla^\omega$ and $\nabla \mapsto \bmap_\nabla$.
  Recall that $I\circ V_e=\Id$ and $V_e\circ I|_{T_eE}=\Id$ for all $e\in E$.
  We have
  \begin{align*}
    (\nabla^{\omega_\nabla}\alpha)(Y)
    &= I\circ\left( T^l\alpha - \bmap_\nabla(\alpha)\right) (Y) \\
    &= I\circ\left( T^l\alpha - T^l\alpha + V_{\alpha(p)}\circ \nabla(\alpha)\right) (Y) \\
    &= (\nabla\alpha)(Y)
  \end{align*}
  and
  \begin{align*}
    \bmap_{\nabla^\bmap}(\alpha)(Y)
    &= \left( T^l\alpha(Y) - V_{\alpha(p)} \circ (\nabla^\omega\alpha) \right) (Y) \\
    &= \left( T^l\alpha - T^l\alpha + \bmap(\alpha) \right) (Y) \\
    &= \bmap(\alpha)(Y)
  \end{align*}
  for all $\alpha\in\Omega^0_{\text{hor}}(P,E)^{\fb(E)}$, $Y\in T^l_pP$ and $p\in P$, which completes the proof.
\end{proof}


\section{Global Atiyah class}
\label{sec:global-atiyah-class}

Throughout this section, we consider a \emph{Lie groupoid pair} $(\L,\A)$ over a manifold $M$, i.e.\ a Lie groupoid $ \L$ over $M$ and a closed wide Lie subgroupoid $ \A $ of $\L$. 
We define, for any transitive Lie groupoid $\U$, the Atiyah class of a generalized morphism from $\A$ to $\U$ with respect to such a pair.
We then specialize to the case where $\U$ is the frame groupoid $\fb(E)$ of a vector bundle over the same base as $\A$ and the generalized morphism is (the bundlization of) an $\A$-module structure on $E$.
Finally, we consider the special case of the $\A$-module $E=\ll/\aa$  (see section \ref{ssec:bigA-action-on-L-over-A}), which defines the Atiyah class of the Lie groupoid pair $(\L,\A)$.


\subsection{\texorpdfstring{$\A$-action on $\ll/\aa$}{A-action on L/A}}
\label{ssec:bigA-action-on-L-over-A}

Let $ \ll$ be the algebroid of $\L$ and $ \aa \subset \ll$ the Lie subalgebroid corresponding to the Lie subgroupoid $ \A$. Although there is no adjoint action of a Lie groupoid on its Lie algebroid, there is a canonical $\A$-action on the quotient vector bundle $\ll/\aa \to M$, as we proceed to show. 

\begin{prop}
  \label{prop:bigA-action-on-L-over-A}
  There is a unique left $\A$-module structure 
  \[  \A\times_M \ll/\aa \to \ll/\aa :(\generic,\beta) \mapsto \Ad_\generic\beta  \]
  on $\ll/\aa$ 
  which, when extended to bisections, satisfies
  \begin{equation}
  \label{eq:canonical-action-unique}
    \Ad_\bisection \overline l = \overline{\Ad_\bisection l}
  \end{equation}
  for all $\bisection\in\Bis(\A)$ and $l\in\Gamma(\ll)$.
  It satisfies, for all $\generic\in \A$ and $\beta\in(\ll/\aa)_{\so(\generic)}$, the relation
  \begin{equation}
  \label{eq:canonical_action}
    \Ad_\generic \beta = \ol{ \Add_{a^{-1}} \lambda } ,
  \end{equation}
  where $a\in T_\generic\A$ is any tangent vector and $\lambda\in\ll$ is any element such that $T\ta(u)=\rho(\lambda)$ and $\ol\lambda=\beta$. Above, $\lambda \mapsto \ol{\lambda} $ stands the natural projection $ \ll \to \ll/\aa$.
\end{prop}

\begin{proof}
  Let $(a,\lambda)$ and $(a',\lambda')$ be two pairs satisfying the conditions in the statement.
  Then $a-a'\in T_\generic\A$ and $\lambda-\lambda'\in\aa$, and
  \begin{align*}
    \Add_{a^{-1}}\lambda - \Add_{(a')^{-1}}\lambda'
    &= a\bullet\lambda\bullet 0_{\generic^{-1}} - a'\bullet\lambda'\bullet 0_{\generic^{-1}} \\
    &= (a-a') \bullet (\lambda-\lambda') \bullet 0_{\generic^{-1}}
  \end{align*}
  is an element in $\aa$ since all three factors are in $T\A$.
  This proves that $\Ad_\generic \beta$ is well-defined by \eqref{eq:canonical_action}.
  
  The fact that it defines an $\A$-module structure is straightforward, and Lemma~\ref{lem:adjoint-formula} shows that it is a solution of \eqref{eq:canonical-action-unique}. On the other hand, solutions to \eqref{eq:canonical-action-unique} are obviously unique, which concludes the proof.
\end{proof}

\begin{rmk}
  As can easily be checked, for Lie groups, the action above is simply induced from the adjoint action of $\A \subset \L$  on $\ll$.
\end{rmk}

\begin{rmk}
Alternatively, using the global structure of the Lie groupoid pair (and not just the $T\A$-action on $\ll$), there is geometrical picture describing the above action.
Indeed, the Lie groupoid $\A$ has a natural left action on the quotient space $\L/\A$, which is a manifold that fibers over $M$ through $\underline{\ta}$ as in \eqref{eq:t_bar}. 
By construction, this action preserves the unit manifold $M$ seen as a closed embedded submanifold of $\L/\A$, which automatically implies that the action by an element $\generic \in \A $ with source $x$ and target $y$ maps $ T_x \L/\A$ to $ T_y \L/\A$. 
Since the latter are isomorphic to $(\ll/\aa)_x$ and $(\ll/\aa)_y$ respectively, we obtain a linear map that coincides with the map $\Ad_\generic$ defined in 
\eqref{eq:canonical_action}.
\end{rmk}


\subsection[Atiyah class]{Atiyah class}
\label{ssec:atiyah-class}

Let $P$ be a generalized morphism from $\A$ to a Lie groupoid $\U$ over $N$ with Lie algebroid $\uu$.
Let us denote by $l$ and $r$ the left and right moment maps of $P$, respectively.

An \emph{$\ll$-$\uu$-connection over $P$} is an anchored map $\nabla:l^!\ll\to\uu$ over $r$, i.e.~a bundle map such that
\[
  \vcenter{\xymatrix{
    l^!\ll \ar[d]_{\rho=\pr_1} \ar[r]^{\nabla}  & \uu \ar[d]^{\rho} \\ 
    TP\ar[r]^{Tr} & TN 
  }}
\]
commutes.
It is said to \emph{extend the generalized morphism} if $\nabla|_{l^!\aa}=\Lie(\Phi_P)$, see \eqref{eq:gen-mor-induced-morphism}--\eqref{eq:diagram-generalized-morphism-pullback-and-morphism}.
It is said to be \emph{$T\U$-equivariant} if $\nabla \circ R_u = \Add_u \circ \nabla$ for all $u\in T\U$, where $R_u(X,(p,\lambda))=(X\cdot u,(q,\lambda))$ (here, $q$ is such that $X\cdot u\in T_qP$).

\begin{defn}
  \label{defn:compatible-connection}
  An $\ll$-$\uu$-connection $\nabla$ over $P$ is said to be \emph{$P$-compatible} if it commutes with the $T(l^!\A)$-actions, in the sense that the following condition holds
  \begin{enumerate}[label=(C)]
    \item \label{defn:compatible-connection-2-invariant}
      $\Add_{T\Phi_P(a)}\circ\nabla = \nabla \circ \Add_a$ for all $a\in T(l^!\A)$.
  \end{enumerate}
\end{defn}

Restricting \ref{defn:compatible-connection-2-invariant} to zero vectors $(0,(p,0))\in l^!\ll$ and to $a=(0,\alpha^{-1},X)$ with $\alpha\in\aa$ shows that $P$-compatible $\ll$-$\uu$-connections automatically extend the generalized morphism.
Restricting \ref{defn:compatible-connection-2-invariant} to $a=(X,0_1,X\cdot u)$ with $(X,u)\in TP\times_{TN} T\U$ shows that $P$-compatible $\ll$-$\uu$-connections automatically are $T\U$-equivariant.

A $T\U$-equivariant $\ll$-$\uu$-connection over $P$ is entirely determined by its restriction to a section of the left moment map, if such a section exists, i.e.~if $P$ is the bundlization of a morphism $\phi:\A\to\U$.
So a $T\U$-equivariant $\ll$-$\uu$-connection over the bundlization of a morphism is equivalent to an anchored map $\nabla:\ll\to\uu$, and the former exist if and only if the latter do.
For that reason, we make the following definition.
\begin{defn}
  \label{defn:phi-compatible-connection}
  Given a morphism $\varphi:\A\to\U$, an anchored map $\nabla$ from $\ll$ to $\uu$ is said to be \emph{$\varphi$-compatible} if it commutes with the $T\A$-actions, in the sense that the following condition holds
  \begin{enumerate}[label=(C')]
    \item \label{defn:phi-compatible-connection-2-invariant}
      $\Add_{T\varphi(a)}\circ\nabla = \nabla \circ \Add_a$ for all $a\in T\A$.
  \end{enumerate}
\end{defn}

We invite the reader to have the next example in mind.
\begin{ex}
\label{ex:module_case}
  For an $\A$-module $E$, take $\U = \fb (E)$ its frame groupoid, and $P$ the bundlization of the morphism $\A\to\fb(E)$ defining the $\A$-module structure. 
  Then $T\U$-equivariant $\ll$-$\uu$-connections over $P$ are equivalent to anchored maps $\ll\to\Tlin E$ or, after composing with the Lie algebroid isomorphism $\mathcal L^{-1}:\Tlin E\to\D(E)$ defined in \eqref{eq:isomorphism-derivative-endomorphisms-derivations}, to $\ll$-connections on $E$.
\end{ex}

In fact, when the Lie groupoid $\U$ is over the same base $M$ as $\A$ and $\L$ and $\varphi$ is a morphism covering the identity of $M$, then $\varphi$ sends bisections of $\A$ to bisections of $\U$ and $\Lie(\varphi)$ sends sections of $\aa$ to sections of $\uu$, which makes sense of the following proposition.

\begin{prop}
	Given a morphism $\varphi:\A\to\U$ between Lie groupoids over the same base, an anchored map $\nabla$ from $\ll$ to $\uu$ is $\varphi$-compatible if and only if 
  \begin{enumerate}[label=(C\arabic*)]
    \item \label{defn:compatible-connection-1-extend}
      it \emph{extends the morphism $\Lie(\varphi):\aa\to\uu$}, i.e.
      \[ {\nabla|_\aa}=\Lie(\varphi), \]
    \item \label{defn:compatible-connection-2-prime-invariant}
      it is \emph{equivariant}, in the sense that for all bisections $\bisection\in\Bis(\A)$,
		  \[
		    \vcenter{\xymatrix{
		      \Gamma(\ll) \ar[r]^{\nabla} \ar[d]_{\Ad_{\bisection}} & \Gamma(\uu) \ar[d]^{\Ad_{\varphi(\bisection)}} \\ 
		      \Gamma(\ll) \ar[r]^{\nabla} & \Gamma(\uu) 
		    }} .
		  \]
  \end{enumerate}
\end{prop}

\begin{proof}
	Definition \ref{defn:phi-compatible-connection-2-invariant} restricted to the zero vectors in $\ll$ and with $u=v^{-1}$ for $v\in\aa$ indeed implies condition \ref{defn:compatible-connection-1-extend}.
  Also, condition \ref{defn:compatible-connection-2-prime-invariant} follows from Lemma \ref{lem:adjoint-formula}.
	The other implication follows from arguments similar to those in the proof of Proposition \ref{prop:characterization-bundle-maps-of-connection-forms}.
\end{proof}

Assume that $\U$ is a \emph{transitive} Lie groupoid, and let $\uu_0=\ker \rho_\uu$ be the isotropy Lie subalgebroid of $\uu$.
Recall that $\uu_0$ is naturally a $\U$-module with action defined by
\[ \genericU\cdot u=\Add_{0_{\genericU^{-1}}}(u)=0_{\genericU}\bullet u\bullet 0_{\genericU^{-1}} \]
for all $\genericU\in\U$ and $u\in(\uu_0)_{\so(\genericU)}$. The generalized morphism $P$ therefore makes the associated vector bundle $P(\uu_0)$ an $\A$-module (see Definition~\ref{defn:associated-vector-bundle}).
There is also an $l^!\A$-module $r^*\uu_0$ which is isomorphic to $l^*P(\uu_0)$ as an $l^!\A$-module. 

\begin{ex} \label{ex:module_case2}
  The frame groupoid in Example \ref{ex:module_case} is transitive, and its isotropy Lie algebroid is $\uu_0=\End E$.
\end{ex}

\begin{prop}
\label{prop:atiyah-class-transitive-case}
Let $(\L,\A)$ be a Lie groupoid pair, and let $P$ be a generalized morphism from $\A$ to a transitive Lie groupoid $\U$.
\begin{enumerate}
  \item There exist $\ll$-$\uu$-connections over $P$ extending the generalized morphism.
  \item For any such $\ll$-$\uu$-connection $\nabla$, there is a smooth groupoid 1-cochain $R^\nabla\in C^1(l^!\A,l^*(\ll/\aa)^*\otimes r^*\uu_0)$ defined for all $\generic\in l^*\A$ and $\beta\in (l^*(\ll/\aa))_{\ta(\generic)}$ by
    \begin{align}
    \label{eq:definition-R-nabla}
      R^\nabla (\generic)(\beta) = \left( \Add_{T\Phi_P(a)^{-1}} \circ\nabla\circ\Add_{a} - \nabla \right)(\lambda)
    \end{align}
    where $a\in T_\generic(l^!\A)$ is any tangent vector and $\lambda\in l^!\ll$ any element such that $T\ta(a)=\rho(\lambda)$ and $\ol\lambda=\beta$.
  \item The 1-cochain $R^\nabla$ is closed and its cohomology class $\tilde\alpha_{(\L,\A),P}=[R^\nabla]$ is independent of the $\ll$-$\uu$-connection $\nabla$ over $P$ extending the generalized morphism.
  \item The class $\tilde\alpha_{(\L,\A),P}\in H^1(l^!\A,l^*(\ll/\aa)^*\otimes r^*\uu_0)$ is zero if and only if there exists a $P$-compatible $\ll $-$\uu$-connection over $P$.
\end{enumerate}
\end{prop}

\begin{proof}
  \begin{enumerate}
    \item Any transitive Lie algebroid $\uu\to N$ has Lie algebroid connections \cite[Corollary 5.2.7]{mackenzie_general_2005}, i.e.\ anchored maps from $TN$ to $\uu$ over the identity.
      Let $\nabla^0$ be such a connection.
      Let $\bb\subset l^!\ll$ be a vector subbundle supplementary to $l^!\aa$, and define a map $\nabla:l^!\ll\to\uu$ by $\nabla|_{l^!\aa}=\Lie(\Phi_P)$ and $\nabla|_\bb=\nabla^0\circ Tr\circ\rho_{l^!\ll}$.
      By construction, $\nabla$ is an anchored map which extends the morphism $\Lie(\Phi_P)$.
      
    \item Let $(a,\lambda)$ and $(a',\lambda')$ be two pairs satisfying the conditions in the statement.
      Then $a-a'\in T_\generic l^!\A$ and $\lambda-\lambda'\in l^!\aa$, and
      \begin{align*}
        &\left( \Add_{T\Phi_P(a)^{-1}} \circ\nabla\circ\Add_{a} - \nabla \right)(\lambda)
        - \left( \Add_{T\Phi_P(a')^{-1}} \circ\nabla\circ\Add_{a'} - \nabla \right)(\lambda') \\
        &\qquad\quad= T\Phi_P(a) \bullet \nabla \left( a^{-1} \bullet \lambda \bullet 0_\generic \right) \bullet 0_{\Phi_P(\generic)^{-1}} \\
        &\phantom{\qquad\quad={}} - T\Phi_P(a') \bullet \nabla \left( (a')^{-1} \bullet \lambda' \bullet 0_\generic \right) \bullet 0_{\Phi_P(\generic)^{-1}} \\
        &\phantom{\qquad\quad={}} - \nabla \left( \lambda-\lambda' \right) \\
        &\qquad\quad= T\Phi_P \left( a-a' \right) \bullet \nabla \left( (a-a')^{-1} \bullet  \left( \lambda-\lambda' \right) \bullet 0_\generic \right) \bullet 0_{\Phi_P(\generic)^{-1}} \\
        &\phantom{\qquad\quad={}} - \nabla \left( \lambda-\lambda' \right) \\
        &\qquad\quad= T\Phi_P \left( a-a' \right) \bullet \Lie(\Phi_P) \left( (a-a')^{-1} \bullet  \left( \lambda-\lambda' \right) \bullet 0_\generic \right) \bullet 0_{\Phi_P(\generic)^{-1}} \\
        &\phantom{\qquad\quad={}} - \Lie(\Phi_P) \left( \lambda-\lambda' \right) \\
        &\qquad\quad= 0.    
      \end{align*}
      This proves that the right-hand side of \eqref{eq:definition-R-nabla} only depends on $\generic$ and $\beta$, and justifies the definition of $R^\nabla(\generic)(\beta)$.
      
    \item First note that if $\nu_0\in\uu_0$ is written as $\nu_0=\nu-\nu'$ with $\nu,\nu'\in\uu$,
      then for any $u\in T_{\genericU^{-1}}\U$ such that $T\ta(u)=\rho(\nu)$ we have
      $\genericU \cdot \nu_0 = \Add_{0_{\genericU^{-1}}}(\nu_0) = \Add_{u-u}(\nu-\nu') = \Add_u(\nu)-\Add_u(\nu')$.
      
      Let $\generic,\generic'\in l^!\A$ be two composable elements and $\beta\in l^*(\ll/\aa)_{\ta(\generic)}$.
      Let $a\in T_\generic l^!\A$, $a'\in T_{\generic'} l^!\A$, and $\lambda\in\ll_{\ta(\generic)}$ be such that $(a\bullet a')^{-1}\bullet \lambda$ is defined and $\overline\lambda=\beta$.
      We have
      \begin{align*}
        \left( \generic\cdot R^\nabla(\generic') + R^\nabla(\generic) \right) (\beta)
        &= \Phi_P(\generic) \cdot \left( R^\nabla(\generic')(\Ad_{\generic^{-1}}\beta) \right) + R^\nabla(\generic)(\beta) \\
        &= \Add_{T\Phi_P(a)^{-1}} \circ \left( \Add_{T\Phi_P(a')^{-1}} \circ\nabla\circ\Add_{a'} \right) (\Add_{a}\lambda) \\ 
        &\phantom{={}}- \Add_{T\Phi_P(a)^{-1}} \circ \nabla (\Add_{a}\lambda) \\ 
        &\phantom{={}}+ \left(\Add_{T\Phi_P(a)^{-1}} \circ\nabla\circ\Add_{a}\right) (\lambda) - \nabla(\lambda) \\
        &= R^\nabla(\generic\cdot\generic') (\beta),
      \end{align*}
      which shows that the cocycle identity \eqref{eq:1cocycle} is satisfied.
      
      Let now $\nabla,\nabla'$ be two $\ll$-$\uu$-connections over $P$ extending the morphism $\Lie(\Phi_P)$.
      We have
      \begin{align*}
        \left( R^\nabla(\generic) - R^{\nabla'}(\generic) \right) (\beta)
        &= \left(\Add_{T\Phi_P(a)^{-1}} \circ\nabla\circ\Add_{a}\right) (\lambda) - \nabla(\lambda) \\
        &\phantom{={}}- \left(\Add_{T\Phi_P(a)^{-1}} \circ\nabla'\circ\Add_{a}\right) (\lambda) - \nabla'(\lambda) \\
        &= \Add_{0_{\Phi_P(\generic)^{-1}}} \circ(\nabla-\nabla')\circ\Add_{a}(\lambda)
        - (\nabla-\nabla')(\lambda) \\
        &= (\generic \cdot \mu - \mu)(\lambda)
      \end{align*}
      where $\mu$ is $\nabla-\nabla'$ seen as an element of $\Gamma(l^*(\ll/\aa)^*\otimes r^*\uu_0)$.
      Hence, $\tilde \alpha_{(\L,\A),P}=[R^\nabla]$ is independent of the chosen connection $\nabla$.
      
    \item If $\nabla$ is a $P$-compatible $\ll$-$\uu$-connection over $P$, then $R^\nabla$ vanishes and $\tilde\alpha_{(\L,\A),P}$ as well.
    Conversely, assume that $\tilde\alpha_{(\L,\A),P}=0$. Then for any $\ll$-$\uu$-connection over $P$ extending the morphism $\Lie(\Phi_P)$, there exists $\mu\in \Gamma(l^*(\ll/\aa)^*\otimes r^*\uu_0)$ such that $R^\nabla(\generic)=\generic\cdot\mu-\mu$ for all $\generic\in l^!\A$.
    Let $\overline \mu:l^!\ll\to\uu_0$ be the corresponding bundle map vanishing on $l^!\aa$, and let $\nabla'=\nabla-\overline\mu$.
    We get $0=R^\nabla(\generic)-(\generic\cdot\mu-\mu)=R^{\nabla'}$, proving the claim.
  \end{enumerate}
\end{proof}

Since the left moment map is a surjective submersion, it induces an isomorphism in cohomology
\[
	l^*: H^1(\A,(\ll/\aa)^* \otimes P(\uu_0)) \to H^1(l^!\A,l^*(\ll/\aa)^* \otimes l^*P(\uu_0)) .
\]
Using the isomorphism of $l^!\A$-modules $r^*\uu_0 \cong l^*P(\uu_0)$, we can now define the Atiyah class.

\begin{defn}
  \label{defn:atiyah-class-transitive-case}
	The \emph{Atiyah class of the generalized morphism $P$ with respect to the Lie groupoid pair $(\L,\A)$} is the class
	\[ \alpha_{(\L,\A),P} \in H^1(\A,(\ll/\aa)^* \otimes {P(\uu_0)}) \] 
	defined by $\alpha_{(\L,\A),P} = (l^*)^{-1}\tilde\alpha_{(\L,\A),P}$ where $\tilde\alpha_{(\L,\A),P}$ is defined in Proposition~\ref{prop:atiyah-class-transitive-case}.
\end{defn}

When $P$ is the bundlization of a morphism $\varphi$, the above Atiyah class may be directly defined using anchored maps from $\ll$ to $\uu$.

The main example is when $\U=\fb(E)$ is the frame groupoid of a vector bundle $E\to M$, and $\varphi:\A\to\fb(E)$ is an $\A$-module structure on $E$.
The Atiyah class, written $\alpha_{(\L,\A),E}$, is an element of $H^1(\A,(\ll/\aa)^*\otimes \End E)$.
Recall the definition of $\ll$-connections on a vector bundle from the end of Section \ref{ssec:lie-algebroids}.

\begin{cor}
\label{cor:atiyah-class-A-modules}
  Let $(\L,\A)$ be a Lie groupoid pair and  $E$ an $\A$-module.
  \begin{enumerate}
    \item There exist $\ll$-connections on $E$ extending the $\aa$-action.
    
    \item For any such $\ll$-connection $\nabla$, there is a smooth groupoid 1-cochain $R^\nabla\in C^1(\A,(\ll/\aa)^*\otimes \End E)$ defined by
      \begin{align}
      \label{eq:definition-R-nabla-particular}
        R^\nabla(\gamma)(l,e) =\bisection \star \nabla_{\Ad_{\bisection^{-1}}  l} \left( \bisection^{-1} \star e \right) - \nabla_l e  
      \end{align}
	    for all bisections $\bisection \in \Bis (\A)$ and all sections $l\in\Gamma(\ll)$ and $e\in\Gamma(E)$.
	    
    \item The 1-cochain $R^\nabla \in C^1(\A,(\ll/\aa)^*\otimes \End E) $ is closed and its cohomology class $\alpha_{(\L,\A),E}=[R^\nabla]$ is independent of the choice of $\nabla$.
    
    \item The class $\alpha_{(\L,\A),E}$ is zero if and only if there exists an $\A$-compatible $\ll $-connection on $E$.
  \end{enumerate}
\end{cor}

Here, we used the following definition.

\begin{defn}
\label{def:def-A-compatible}
  An $\ll$-connection extending the $\aa$-action is said to be \emph{$\A$-com\-pa\-ti\-ble} if the 1-cocycle $R^\nabla$ defined in \eqref{eq:definition-R-nabla-particular} vanishes.
\end{defn}

In the setting of Corollary~\ref{cor:atiyah-class-A-modules}, the class 
\[ \alpha_{(\L,\A),E}\in H^1(\A,(\ll/\aa)^* \otimes \End E) \] 
is called the \emph{Atiyah class of the $\A$-module $E$ with respect to the Lie groupoid pair $(\L,\A)$}.

The case of the $\A$ module $\ll/\aa$ of Section \ref{ssec:bigA-action-on-L-over-A} is of special interest.

\begin{defn}
\label{def:Atiyah-class-of-a-Lie-groupoid-pair}
The \emph{Atiyah class of the Lie groupoid pair $(\L,\A)$} is the class 
\[ \alpha_{(\L,\A)}\in H^1(\A,(\ll/\aa)^* \otimes (\ll/\aa)^* \otimes \ll/\aa) \] 
defined by $\alpha_{(\L,\A)}=\alpha_{(\L,\A),\ll/\aa}$.
\end{defn}

We relate our global Atiyah class to the Atiyah class  of the $\aa$-module $E$ with respect to $(\ll,\aa)$, the construction of which we briefly recall, using \cite{chen_atiyah_2016} as a guideline. 
Let $(\ll,\aa)$ be a Lie algebroid pair and let $E$ be an $\aa$-module.
Given an $\ll$-connection $\nabla$ on $E$ extending the flat $\aa$-connection $\nabla^\aa:\aa\to\D(E)$ that defines the module structure, the formula
\begin{align}
\label{eq:atiyahcocycle1}
  \atiyah^\nabla (a) (l, e) = \nabla_a \nabla_l e - \nabla_l \nabla_a e - \nabla_{[a,l]} e
\end{align}
with $l\in\Gamma(\ll)$, $ a \in \Gamma(\aa)$ and $e\in\Gamma(E)$ defines a Lie algebroid
1-cochain $\atiyahTensor^\nabla \in C^1(\aa,(\ll/\aa)^*\otimes \End E)$ as follows.
For all $\alpha \in \aa_x$, $\beta \in (\ll/\aa)_x $ and $\epsilon \in E_x$, set
 \begin{align*}
  \atiyahTensor^\nabla (\alpha)(\beta,\epsilon) = \left. \atiyah^\nabla (a)(l,e) \right|_x
\end{align*}
with $l$ any section of $\ll$	such that $\ol{l_x}=\beta$ and $a,e$ any sections through $\alpha$ and $ \epsilon$ respectively. 
We say that $ \nabla$ is \emph{$\aa$-compatible} when $\atiyahTensor^\nabla=0$.
The 1-cochain  $\atiyahTensor^\nabla$ is closed and its cohomology class $\atiyahTorsionClass_{(\ll,\aa),E}=[\atiyahTensor^\nabla]$ is independent of the $\ll$-connection $\nabla $ on $E$ extending the natural flat $\aa$-connection. 
The class $\atiyahTorsionClass_{(\ll,\aa),E}$ is zero if and only if there exists an $\aa$-compatible $\ll $-connection on $E$ extending the natural $ \aa$-action, defining therefore a class 
\[ \atiyahLieAlgebroid_{(\ll,\aa),E} \in H^1(\aa,(\ll/\aa)^* \otimes {\End E}) , \] 
called the \emph{Atiyah class of the $\aa$-module $E$ with respect to the Lie algebroid pair $(\ll,\aa)$}.
The next proposition relates both classes, we refer to the work of Marius Crainic \cite{crainic_differentiable_2003} for a definition of the van Est functor.

\begin{prop}
\label{prop:vanest}
  Let $ (\L,\A)$ be a Lie groupoid pair, and $ (\ll,\aa)$ be its infinitesimal Lie algebroid pair. 
  Let $ E$ be an $\A$-module. The van Est functor $H^1(\A,(\ll/\aa)^* \otimes {\End E})  \to H^1(\aa,(\ll/\aa)^* \otimes {\End E}) $ maps the Atiyah class of the $\A$-module $E$ with respect to $(\L,\A)$
  to minus the Atiyah class of the $\aa$-module $E$ with respect to $(\ll,\aa)$.
\end{prop}
\begin{proof}
  Recall that the van Est functor simply assigns to an $\A$-cocycle $\Phi$ valued in an $\A$-module $F$
  the $F$-valued $1$-form $\phi : \alpha \mapsto T \Phi(\alpha)$, with the understanding that $\alpha \in \aa_x$
  is seen as an element of $T_{x} \A$. An important feature of this assignment is that
  for any section $a$ of $ \aa$,
  $ \phi(a) = \left. \frac{\diff}{\diff t} \Phi( \exp (ta)) \right|_{t=0}$.
	Applying this construction to the Atiyah cocycle $R^\nabla$, we are left with the task of taking
	the derivative at $t=0$ of the quantity
  \[  {\exp(ta)} \star \nabla_{\Ad_{\exp(ta)^{-1}}  l} \left( {\exp(ta)}^{-1} \star e \right) - \nabla_l e   \]
  for arbitrary sections $e\in\Gamma(E)$ and $l\in\Gamma(L)$, which yields precisely minus the expression given in \eqref{eq:atiyahcocycle1}  in view of \eqref{eq:relation-algebroid-groupoid}--\eqref{expon-lie-bracket} and completes the proof.
\end{proof}

According to Theorem 3 in \cite{crainic_differentiable_2003}, the van Est map is an isomorphism in degree $\leq n$ and is injective in degree $n+1$ provided that the fibers of the source map of the Lie groupoid are $n$-connected. As an immediate corollary of the previous proposition, we have the following result: 

\begin{cor}
\label{cor:infinitesimalAtiyahVanishes-implies-globalAtiyahVanishes}
  Let $(\L,\A)$ be a Lie groupoid pair with $\A$ source-connected,
  let $(\ll,\aa)$ be its infinitesimal Lie algebroid pair, and $E$ an $\A$-module. 
  Then the Atiyah class of the $\aa$-module $E$ with respect to $(\ll,\aa)$ vanishes if and only if the Atiyah class of the $\A$-module $E$ with respect to $(\L,\A)$ vanishes.	
\end{cor}


\subsection{Morita invariance}
In this section, we prove that our Atiyah classes are invariant under Morita equivalences.

Let us first  recall \cite{crainic_differentiable_2003} how Morita equivalent Lie groupoids have equivalent categories of representations.
Let $\A'$ and $\A$ be two Lie groupoids, and let $Q$ be a generalized morphism from $\A'$ to $\A$.
Given an $\A$-module $E$, $Q$ and $E$ define the associated $\A'$-module
$ Q(E) = \frac{Q\times_M E}{\A} $
(see Definition \ref{defn:associated-vector-bundle}).
Given an $\A$-module map $E\to F$, $Q$ naturally induces an $\A'$-module map $\amod Q(E)\to \amod Q(F)$, and this defines a functor $\amod Q$ from the category of $\A$-modules to that of $\A'$-modules.
When $Q$ is a Morita morphism, the functor $\amod Q$ becomes an equivalence of categories (see \eqref{eq:P-Pop}--\eqref{eq:Pop-P}).

Recall \cite{crainic_differentiable_2003} also that a Morita morphism $Q$ induces an isomorphism, denoted by $Q^*$, from the Lie groupoid cohomology of $\A$ valued in $E$ to the  Lie groupoid cohomology of $\A'$ valued in $\amod Q(E)$.

From now on, let us fix two Lie groupoid pairs, $(\L',\A')$ over $M'$ and $(\L,\A)$ over $M$.
Recall the following notion from \cite{laurent-gengoux_equivariant_2009}.
A \emph{Morita morphism from $(\L',\A')$ to $(\L,\A)$} is a pair $(\tilde Q,Q)$ with
\begin{enumerate}
  \item $\tilde Q$ a Morita morphism from $\L'$ to $\L$, 
  \item $Q$ a Morita morphism from $\A'$ to $\A$,
  \item $i: Q \hookrightarrow \tilde Q$ an inclusion map that makes $Q$ an immersed submanifold of $\tilde Q$,
\end{enumerate}
such that
\begin{enumerate}
  \item the following diagram commutes:
  \begin{equation}
  \label{eq:ileftandright}
    \vcenter{\xymatrix{
         & \tilde Q\ar[ld]_{\tilde l}\ar[rd]^{\tilde r}          &   \\ 
      M' & Q \ar[l]^l \ar[u]^i \ar[r]_r & M
    }} ,
  \end{equation}
  \item the inclusion map $i$ is equivariant with respect to the left $\A'$-action and the right $\A$-action.
\end{enumerate}

Note that such a Morita morphism may be expressed with bundlizations of morphisms as
\begin{equation}
\label{eq:decomp-morita-morphism-of-pairs}
  (\tilde Q, Q) \cong ((P_{\tilde l})^{-1}, (P_l)^{-1}) \circ (P_{\Phi_{\tilde Q}}, P_{\Phi_{Q}}) .
\end{equation}

\begin{thm}
\label{thm:Morita-equivalence}
  Let $(\tilde Q,Q) $ be a Morita morphism from $(\L',\A')$ to $(\L,\A)$.
  Let $P$ be a generalized morphism from $\A$ to a transitive Lie groupoid $\U$ and let $P'=Q\circ P$ be its composition with $Q$.
  Then
  \begin{enumerate}
    \item \label{thm:Morita-equivalence-1}
      the functor $\amod{Q}$ maps the $\A$-module $(\ll/\aa)^* \otimes P(\uu_0)$ to the $\A'$-module $(\ll'/\aa')^* \otimes P'(\uu_0)$,
    \item \label{thm:Morita-equivalence-2}
      the Lie groupoid cohomology isomorphism 
      \[ Q^*: H^1(\A, (\ll/\aa)^* \otimes P(\uu_0)) \to H^1 (\A',(\ll'/\aa')^* \otimes P'(\uu_0)) \]
      associated to the Morita morphism $Q$ maps the Atiyah class of $P$ with respect to $(\L,\A)$ to the Atiyah class of $P'$ with respect to $(\L',\A')$.
  \end{enumerate}
\end{thm}

We start with a lemma. 

\begin{lem}\label{lem:LsurA}
  Let $(\tilde Q,Q) $ be a Morita morphism from $(\L',\A')$ to $(\L,\A)$. 
  Then \[ \amod Q(\ll/\aa)\cong\ll'/\aa' \] as $\A'$-modules.
\end{lem}

\begin{proof}
  Since $\tilde Q$ is a Morita morphism from $\L'$ to $\L$, there is a natural base-preserving isomorphism \eqref{eq:isomorphism-of-pullback-groupoids-from-Morita-equivalence} of Lie groupoids $\tilde\Phi_{\tilde Q}:\tilde l^\pb\L'\to \tilde r^\pb\L$.
  Pulling back by the map $i:Q\to \tilde Q$ and using $\tilde l\circ i=l$ and $\tilde r\circ i=r$ yields an isomorphism $l^\pb\L'\to r^\pb\L$.
  Applying the Lie functor, we obtain a base-preserving Lie algebroid isomorphism
  \[  l^\pb \ll' \to r^\pb \ll .  \]
  Since $Q$ is a Morita morphism from $\A'$ to $\A$, similarly, we obtain a base-preserving Lie algebroid isomorphism
  \[  l^\pb \aa' \to r^\pb \aa .  \]
  The inclusion map $i$ of $Q$ into $\tilde Q$ being compatible with the bibundle structures (see \eqref{eq:ileftandright}), there is a commutative diagram of Lie algebroids over $Q$
  \[
    \vcenter{\xymatrix{
      l^\pb\aa' \ar[r] \ar@{^{(}->}[d] \vphantom{|} & r^\pb\aa \ar@{^{(}->}[d] \vphantom{|} \\
      l^\pb\ll' \ar[r] & r^\pb\ll
    }}
  \]
  where the horizontal arrows are isomorphisms and the vertical arrows are inclusions.
  As a result, we get an isomorphism 
  \begin{equation}
  \label{eq:isomorphism-pullbacks-LprimeAprime-LA}
    l^\pb\ll'/l^\pb\aa' \to r^\pb\ll/r^\pb\aa
  \end{equation}
  of vector bundles over $Q$. By construction, this isomorphism intertwines the $l^\pb \A$-module structure and the $r^\pb \A$-module structure.

  Since $l$ is a surjective submersion, the projection on the second component $l^\pb\ll'\to l^*\ll'$ induces an isomorphism
  \begin{equation}
  \label{eq:isomorphism-quotient-of-pullback-bundle-and-algebroid}
    l^\pb\ll'/l^\pb\aa' \to l^*(\ll'/\aa') . 
  \end{equation}
  Let us prove that this isomorphism is in fact an $l^\pb \A'$-module isomorphism. Notice that bisections can be pulled-back---for $\Sigma'$ a local bisection of $\A'$, a bisection of $ l^\pb \A'$ is defined by
  \begin{equation}
  \label{eq:pull_back_bisection}
    l^* \Sigma' = \left\{ (q,\generic',q') \mid \generic' \in \Sigma', l(q)=\ta (\generic'),  l(q')=\so (\generic') \right\} .
  \end{equation}
  Equation (\ref{eq:canonical-action-unique}), applied to bisections and sections of the previous form, 
  gives immediately that the natural projection from $l^\pb \A'$ to $ \A' $ intertwines the module structures on 
  $l^\pb\ll'/l^\pb\aa'$  and  $\ll'/\aa'$. This implies the result. The same procedure applies with $\A$ and $\ll/\aa$,
  and yields an isomorphism of modules
  \begin{equation}
  \label{eq:isomorphism-pullbacks-LprimeAprime-LA2}
    l^*(\ll'/\aa' )\to r^*(\ll/\aa) .
  \end{equation}
		
  Now, in view of Definition \ref{defn:associated-vector-bundle}, $\frac{r^*(\ll/\aa)}{\A} = \amod Q(\ll/\aa)$. Since $\A$ acts freely on the leaves of $l:Q\to M'$, we of course have $ \ll'/\aa' \simeq \frac{l^*(\ll'/\aa')}{\A}$.
  Taking the quotient by the right action of $\A$ on both sides of \eqref{eq:isomorphism-pullbacks-LprimeAprime-LA2} finally yields
  \begin{equation*}
    \ll'/\aa' 
    \simeq \frac{l^*(\ll'/\aa')}{\A}
    \simeq \frac{r^*(\ll/\aa)}{\A}
    = \amod Q(\ll/\aa) .
  \end{equation*}
  This completes the proof.
\end{proof}


\begin{proof}[Proof of Theorem \ref{thm:Morita-equivalence}]
  We start with the first item. It is easy to check that the functor $\amod{Q}$ is compatible with duality and tensor products. Hence, Item \ref{thm:Morita-equivalence-1} follows from Lemma \ref{lem:LsurA}.

  We now turn to Item \ref{thm:Morita-equivalence-2}. 
	Since $Q\mapsto Q^*$ is a functor and sends Morita morphisms to isomorphisms in cohomology, and since we have the decomposition \eqref{eq:decomp-morita-morphism-of-pairs}, it suffices to prove the result for a Morita morphism of the kind $(P_{\tilde\varphi}, P_\varphi)$ with $\tilde\varphi:\L'\to \L$ and $\varphi:\A'\to \A$ morphisms such that $\tilde \varphi|_{\A'} = \varphi$. 
	In that case, we need to show that the top arrow in the diagram
  \begin{equation}
  \label{eq:forMorita0}
    \vcenter{\xymatrix{
      \ar[d] H^1(l^\pb \A, l^*(\ll/\aa)^* \otimes l^*P(\uu_0)) \ar[r]^{(l^!\varphi)^*} 
      &  H^1(l^\pb \A', l^*(\ll'/\aa')^* \otimes l^*P'(\uu_0)) \ar[d] \\
      H^1(\A, (\ll/\aa)^* \otimes P(\uu_0))\ar[r]^{\varphi^*} 
      &  H^1 (\A',(\ll'/\aa')^* \otimes P'(\uu_0))
    }} 
  \end{equation}
	sends $\tilde\alpha_{(\L,\A),P}$ to $\tilde\alpha_{(\L',\A'),P'}$.
	Here, $P'=P_\varphi\circ P \cong M'\times_M P$ and we have $\Phi_{P'} = \Phi_P \circ (l^!\varphi)$.
	Now, if $\nabla:l^!\ll\to\uu$ is an $\ll$-$\uu$-connection over $P$ extending the generalized morphism, a direct computation shows that
	\[  (l^!\varphi)^*R^\nabla = R^{\nabla \circ \Lie(l^!\tilde\varphi)} ,  \]
	which yields the result.
\end{proof}


\section{Connections on homogeneous spaces}

\label{sec:main-results}

Throughout this section, $(\L,\A)$ is a Lie groupoid pair integrating a Lie algebroid pair $(\ll,\aa)$ over $M$, and $\U$ is a Lie groupoid  over $N$.


\subsection{Equivariant principal bundles}

\begin{thm}
  \label{thm:general}
  Let $Q$ be an $\L$-equivariant principal $\U$-bundle over the homogeneous space $X=\L/\A$.
  Let $P$ be a generalized morphism, given by Proposition~\ref{prop:exhaust}, such that $Q\cong{}_{\L}\L_\A \circ P$.
  There is a bijection between the following affine spaces:
  \begin{enumerate}
    \item the fibered connection forms on $Q$ (Definition~\ref{defn:connection-epb}),
    \item the $P$-compatible $\ll$-$\uu$-connections (Definition~\ref{defn:compatible-connection}).
  \end{enumerate}
\end{thm}  

\begin{proof}
  Without loss of generality, we may identify $Q$ with ${}_{\L}\L_\A \circ P$.
  
  Let $\nabla:l^!\ll\to\uu$ be a $P$-compatible $\ll$-$\uu$-connection.
  The elements of $Q$ are equivalence classes $[(\generic,p)]$ of pairs $(\generic,p)\in\L\times_M^{\so,l} P$.
  We use the same notation for tangent vectors $[(v,X)]\in T_{[(\generic,p)]} Q$.
  There is a map 
  \[ \tilde\tau:T^\ta\L \times_{TM} TP \to l^!\ll : (v,X) \mapsto (X,(p,\tau(v)))  \]
  where $X$ is in $T_pP$ (essentially, $\tilde\tau=\tau_{l^!\L}$).
  Define then $\bmap:T^lQ\to \uu$ by
  \begin{align}
  \label{eq:from-compatible-connection-to-connection-form}
    \bmap([(v,X)]) = \nabla\circ\tilde\tau(v,X). 
  \end{align}
  The equivariance properties of $\nabla$ and $\tau$ and the identity $\Phi_P(p,\generic^{-1},\generic\cdot p) = 1_{r(p)}$ for all compatible $\generic\in\A$ and $p\in P$ imply that $\nabla \circ \tilde\tau$ is $\A$-basic (i.e.\ $T\A$-invariant), so $\bmap$ is well-defined.
  Now \ref{prop:characterization-bundle-maps-of-connection-forms-2} directly follows from the equivariance of $\nabla$ and the identity $\Phi_P(p,1,p\cdot\genericU)=\genericU$ for all compatible $p\in P$ and $\genericU\in \U$, and \ref{prop:characterization-bundle-maps-of-connection-forms-3} follows from the fact that $\tau$ is left-invariant.
  
  Conversely, let $\bmap:T^lQ\to\uu$ be a fibered connection form on $Q$.
  Define, for all $(X,(p,\lambda))\in l^!\ll$,
  \begin{align}
  \label{eq:from-connection-form-to-compatible-connection}
    \nabla(X,(p,\lambda)) = \bmap([(\lambda^{-1},X)]) .
  \end{align}
  For all $(X,(p,\lambda))\in l^!\ll$ and $(X,a,X')\in T(l^!\A)$, using successively \ref{prop:characterization-bundle-maps-of-connection-forms-2}, the defining property of $\Phi_P$, and \ref{prop:characterization-bundle-maps-of-connection-forms-3} we get
  \begin{align*}
    \Add_{T\Phi_P(X,a,X')}\circ \nabla(X,(p,\lambda))
    &= \bmap([(\lambda^{-1},X\cdot T\Phi_P(X,a,X'))]) \\
    &= \bmap([(\lambda^{-1},a\cdot X')]) \\
    &= \bmap(0_\gamma\cdot [(\lambda^{-1}\bullet u, X')]) \\
    &= \bmap([((\Add_a\lambda)^{-1},X')]) \\
    &= \nabla\circ\Add_{(X,a,X')} (X,(p,\lambda)) .
  \end{align*}
  This proves \ref{defn:compatible-connection-2-invariant}.
  
  The two associations $\bmap\mapsto \nabla$ and $\nabla\mapsto \bmap$ are obvious inverses of each other, which concludes the proof.
\end{proof}


\subsection{Associated vector bundles}

Consider now the case where $\U=\fb(E)$ for some vector bundle $E\to M$ and $P$ is the bundlization of a morphism $\varphi:\A\to \fb(E)$, i.e.~of an $\A$-module structure on $E$.
We thus have an $\L$-equivariant principal $\fb(E)$-bundle $Q=\frac{\L\times_M \fb(E)}{\A}$ (see Examples~\ref{ex:LLA-morphism-phi} and \ref{ex:avecE} for more details).
Composing Theorem \ref{thm:general} and Theorem \ref{thm:connection-forms-connections-on-associated-bundles} yields:

\begin{thm}
  \label{thm:L-conn-on-E--conn-on-assoc-bundle}
  Let $(\L, \A)$ be a Lie groupoid pair over $M$,
  and let $E$ be an $\A$-module.
  There is a bijective correspondence between
  \begin{enumerate}[label=(\arabic*)]
    \item $\A$-compatible $\ll$-connections on $E$, and
    \item $\L$-invariant fibered connections on the associated vector bundle $ \frac{\L \times_M E}{ \A} \to \L/\A \to M$ defined in Example \ref{ex:avecE}.
  \end{enumerate}
\end{thm}

We thus arrive at the main theorem of this section.
\begin{thm}
  \label{thm:L-conn-on-LoverA--conn-on-tgt-bundle}
  Let $(\L, \A)$ be a Lie groupoid pair over $M$.
  There is a bijective correspondence between
  \begin{enumerate}[label=(\arabic*)]
    \item $\A$-compatible $\ll$-connections on $\ll/\aa$, and
    \item $\L$-invariant fibrewise affine connections on $\L/\A \to M$.
  \end{enumerate}
\end{thm}

\begin{proof}
  There is a natural $\L$-equivariant isomorphism $\theta : \frac{\L\times_M \ll/\aa}{\A} \to T^{\ol \ta}(\L/\A)$ of vector bundles over $\L/\A$, defined as follows.
  The projection $\pi:\L\to \L/\A$ induces an $\L$-equivariant map $\pi':T^\ta\L \to T^{\ol\ta}(\L/\A)$.
  Now the composition of the ($\L$-equivariant) map $\L\times_M \ll \to T^\ta \L : (\generic,\lambda) \mapsto 0_\generic \bullet \lambda^{-1}$ with $\pi'$ vanishes on $\L\times_M \aa$, hence descends to a map $\L\times_M \ll/\aa \to T^{\ol\ta}(\L/\A)$.
  The latter is $\A$-invariant, so that it descends to a map $\theta$ as above.
  
  Through $\theta$, $\L$-invariant fibered connections on the vector bundle $\frac{\L \times_M \ll/\aa}{\A} \to \L/\A \to M$ become $\L$-invariant fibered connections on the vector bundle $T^{\ol\ta}(\L/\A) \to \L/\A \to M$, so that the result follows by applying Theorem~\ref{thm:L-conn-on-E--conn-on-assoc-bundle} to $E=\ll/\aa$.
\end{proof}

\begin{rmk}
\label{rmk:simple-construction-of-main-correspondence}
  The correspondence of Theorem~\ref{thm:L-conn-on-E--conn-on-assoc-bundle} is by construction obtained by composing the correspondences of Theorems \ref{thm:general} and \ref{thm:connection-forms-connections-on-associated-bundles}.
  We give here a shorter description of it. Let $\nabla^E $ be an $\ll$-connection on $E$.
  There exists a unique fibered connection $\nabla^{\so^* E}$ on $ \so^* E \to \L \xrightarrow{\ta} M$ such that
  \begin{equation*}
    \nabla_{L(l)}^{\so^* E} \so^*  e = \so^* \nabla^E_l e ,
  \end{equation*}  
  where $L(l)$ is the left-invariant vector field associated to $l\in \ll$.
  The vector bundle $\so^* E \to \L$ is canonically isomorphic to the pullback of $\frac{\L \times E}{ \A} \to \L/\A$ through $\L \to \L/\A$.
  If $\nabla^E$ is $\varphi$-compatible, then one can see that the fibered connection $ \nabla^{\so^* E}$ is in fact the pullback of some (unique) fibered connection $\nabla^{\frac{\L \times E}{\A}}$ on the vector bundle $\frac{\L \times E}{\A} \to \L/\A$. 
  Spelling out the construction, one can see that
  \begin{equation*}
    \nabla^E \mapsto \nabla^{\frac{\L \times E}{ \A}}
  \end{equation*}
  is the correspondence that we obtained. 
\end{rmk}

We now derive an immediate consequence of Theorem \ref{thm:L-conn-on-E--conn-on-assoc-bundle} and the last item of Corollary \ref{cor:atiyah-class-A-modules}.

\begin{cor}
\label{cor:L-conn}
  Let $(\L, \A)$ be a Lie groupoid pair over $M$, and let $E$ be an $\A$-module.
  The Atiyah class of $E$ with respect to $(\L,\A)$ vanishes if and only if there exist $\L$-invariant fibered connections on the associated vector bundle $\frac{\L \times_M E}{\A} \to \L/\A \to M$. 
\end{cor}

In the case where $E=\ll/\aa$, using Theorem \ref{thm:L-conn-on-LoverA--conn-on-tgt-bundle} instead of Theorem \ref{thm:L-conn-on-E--conn-on-assoc-bundle} yields:

\begin{cor}
\label{cor:atiyah=0GivesConnections}
  Let $(\L, \A)$ be a Lie groupoid pair over $M$.
  The Atiyah class of the Lie groupoid pair $(\L,\A)$ vanishes if and only if there exist $\L$-invariant fibrewise affine connections on $\L/\A \to M$. 
\end{cor}

In view of Corollary \ref{cor:infinitesimalAtiyahVanishes-implies-globalAtiyahVanishes}, the following result can be derived: 

\begin{cor}
\label{cor:L-conn2}
	Let $(\L, \A)$ be a Lie groupoid pair over $M$ with $A$ source-connected, and let $(\ll,\aa)$ be the corresponding infinitesimal Lie algebroid pair.
  \begin{enumerate}
    \item The Atiyah class of an $\A$-module $E$ with respect to $(\ll,\aa)$ vanishes if and only if there exist $\L$-invariant fibered connections on the associated vector bundle  $\frac{\L \times_M E}{\A} \to \L/\A \to M$. 
    \item The Atiyah class of the Lie algebroid pair $(\ll,\aa)$ vanishes if and only if there exist $\L$-invariant fibrewise affine connections on $\L/\A \to M$.
	\end{enumerate}
\end{cor}

\subsection{Reductive homogeneous spaces}
\label{sec:reductive-homogeneous-spaces}

Consider a Lie group $G$ seen as a $G$-equivariant principal $H$-bundle over $G/H$ for some closed Lie subgroup $H$. 
It is well-known (see e.g.\ \cite[II, Theorem 11.1]{kobayashi_1963_foundations}) that the principal $H$-bundle $G$ admits a $G$-invariant connection if and only if $G/H$ is a \emph{reductive homogeneous space}, in the sense that $\h$ admits an $\Ad_H$-invariant complement in $\g$.
In that case, any $G$-invariant principal $U$-bundle over $G/H$ (and any associated vector bundle thereof) also admits a $G$-invariant connection.
  
A similar statement holds for Lie groupoid pairs. \emph{For a Lie groupoid pair $(\L,\A)$, the principal $\A$-bundle $\L$ admits an $\L$-invariant fibered connection form if and only if $\L/\A$ is a reductive homogeneous space, in the sense that there exists a vector subbundle $\bb$ in $\ll$ supplementary to $\aa$ and which is invariant under the $\Bis(\A)$-action on $\ll$. In that case, any $\L$-equivariant principal bundle over $\L/\A$ admits fibered connection forms.}
	
Indeed, by left-invariance, a fibered connection form $\omega:T^\ta\L\to \so^*\aa$ is equivalent, through the formula $\omega=\underline\omega\circ\tau_\L$, to a projection $\underline\omega:\ll\to \aa$ which is anchored and invariant under the $\Bis(\A)$-actions on $\ll$ and $\aa$. The kernel of $\underline\omega$ is then the desired subbundle $\bb$. Note that the anchor map restricted to $\bb$ necessarily vanishes, so $\L/\A$ is in some sense a ``bundle of (Lie group) homogeneous spaces''.
	
By Proposition~\ref{prop:exhaust}, any $\L$-equivariant principal bundle $Q$ over $\L/\A$ is a composition $_\L\L_\A \circ P$ with a generalized morphism $P$. Such morphisms always have a Maurer--Cartan form (see Remark~\ref{rmk:fibered-connection-forms-unique-on-generalized-morphisms}), hence by Proposition~\ref{prop:composition-of-connections} $Q$ always admits fibered connection forms.

\subsection{Proper Lie groupoids}

The vanishing theorem for proper Lie groupoids \cite[Proposition 1]{crainic_differentiable_2003} states that $H^d(\A,E)=0$ for $d\geq 1$ whenever $\A$ is a proper Lie groupoid and $E$ is an $\A$-module. 
Hence the Atiyah class of any generalized morphism from $\A$ to a transitive Lie groupoid $\U$ vanishes.

In particular, for such an $\A$,
\begin{enumerate}
	\item any equivariant principal $\U$-bundle over a homogeneous space $\L/\A$ with $\U$ transitive admits fibered connection forms,
	\item any equivariant vector bundle over a homogeneous space $\L/\A$ admits invariant connections,
	\item any homogeneous space $\L/\A$ admits invariant fibrewise affine connections.
\end{enumerate}

Note that, in contrast to the Lie group case, homogeneous spaces of proper Lie groupoids are in general not reductive homogeneous spaces in the sense of Section~\ref{sec:reductive-homogeneous-spaces}.
Indeed, there might not even exist an anchored projection $\nabla:\ll\to\aa$ since this would imply that the anchor vanishes on the kernel of $\nabla$.


\section{Poincaré--Birkhoff--Witt theorem}
\label{sec:PBW-theorem}

In this section, for a Lie groupoid pair $(\L,\A)$ we relate the vanishing of the Atiyah class (Definition \ref{def:Atiyah-class-of-a-Lie-groupoid-pair}) to the existence of an $\A$-equivariant Poincaré--Birkhoff--Witt map.
The latter map is obtained as the infinite jet of the exponential map of an invariant connection on $\L/\A$ whose existence was shown in Corollary \ref{cor:atiyah=0GivesConnections} to be equivalent to the vanishing of the Atiyah class.


\subsection{Exponential map of invariant connections}
\label{ssec:exponential-map-of-invariant-connections}

Let $\nabla$  be a fibrewise affine connection on $\L/\A\to M$.

For any $ x \in M$ there is, in view of (\ref{eq:t_bar}), a natural isomorphism of vector spaces
\[  T^{\underline\ta}_x (\L/\A) \cong T^\ta_x\L / T^\ta_x \A . \]
Since the differentials at $ x$ of the inverse maps $ \inv $ of $ \L$ and $ \A$ restrict to isomorphisms $ \ll_x \to T^\ta_x\L  $ and 	$ \aa_x \to T^\ta_x\A $, they further induce an isomorphism
\begin{equation}
\label{eq:inverse_canonique} 
  T_x\inv :(\ll/\aa)_x \simeq T^{\underline\ta}_x(\L/\A) . 
\end{equation}
Now, the exponential map $ \exp_x^\nabla $ at $x$ of the connection $\nabla$ is a diffeomorphism from a neighborhood of $0 \in  T^{\underline\ta}_x(\L/\A)$ to a neighborhood of $ x$ in the fiber $\underline\ta^{-1}(x)$. 
Composing $ \exp_x^\nabla $ with the isomorphism (\ref{eq:inverse_canonique}),	we get a diffeomorphism from a neighborhood of $0$ in $\ll_x/\aa_x$	to a neighborhood of $ x$ in the fiber ${\underline\ta}^{-1}(x)$. 
By a slight abuse of notation, we shall still denote by $ \exp_x^\nabla $ this diffeomorphism. 
This construction can now be done at all points $ x \in M$. 
The henceforth obtained assignment
\[  \beta \mapsto \exp^\nabla_{q (\beta)}  (\beta) \]
is a diffeomorphism from a neighborhood 
$[M,\ll/\aa]$
of $M $ in $\ll/\aa $	to a neighborhood 
$[M,\L/\A]$
of $M $ in $\L/\A $. We shall denote by 
$ \exp^\nabla : [M,\ll/\aa] \to [M,\L/\A]$  
this diffeomorphism.
	
Let us study the equivariance of $ \exp^\nabla$.	Notice first that $ \exp^\nabla$ is a fibered diffeomorphism that respects $M$, i.e.\ the diagrams
\[
  \vcenter{\xymatrix{
    & \L/\A \\
    M \ar[r]_0 \ar[ur]^1 & \ll/\aa \ar[u]_{\exp^\nabla} 
  }} 
  \text{\quad and\quad} 
  \vcenter{\xymatrix{
    \L/\A \ar[dr]^{\underline\ta} & \\
    \ll/\aa \ar[u]^{\exp^\nabla} \ar[r]_q & M     
  }}
\]
are commutative.
The commutativity of these diagrams implies that,  for all $\generic \in \A $ whose  action on a given $\beta \in \ll/\aa$ makes sense (i.e.\ when $\so(\generic)=q(\beta)$), the action of $\generic$ on $ \exp^{\nabla}(\beta) $ also makes sense (i.e.\ $\so(\generic) = \ta \circ \exp^{\nabla}(\beta) $), at least when $\exp^{\nabla}(\beta) $ is defined. 
Let $U$ be a neighborhood of $ M$  in $\ll/\aa$ on which $\exp$ is defined and let
$V$ be its image through the exponential map, we say that the diffeomorphism $\exp^\nabla $ is \emph{$\A$-equivariant on $U$} when the equality
\[  \exp^\nabla (\generic \cdot \beta ) = \generic \cdot \exp^\nabla (\beta)  \]
holds for all $ \generic \in \A $ and $\beta \in U \subset \ll/\aa $ such that $\so(\generic)=q(\beta)$ and 
$\generic \cdot \beta \in U $. 

\begin{thm}
\label{thm:exponential}
  A Lie groupoid pair $(\L,\A)$ has vanishing Atiyah class if and only if there exists an $\A$-equivariant diffeomorphism from a neighborhood of $M $ in $\ll/\aa $	to a neighborhood of $M $ in $\L/\A $.
  In particular, under these equivalent conditions, the $\A$-action on $ \L/\A$ is linearizable.
  For the ``if'' part, only the fibered 2-jet of such a diffeomorphism is needed.
\end{thm}

\begin{proof}
  According to the fourth item in Corollary  \ref{cor:atiyah-class-A-modules}, the Atiyah class of $(\L,\A) $ vanishes if and only if an $\A$-compatible $\ll$-connection on $L/A$ exists. 
  By Theorem \ref{thm:L-conn-on-LoverA--conn-on-tgt-bundle}, $\A$-compatible $\ll $-connections on $L/A$ exist if and only if $\L$-invariant fibrewise affine connections on $ \L/\A\to M$ exist. 
  We are thus left the task of relating the existence of invariant fibrewise affine connections on $\L/\A$ to the existence of equivariant diffeomorphisms from a neighborhood of $M $ in $\ll/\aa $	to a neighborhood of $M $ in $\L/\A $.
  The statement about the linearizability of the action will then follow from the linearity of the $\A$-action on $\ll/\aa$.

	Let first $\nabla$ be an $\L$-invariant fibrewise affine connection on $\L/\A$.
  For any diffeomorphism $\Psi$ from $(\L/\A)_x$ to $ (\L/\A)_y$ mapping $x$ to $y$ and preserving the connection $\nabla$, the diagram
  \[
    \vcenter{\xymatrix{
      T_x  (\L/\A)_x \ar[d]^{ T_x \Psi }  \ar[r]^{ \exp^\nabla } 
      & (\L/\A)_x \ar[d]^{ \Psi } \\ 
      T_y (\L/\A)_y   \ar[r]^{ \exp^\nabla } 
      &  (\L/\A)_y
    }}
  \]
  is commutative.
  Applying this to a left-translation $\Psi=L_\generic $ for some $\generic \in \A$ with source $x$ and target $y$ yields the desired result, since $T L_\generic: T_x  (\L/\A)_x \to T_y (\L/\A)_y$ coincides with the $\A $-action defined in section \ref{ssec:bigA-action-on-L-over-A}, after identifying $  T_x  (\L/\A)_x $ with $(\ll/\aa)_x$ and $  T_y  (\L/\A)_y $ with $(\ll/\aa)_y$.
	
	Conversely, let $\phi$ be an $\A$-equivariant diffeomorphism from a neighborhood of the zero section in $\ll/\aa $	to a neighborhood of $M $ in $\L/\A $.
	Although it is not the exponential map of a connection, it does define a unique torsionfree fibrewise affine connection $\nabla$ on $\L/\A $ whose geodesic symmetry agrees with that of $\phi$ up to order~2.
	Indeed, transporting $\phi$ using \eqref{eq:inverse_canonique} and the transitive left $\L$-action yields a well-defined $\L$-equivariant diffeomorphism from a neighborhood of the zero section in $T^{\underline\ta}(\L/\A)$ to a neighborhood of the diagonal in $\L/\A \times_M \L/\A$.
	Let us denote it by $\Phi:X_p\mapsto (p,\Phi_p(X_p))$.
	Let, for each $p\in \L/\A$, $s_p$ be the fibrewise ``geodesic symmetry'' at $p$ defined by $s_p(\Phi_p(X)) = \Phi_p(-X)$ for sufficiently small $X\in T_p^{\underline\ta}(\L/\A)$.
	Then, we have $s_p(p)=p$ and $T^{\underline \ta}_ps_p=-\Id$, and the formula
	\begin{equation*}
		\left( \nabla_XY \right)_p = \frac12 \left[ X, Y + (s_p)_*Y \right]_p
	\end{equation*}
	defines \cite[Proposition 1.3]{bertelson_affine_2011} the announced connection.
	It only depends on the 2-jet of each $s_p$ at $p$, and thus of each $\Phi_p$ at $p$.
	It is $\L$-invariant since $\Phi$ and hence all $s_p$ are $\L$-equivariant.
\end{proof}


\subsection{Poincaré--Birkhoff--Witt theorem}
\label{ssec:PBW-theorem}

The space $ J^\infty([M,\ll/\aa])$ of fibered (along the projection $\ll/\aa\to M$) jets at $M$ of smooth real-valued functions on $\ll/\aa$ is an algebra. 
It is also a $C^\infty(M)$-module.
Let ${\mathcal I}([M,\ll/\aa]) $ be the ideal of jets of  smooth functions on $[M,\ll/\aa]$ vanishing identically on $M$, and consider the decreasing sequence of ideals $({\mathcal I}^k([M,\ll/\aa]) )_{k \geq 0}$ of jets along $M$ of functions vanishing on $M$ together with their $k$ first derivatives. Let us denote by $\Hom_{C^\infty(M)}(J^\infty([M,\ll/\aa]),C^\infty(M)) $  the dual relative to this filtration, i.e.\ the space of $C^\infty(M) $-linear maps from $J^\infty([M,\ll/\aa])$ to $C^\infty(M) $ vanishing on ${\mathcal I}^k([M,\ll/\aa]) $ for some integer $k$. 

\begin{rmk}
  Let us recall that we do not need to consider the topological dual here.
  The $C^\infty(M) $-linear maps from $J^\infty([M,\ll/\aa])$ to $C^\infty(M) $ vanishing on ${\mathcal I}^k([M,\ll/\aa]) $ for a given integer $k$ are automatically isomorphic to the space  of sections of an ordinary vector bundle over $M$.  
  When $M$ is a point for instance, this dual is a vector space of finite dimension.
\end{rmk}

By construction, $\Hom_{C^\infty(M)}(J^\infty([M,\ll/\aa]),C^\infty(M)) $ is a filtered coalgebra, with a filtration given by 
$\Hom_{C^\infty(M)}(J^\infty([M,\ll/\aa]),C^\infty(M)) = \cup_{k \geq 0} J^\infty([M,\ll/\aa]))_k^\perp $
with $J^\infty([M,\ll/\aa])_k^\perp $ being, for all $ k \geq 0$ the subspace of elements in 
\[ \Hom_{C^\infty(M)}(J^\infty([M,\ll/\aa]),C^\infty(M)) \]
vanishing on $ {\mathcal I}^k([M,\ll/\aa]) $.
The coalgebra structure is the dual of the algebra structure on $J^\infty([M,\ll/\aa])$ and is compatible with the filtration, since 
\[
  {\mathcal I}^k([M,\ll/\aa])  {\mathcal I}^{k'}([M,\ll/\aa])
\subset {\mathcal I}^{k+k'}([M,\ll/\aa]).
\]
We now intend to describe this coalgebra explicitly. Recall that a section $X$ of $\ll/\aa$ can be seen as a vertical vector field $\overline{X}$ on the fibered manifold $\ll/\aa\to M$ that is constant on each fiber.
The required filtered coalgebra isomorphism is obtained by mapping a section $X_1 \odot \dots \odot X_k$  in $S^k(\ll/\aa)$ (with $X_1, \dots,X_k \in \Gamma(\ll/\aa)$),
to the element of  $\Hom_{C^\infty(M)}(J^\infty([M,\ll/\aa]),C^\infty(M))$ given by
 \begin{equation}\label{eq:co_iso_symetrique} F \mapsto  \left. \overline{X_1} \left[ \dots \overline{X_k} \left[ F \right] \right] \right|_M \end{equation}  
for all $ F \in J^\infty([M,\ll/\aa])$.
The differential operator above vanishes for $F \in {\mathcal I}^k([M,\ll/\aa]) $,
and the correspondence obviously sends $S^k(\ll/\aa)$ bijectively to $J^\infty([M,\ll/\aa])^\perp_k$. 

\begin{lem}\label{lem:coalgebras1}
  The map \eqref{eq:co_iso_symetrique} induces a natural filtered coalgebra isomorphism from  $ \Gamma(S(\ll/\aa))$ to $\Hom_{C^\infty(M)}(J^\infty([M,\ll/\aa]),C^\infty(M))$.
\end{lem}

The same construction can be done considering $J^\infty([M,\L/\A])$, yielding a filtered coalgebra $\Hom_{C^\infty(M)}(J^\infty([M,\L/\A]),C^\infty(M))$ of fibered jets at $M$ of smooth functions on $\L/\A$. 
Let ${\mathcal I}([M,\L/\A]) $ be the ideal of smooth functions on $[M,\L/\A]$ vanishing identically on $M$, and consider the decreasing sequence of ideals $({\mathcal I}^k([M,\L/\A]) )_{k \geq 0}$ of jets of functions vanishing on $M$ together with their $k$ first derivatives. 
Let us denote by $\Hom_{C^\infty(M)}(J^\infty([M,\L/\A]),C^\infty(M)) $  the dual relative to this filtration, i.e.\ the space of $C^\infty(M) $-linear maps from $J^\infty([M,\L/\A])$ (which is a $ C^\infty(M) $-module since $\L/\A$ fibers over $M$) to $C^\infty(M) $ vanishing on ${\mathcal I}^k([M,\L/\A]) $ for some integer $k$.

Since $\L$ acts on the left on $\L/\A$, a section $X$ of $\ll$ can be seen  as a vector field along the fibers of $ \L/\A$  that we denote by $\overline{X} $.
Let us map an element $X_1 \cdot \dots \cdot X_k$  in $U(\ll)$ (with $X_1, \dots,X_k \in \Gamma(\ll)$),
to the element of  $\Hom_{C^\infty(M)}(J^\infty([M,\L/\A]),C^\infty(M))$ given by
  \begin{equation}\label{eq:co_iso_univ} F \mapsto  \left. \overline{X_k} \left[ \dots \overline{X_1} \left[ F \right] \right] \right|_M \end{equation}
for all $ F \in J^\infty([M,\ll/\aa])$.
The map above vanishes on ${\mathcal I}^k([M,\L/\A])$ and every element in 
$\Hom_{C^\infty(M)}(J^\infty([M,\L/\A]),C^\infty(M))$ vanishing on  ${\mathcal I}^k([M,\L/\A])$ is uniquely a linear combination
of differential operators of this form. 
Also, for $X_k  \in \Gamma(\aa)$, the previous differential operator is clearly equal to zero, which eventually leads to the following lemma.

\begin{lem}\label{lem:coalgebras2}
The map \eqref{eq:co_iso_univ} induces a filtered coalgebra isomorphism from $ U(\ll) / \allowbreak U(\ll) \cdot \Gamma(\aa)$ to $\Hom_{C^\infty(M)}(J^\infty([M,\L/\A]),C^\infty(M))$.
\end{lem}

Assume that we are given a diffeomorphism $\Phi $ from a neighborhood $[M,\ll/\aa]$ of $M$ in $\ll/\aa $
to a neighborhood  $[M,\L/\A]$ of $M$ in $\L / \A$
which is a \emph{fibered map}, i.e.\ assume that the diagrams
\begin{equation}
\label{eq:PhiisLocal}
  \vcenter{\xymatrix{
    & \L/\A \\
    M \ar[r]_0 \ar[ur]^1 & \ll/\aa \ar[u]_{\Phi} 
  }}
  \text{\quad and\quad}
  \vcenter{\xymatrix{
    \L/\A \ar[dr]^{\underline\ta} & \\
    \ll/\aa \ar[u]^{\Phi} \ar[r]_q & M     
  }}
\end{equation}
are commutative.
Since $\Phi$ is a diffeomorphism, its pullback
\begin{equation}\label{eq:Phi*1} \Phi^* :J^\infty([M,\L/\A]) \to J^\infty([M,\ll/\aa])  \end{equation}
is an algebra isomorphism. The commutativity of the diagram on the right-hand side of (\ref{eq:PhiisLocal})
implies that $ \Phi^*$ is $ C^\infty(M)$-linear and the commutativity of the diagram on the left hand side of \eqref{eq:PhiisLocal} implies that the algebra isomorphism $\Phi^*$ restricts to an algebra isomorphism, valid for all integer $k$,
  \begin{equation}\label{eq:Phi*2}  \Phi^* : {\mathcal I}^k([M,\L/\A]) \to {\mathcal I}^k([M,\ll/\aa]). \end{equation}
Altogether, relations \eqref{eq:Phi*1} and \eqref{eq:Phi*2} imply that 
the algebra isomorphism $\Phi^*$ can be dualized again to induce a filtered coalgebra isomorphism
\[ \Hom_{C^\infty(M)}(J^\infty([M,\ll/\aa]),C^\infty(M)) \to \Hom_{C^\infty(M)}(J^\infty([M,\L/\A]),C^\infty(M)). \]
By Lemmas \ref{lem:coalgebras1} and \ref{lem:coalgebras2}, $\Phi^*$ therefore induces
a coalgebra isomorphism that we denote by 
\[ PBW_{\Phi} :  \Gamma(S(\ll/\aa))\to  U(\ll) / U(\ll) \cdot \Gamma(\aa)  \]
and call the \emph{Poincaré--Birkhoff--Witt map of $\Phi$}. By construction, $PBW_{\Phi}$ makes this diagram
\begin{equation}\label{def:PBW} 
\xymatrix{\Hom_{C^\infty(M)}(J^\infty([M,\ll/\aa]),C^\infty(M))  \ar[r]\ar[d]& \Hom_{C^\infty(M)}(J^\infty([M,\L/\A]),C^\infty(M)) \ar[d]\\  \Gamma(S(\ll/\aa)) \ar[r]^-{PBW_\Phi}& U(\ll) / U(\ll) \cdot \Gamma(\aa) } \end{equation}
commutative.

We now intend to explore the consequences of the existence of an $\A$-equivariant local diffeomorphism $\Phi$.
Every bisection $\Sigma$ of $\A$ induces diffeomorphisms of both $\L/\A  $ and $\ll/\aa $
that  we denote by $\underline{\Sigma}$.
The equivariance of $\Phi$ means that the following diagram is commutative:
\begin{equation}
\label{eq:equivMeaning}
  \vcenter{\xymatrix{
    J^\infty([M,\L/\A]) \ar[r]^{\Phi^*} \ar[d]_{\underline{\Sigma}^*} 
    & J^\infty([M,\ll/\aa] )  \ar[d]^{\underline{\Sigma}^*} \\ 
    J^\infty([M,\L/\A]) \ar[r]^{\Phi^*} 
    & J^\infty([M,\ll/\aa] ) 
  }} .
\end{equation}
For all integer $k$, it also restricts to
\[
  \vcenter{\xymatrix{
    {\mathcal I}^k ([M,\L/\A]) \ar[r]^{\Phi^*} \ar[d]^{\underline{\Sigma}^*} 
    & {\mathcal I}^k([M,\ll/\aa] )  \ar[d]^{\underline{\Sigma}^*} \\ 
    {\mathcal I}^k([M,\L/\A]) \ar[r]^{\Phi^*} 
    & {\mathcal I}^k([M,\ll/\aa] ) 
  }} .
\]
Altogether, these commutative diagrams allow to dualize (\ref{eq:equivMeaning}), to wit:  
\[
  \vcenter{\xymatrix{
    \Hom_{C^\infty(M)}(J^\infty([M,\L/\A]),C^\infty(M))   & \ar[l]_-{\Phi}   \Hom_{C^\infty(M)}(J^\infty([M,\ll/\aa]),C^\infty(M))   \\ \Hom_{C^\infty(M)}(J^\infty([M,\L/\A]),C^\infty(M))  \ar[u]^{\underline{\Sigma}}& \ar[l]_-{\Phi} \Hom_{C^\infty(M)}(J^\infty([M,\ll/\aa]),C^\infty(M)) \ar[u]_{\underline{\Sigma}}
  }} .
\]
The isomorphism of filtered coalgebras described in Lemma \ref{lem:coalgebras1} intertwines the action of the pseudogroup of local bisections $\Bis(\A)$ on $\Hom_{C^\infty(M)}(J^\infty([M,\L/\A]), \allowbreak C^\infty(M))$ with the action defined in Section \ref{ssec:bigA-action-on-L-over-A}.
The isomorphism of filtered coalgebras described in Lemma \ref{lem:coalgebras2} intertwines the action of the pseudogroup $\Bis(\A)$ on $\Hom_{C^\infty(M)}( \allowbreak J^\infty([M,\L/\A]), \allowbreak C^\infty(M))$ with the action on  $U(\ll)/U(\ll) \cdot \Gamma(A)$ defined in Proposition \ref{bisection_action}.
This proves, in view of (\ref{def:PBW}), the commutativity of the diagram
\[
  \vcenter{\xymatrix{
  U(\ll)/ U(\ll) \cdot  \Gamma (\aa)  & \ar[l]_-{PBW_\Phi} \Gamma( S(\ll/\aa))   \\ 
  U(\ll)/ U(\ll) \cdot  \Gamma (\aa)    \ar[u]^{\underline{\Sigma}}& \ar[l]_-{PBW_\Phi} \Gamma( S(\ll/\aa)) \ar[u]_{\underline{\Sigma}}
  }} .
\]
Tracking the implications in the other direction, we get the following result.

\begin{prop}\label{PhiAndEquiv}
Let $(\L,\A)$ be a Lie groupoid pair.
For every $\A$-equivariant diffeomorphism $\Phi $ from a neighborhood of $M$ in $\ll/\aa$ to a neighborhood of $M$ in $\L / \A$, the map $PBW_{\Phi} :  \Gamma(S(\ll/\aa))\to  U(\ll) / U(\ll) \cdot \Gamma(\aa)$ defined as above is a $\Bis(\A)$-equivariant isomorphism of filtered coalgebras.
Conversely, any such equivariant isomorphism of filtered coalgebras defines an infinite jet of equivariant diffeomorphism from a neighborhood of $M$ in $\ll/\aa$ to a neighborhood of $M$ in $\L / \A$.
\end{prop}

As a direct application of Theorem~\ref{thm:exponential} and Proposition~\ref{PhiAndEquiv}, we get the following theorem, which is an equivalent at the groupoid level of Theorem 1.1 in \cite{calaque_pbw_2014}, which was also re-proved in \cite{laurent-gengoux_kapranov_2014} using different techniques.

\begin{thm}
\label{thm:calaque}
	A Lie groupoid pair $(\L,\A)$ has vanishing Atiyah class if and only if there exists a $\Bis(\A)$-equivariant filtered coalgebra isomorphism from $\Gamma(S(\ll/\aa))$ to $U(\ll) / U(\ll) \cdot \Gamma(\aa)$.
\end{thm}


\section{Local Lie groupoids}
\label{sec:local-Lie-groupoids}

We sketch in this section what happens when we work with local Lie groupoids instead of Lie groupoids.
Recall that unlike Lie algebras, which are always the tangent space at the unit element of a Lie group, Lie algebroids may not be the infinitesimal object of a Lie groupoid (see \cite{crainic_integrability_2003} for a characterization of integrable Lie algebroids).
There is however always a \emph{local Lie groupoid} (see \cite[Section 3]{debord_local_2001} for a definition)  integrating a Lie algebroid $\aa$ (as shown in \cite[Corollary 5.1]{crainic_integrability_2003}).
This local Lie groupoid is not unique: for instance one can always replace it by a neighborhood of the unit submanifold.
For any $\aa$-module $E$, there exists moreover such a local Lie groupoid acting on $E$. 

Now, for any Lie algebroid pair $(\ll,\aa)$, the Lie algebroid $\ll$ can be integrated to an $\so$-connected local Lie groupoid $\L$ and, upon replacing $\L$ by a neighborhood of $M$ in $\L$ if necessary, we can assume that 
$\aa$ is the Lie algebroid of a closed local $\so$-connected Lie subgroupoid $\A \subset \L$. 
In that case, we shall say that the pair $(\L,\A)$ is a \emph{local Lie groupoid pair} that integrates $(\ll,\aa)$.
Again, it is not unique. Since $\A$ is $\so$-connected (hence also $\ta$-connected), the quotient $\L/\A$ coincides with the quotient of $\L$ by the foliation on $\L$ induced by the right Lie algebroid action of $\aa$. 
Since the tangent space of this foliation has, for all $m \in M$, no intersection with $T_m M \subset T_m\L$,
$\L$ can be replaced by a wide, open, local Lie subgroupoid such that the quotient $\L/\A$ is a smooth manifold. 
By construction, this submanifold fibers over $M$ and is acted upon on the left by $\L$.

Also, for every $\aa$-module $E$, upon shrinking both $\L$  and $\A$ if necessary, one can assume that $E$ comes equipped with an $\A$-module structure, and the induced vector bundle $\frac{\L \times_{M} E}{\A} \to \L /\A$ is made sense of through the same construction. Theorems \ref{thm:L-conn-on-E--conn-on-assoc-bundle} and \ref{thm:L-conn-on-LoverA--conn-on-tgt-bundle} extend to local Lie groupoid pairs.

\begin{thm}\label{th:localCase}
  Let $(\ll, \aa)$ be a Lie algebroid pair over $M$, and let $E\to M$ be an $\aa$-module. 
  Then, there exists a local Lie groupoid pair $(\L,\A)$ integrating it such that $\L/\A$ is a manifold and $E$ is an $\A$-module.
  Moreover, there is a bijective correspondence between
  \begin{enumerate}[label=(\arabic*)]
    \item $\aa$-compatible $\ll$-connections on $E$, and
    \item $\L$-invariant fibered connections on the associated vector bundle $\frac{\L \times_M E}{\A} \to \L/\A \to M$.
  \end{enumerate}
  For $E =\ll/\aa$, the associated vector bundle $ \frac{\L \times_M E}{ \A} \to \L/\A \to M$ is isomorphic to the tangent bundle of the natural projection $ \L/\A \to  M$.
\end{thm}

The definition of cohomology for a local Lie groupoid $\A$, see Section \ref{ssec:lie-groupoids}, must be adapted.
Let $\A_n$ stand for the manifold of all $n$-tuples $(\generic_1, \dots, \generic_n) \in \A^n$ such that the product of any two successive elements is defined. Notice that $M$ is for all $n \in {\mathbb N}$ a submanifold of $\A_n$ through
the natural assignment $m \mapsto (\unit(m),\dots,\unit(m))$ ($n$ times). 
For any $\A$-module $E$, the cohomology $H^\bullet(\L,E)$ is defined as being the germification around $M$ of the complex that appears in \eqref{eq:groupoid-cohomology-complex}, i.e.\ the cohomology of 
\begin{equation*}
  \xymatrix{
    C^0_M(\A,E)
    \ar[r]^{\partial_0} &
    C^1_M(\A,E)
    \ar[r]^{\partial_1} &
    C^2_M (\A,E)
    \ar[r]^{\partial_2} &
    C^3_M(\A,E)
    \ar[r]^-{\partial_3} &
    \cdots
  }
\end{equation*}
where, for all $n \in {\mathbb N_*}$, $C^n_M (\A,E)$ is the space of germs around $M$
of smooth functions from $\A_n$ to $E$ such that 
$F(\generic_1, \dots, \generic_n) \in E_{\ta (\generic_1)}$
for all $ (\generic_1, \dots, \generic_n) \in \A_n$ while the differential $\partial_n F$ is defined by a same formula, but taken at the level of germs around $M$. Definition \ref{defn:atiyah-class-transitive-case} of the Atiyah class still makes sense, and Proposition \ref{prop:atiyah-class-transitive-case} is easily adapted, as well as Corollary \ref{cor:atiyah-class-A-modules} at least for the case of interest for our purpose, namely $\aa$-modules.
		
More precisely, let $(\ll,\aa)$ be a Lie algebroid pair and $E$ an $\aa$-module. The first item in Proposition \ref{cor:atiyah-class-A-modules} deals only with algebroids, and therefore holds true again: there exist $\ll$-connections on $E$ extending the $\aa$-action. 
To generalize the next items in Proposition \ref{cor:atiyah-class-A-modules}, notice that, for local groupoids, bisections and the assignment $\Add$ defined in \eqref{eq:def-kappa} still make sense and satisfy essentially the same properties, which allows to extend our proofs to this context without additional difficulties.
For any $\ll$-connection $\nabla$ on $E$ extending the $\aa$-action and for any local Lie groupoid pair $(\L,\A)$ integrating $(\ll,\aa)$, a smooth local groupoid 1-cochain $R^\nabla\in C^1(\A,(\ll/\aa)^*\otimes \End E)$ is still defined by
\begin{align}
\label{eq:definition-R-nabla-particular-local}
  R^\nabla(\gamma)(l,e) =\bisection \star \nabla_{\Ad_{\bisection^{-1}}  l} \left( \bisection^{-1} \star e \right) - \nabla_l e  
\end{align}
with $\bisection \in \Bis (\A)$ and $l\in\Gamma(\ll)$, $e\in\Gamma(E)$.
It is still true that the 1-cochain $R^\nabla \in C^1(\A,(\ll/\aa)^*\otimes \End E) $ is closed, that its cohomology class $\alpha_{(\L,\A),\varphi}=[R^\nabla]$ is independent of the choice of $\nabla$ and that the class $\alpha_{(\L,\A),\varphi}$ is zero if and only if there exists an $\A$-compatible $\ll$-connection on $E$. 

Moreover, Corollary \ref{cor:infinitesimalAtiyahVanishes-implies-globalAtiyahVanishes} holds and yields that the class $\alpha_{(\L,\A),\varphi}$ is zero if the Atiyah class of $E$ with respect to the Lie algebroid $(\ll,\aa)$ is zero.
Altogether, this last point and theorem \ref{th:localCase} imply a variation of Corollary \ref{cor:L-conn2}, which is interesting to state explicitly.

\begin{cor}
\label{cor:L-conn3}
  Let $(\ll, \aa)$ be a Lie algebroid pair over $M$ and let $E$ be an $\aa$-module. Then there exist a local Lie groupoid pair $(\L,\A)$ integrating $(\ll,\aa) $ with $\A$ source-connected and such that $\L/\A$ is a manifold and $E$ is an $\A$-module.
  Moreover, $E$ has vanishing Atiyah class with respect to $(\ll,\aa)$ if and only if there exists an $\L$-invariant fibered connection on the associated vector bundle $ \frac{\L \times_M E}{ \A} \to \L/\A \to M$.
\end{cor}

For $E = \ll/\aa$ for instance, we obtain that there exist $\L$-invariant fibrewise affine connections on $\L/\A \to M$ when the Atiyah class of the Lie algebroid pair vanishes.

An $\L$-invariant fibrewise affine connection on $\L/\A \to M$ as in Corollary \ref{cor:L-conn3} is the only required object for all the arguments in Section \ref{ssec:PBW-theorem}. We  thus obtain:   

\begin{cor}\label{cor:calaque2}
  A Lie algebroid pair $(\ll,\aa)$ has vanishing Atiyah class if and only if there exists a filtered coalgebra isomorphism from $ \Gamma(S(\ll/\aa))$ to $U(\ll) / U(\ll) \cdot \Gamma(\aa)$ which intertwines, for all $a \in \Gamma(\aa)$,
  \begin{enumerate}
    \item the unique derivation of $ \Gamma(S(\ll/\aa))$ which is given by $\rho(a)$ on smooth functions on $M$ and by the canonical $\aa$-action on sections of $\Gamma(\ll/\aa)$, and
    \item the left multiplication $L_a:U(\ll) / U(\ll) \cdot \Gamma(\aa) \to U(\ll) / U(\ll) \cdot \Gamma(\aa) $ by $a \in \Gamma(\aa)$.
  \end{enumerate}
\end{cor}		

Corollary \ref{cor:calaque2} is the content of Theorem 1.1 in \cite{calaque_pbw_2014} when the Lie algebroids considered are over ${\mathbb R}$. 
When the Lie algebroid is over $ {\mathbb C}$, however, the geometrical interpretation used here is not relevant anymore, and only the methods of
\cite{calaque_pbw_2014} or \cite{laurent-gengoux_kapranov_2014} remain valid.
Also:

\begin{thm}
\label{thm:exponential-local}
  A Lie algebroid pair $(\ll, \aa)$ over $M$ has vanishing Atiyah class if and only if there exists a local Lie groupoid pair $(\L,\A)$ integrating $(\ll,\aa)$ with $\A$ source-connected and such that $\L/\A$ is a manifold equipped with an $\A$-equivariant diffeomorphism 
from a neighborhood of $M $ in $\ll/\aa $	to a neighborhood of $M $ in $\L/\A$.
In particular, under these equivalent conditions, the $\A$-action on $ \L/\A$ is linearizable.
\end{thm}


\section{Examples and applications}
\label{sec:examples}


\subsection{Lie group pairs and homogeneous spaces}	

We now explore the case of Lie groups. We shall use the same terminology without the suffix ``-oid'' and simply speak of Lie group pairs, Lie algebra pairs and so on. Notice that, for $\g$ a Lie algebra, a $\g$-connection on a vector space $E$ is simply a bilinear assignment $\g \times E \to E$. In this case, our results give back well-known results on homogeneous spaces.
	
Let $(G,H)$ be a Lie group pair, so $H$ is a closed Lie subgroup of $G$, and let $(\g,\h)$ be the corresponding Lie algebra pair.
Consider the associated homogeneous space $G/H$.
For $E$ an $H$-module, the associated vector bundle is the vector bundle $ \frac{G \times E}{H} \to G/H$. Also, the canonical $H$-module structure on $\g/\h$ described in Section \ref{ssec:bigA-action-on-L-over-A} is simply induced by the adjoint action under the quotient map.
Specializing to this case, Theorem \ref{thm:general} yields Wang's characterization \cite[p.\ 1]{wang_invariant_1958}, Theorem \ref{thm:L-conn-on-E--conn-on-assoc-bundle} yields Proposition 2.7 in \cite{bordemann_atiyah_2012} and Theorem \ref{thm:L-conn-on-LoverA--conn-on-tgt-bundle} yields Theorem 1 in \cite{wang_invariant_1958}.

\begin{rmk}
  When $H$ is not closed in $G$, or when we are only given a Lie algebra pair $(\g,\h)$, we may consider local Lie groups as in Section \ref{sec:local-Lie-groupoids} and obtain similar results.
\end{rmk}

Let us turn our attention to the case  $E = \g/\h$, that is, let us explore the meaning of the vanishing of the Atiyah class of a Lie group pair. First, notice that the Atiyah class of a Lie algebra pair automatically  vanishes in the following cases:
\begin{enumerate}
	\item When $\h$ is reductive in $\g$, i.e.\ when there exists an $\h$-invariant subspace $\mathfrak{m}$ in direct sum (as vector spaces) with $\h$ in $\g$. In that case, extending the $\h$-action on $\g/\h$ by zero on $\mathfrak{m}$ yields an $\g$-connection on $\g/\h$ with vanishing Atiyah cocycle.
	\item When $\h$ is semisimple, since then it has no first degree cohomology.
	\item When $\mathfrak d$ is the double of a Lie bialgebra $(\g,\g^*)$ arising from an $r$-matrix $r\in\g\otimes\g$, the Atiyah class of the Lie pair $(\mathfrak d,\g)$ vanishes. 
\end{enumerate}

\begin{rmk}
Let us explain the third item: we owe this result to Khaoula Abdeljellil, who gave an explicit expression of a connection
whose Atiyah cocycle is zero (see also the recent paper of Hong \cite{hong_atiyah_2018} on this subject). It is, up to a scalar factor, given by $\nabla_\alpha \beta = \ad_{r^\#(\alpha)} \beta$ for all $\alpha,\beta\in\mathfrak g^*$.
It can be seen conceptually as follows: Drinfeld \cite{drinfeld_1989_quasi} showed that a bialgebra $\mathfrak g$ is coboundary if and only if the associated infinitesimal homogeneous space $\mathfrak d/\mathfrak g$ is reductive, where $\mathfrak d$ is the double of $\mathfrak g$.
Reductive homogeneous spaces are well known to have invariant connections \cite{nomizu_invariant_1954}, hence vanishing Atiyah class.
\end{rmk}

In the case of Lie groups, Theorem~\ref{thm:exponential} implies the following result.
			
\begin{prop}
\label{prop:exponential-group}
  A Lie group pair $(G,H)$ has vanishing Atiyah class if and only if there exists an $H$-equivariant diffeomorphism from a neighborhood of $0$ in $\g/\h $	to a neighborhood of $eH$ in $G/H$.
  In particular, under these equivalent conditions, the $H$-action on $G/H$ is linearizable.
\end{prop}	

Now, let  $D = G \Join H$ be a matched pair of Lie groups, i.e.\ $G$ and $H$ are two Lie subgroups such that the multiplication map $m$ of $D$ restricts to a diffeomorphism $m|_{G\times H}:G\times H\to D$.
In this case, there is a $D$-equivariant diffeomorphism between $D/G$ equipped with the natural left $D$-action and the Lie group $H$ equipped with the natural left action of $H$ and the dressing action of $G$.
The following is an immediate consequence of Proposition \ref{prop:exponential-group}.
	
\begin{cor}
\label{cor:dressing-action-linearizable}
	Let $D=G\Join H$ be a matched pair of Lie groups.
	The Lie group pair $(D,G)$ has vanishing Atiyah class if and only if the dressing action of $G$ on $H$ is linearizable.
\end{cor}		
	
Applying this result to an integrable Lie bialgebra $\g$ and its double $ \mathfrak d = \g \oplus \g^*$, one sees that the vanishing of the Atiyah class implies that the dressing action of a Poisson--Lie group $G$ on its dual $G^*$ can be linearized in a neighborhood of the identity.
Its linearized action is of course the coadjoint action, see \cite{lu_multiplicative_1990}.

Alekseev and Meinrenken \cite{alekseev_2016_linearization} have proved that, for Poisson--Lie groups arising from an $r$-matrix, the Poisson structure on the dual group $G^*$ is 
linearizable in a neighborhood of the identity, while our result gives that the dressing action of $G$ on $G^*$ is linearizable. These results are clearly related: 
the symplectic leaves of the Poisson structure on $G^*$ coincide with the orbits of the dressing action \cite{lu_multiplicative_1990}. 
Notice that to linearize the Poisson structure and the dressing action are two different problems.
An interesting question is to investigate whether or not the vanishing of the Atiyah class, which gives the linearizability of the dressing action,
also implies the linearizability of the Poisson structure on $G^*$.

\begin{rmk}
	Since for Poisson--Lie groups, the leaves of the dressing action are precisely the symplectic leaves of the multiplicative Poisson structure while the leaves of the coadjoint action are precisely the symplectic leaves of the linear Poisson structure, it is tempting to believe that a diffeomorphism that intertwines both actions must be a Poisson diffeomorphism. As will be shown in \cite{abdeljellil__2015}, this is \emph{not} the case: there is in general no  $G$-equivariant Poisson diffeomorphisms between the Poisson--Lie group $G^*$ and the linear Poisson structure on $\mathfrak d/\g \simeq \g^*$.
\end{rmk}
	
	Given a $G$-invariant connection on $G/H$, one can $H$-equivariantly relate tensors on $G/H$ in a neighborhood of the unit to tensors on $\g/\h$ in a neighborhood of zero.
	
	\begin{cor}
	A Lie group pair $(G,H)$ has vanishing Atiyah class if and only if there exists an $H$-equivariant one-to-one correspondence between
	\begin{enumerate}[label=(\arabic*)]
	\item germs at $eH$ of $(k,l)$-tensors on $G/H$, and
	\item germs at $0$ of $(k,l)$-tensors on $\g/\h$ (considered as a manifold).
	\end{enumerate}
	\end{cor}
	This corollary might be more relevant when seen at the level of jets where it immediately implies the following result, which gives back Theorem 1.5 of Calaque, C\u{a}ld\u{a}raru and Tu \cite{calaque_pbw_2013} for Lie algebras over ${\mathbb R}$.
	\begin{cor}
	\label{cor:givescalaquecalderaruTu}
	A Lie group pair $(G,H)$ has vanishing Atiyah class if and only if there exists an $H$-equivariant one-to-one correspondence between
	\begin{enumerate}[label=(\arabic*)]
	\item jets at $eH$ of $(k,l)$-tensors on $G/H$, and
	\item elements in 
	  $\widehat S((\g/\h)^*) \bigotimes \otimes^k (\g/\h)^* \bigotimes \otimes^l  (\g/\h) $.
	\end{enumerate}
	\end{cor}


\subsection{An interpretation of the Molino class of regular foliations}

In this section, we give an interpretation of the Molino class of a foliation in terms of linearizability of monodromies.
All foliations are assumed to be regular. Throughout, ${\mathcal F}$ is a foliation of rank $r$ on a manifold $M$ of dimension $d$.
	
We first fix our vocabulary and notation.	We shall denote by ${\mathcal F}_m$ the leaf of ${\mathcal F} $ through $m  \in M$, but we shall use the latin letter $F$ to denote a chosen particular leaf  of  ${\mathcal  F} $. 
A submanifold $T$ of $M$ is said to be \emph{transverse at $m \in F$ to the leaf $F$} if it intersects transversally $F$ at $m$, i.e.\ if $ T_m M = T_{m}F \oplus T_mT$. 
A foliation ${\mathcal T} $ on $M$, defined in a neighborhood $ {\mathcal U}$ of the leaf $F$ in ${\mathcal F}$, is said to be \emph{transverse to  $ {\mathcal F}$} when for all $m \in {\mathcal U}$, the leaf ${\mathcal T}_m $ is transverse to the leaf ${\mathcal F}_m $. 
Given such a transverse foliation, and given a smooth path $\gamma:I=[0,1] \to F$ in $F$, a \emph{parallel lift of $\gamma$ starting at $n \in {\mathcal T}_{\gamma(0)}$} is a path $\tilde{\gamma}: I \to {\mathcal U} \subset M $ satisfying the following three conditions:
\begin{enumerate}
	\item $\tilde{\gamma}(0) = n $,
	\item for all $ t \in I$, $ \tilde{\gamma}(t) \in {\mathcal F}_n$, and
	\item for all $ t \in I$, $ \tilde{\gamma}(t) \in {\mathcal T}_{\gamma(t)}$.
\end{enumerate}
When $ \gamma$ is given, a parallel lift starting at $n$ exists for all $n $ in a neighborhood of $ \gamma(0)$ in ${\mathcal T}_{\gamma(0)}$. Moreover, when it exists, it is unique. 
This allows to define the \emph{parallel transport over a smooth path $\gamma$ in $F$} as being the germ of local diffeomorphism from ${\mathcal T}_{\gamma(0)}$ to ${\mathcal T}_{\gamma(1)}$ around $\gamma(0)$ mapping a point $ n \in {\mathcal T}_{\gamma(0)}$ to the value at $t=1$ of the parallel lift of $ \gamma$ starting at $n$.
Parallel transport is known to depend on the homotopy class of $\gamma$ only and, when applied to loops at $m\in F$, it yields a group morphism from $\pi_1(F) $ to the germs of diffeomorphisms of $ {\mathcal T}_m$ that we call \emph{the monodromy of ${\mathcal  F}$ at $m$ with respect to ${\mathcal T}$}.

\begin{rmk} 
The monodromy does not depend on the transverse foliation ${\mathcal T} $. More precisely, given two transverse foliations ${\mathcal T} $ and ${\mathcal T}'$, the submanifolds ${\mathcal T}_m$ and ${\mathcal T}_m'$ are always diffeomorphic, in a neighborhood of $m$, in a canonical manner: the germ of diffeomorphism is obtained by restricting ourselves to a neighborhood $ {\mathcal V}$ of $m$ where ${\mathcal F}$ is described by the fibers of a surjective submersion, then by mapping a point in ${\mathcal T}_m \cap {\mathcal V}$ to the point (unique if it exists) in ${\mathcal T}_m'$ which is in the same fiber of this surjective submersion. The monodromies of ${\mathcal F}$ at $m$ with respect to ${\mathcal T} $ and ${\mathcal T}'$ are intertwined by these canonical diffeomorphisms, as is easily seen.
\end{rmk}

We consider some particular families of submanifolds transversal to the leaves, that we describe as follows.

\begin{defn}
  Let ${\mathcal F}$ be a foliation of rank $r$ on a manifold $M$, and let $N_{\mathcal F} = TM / T{\mathcal F}$ be the normal bundle of this foliation. 
  A \emph{system of transversals} is a pair $({\mathcal U}(N_{\mathcal F}),p)$ with ${\mathcal U}(N_{\mathcal F})$ a neighborhood of the zero section in the normal bundle $ N_{\mathcal F} $ and $p:{\mathcal U}(N_{\mathcal F}) \to M$ a submersion admitting the zero section $i:M \to {\mathcal U}(N_{\mathcal F}) $ as right-inverse.
\end{defn}

Let us choose a metric on $M$, and let $p$ be the composition of the identification $ N_{\mathcal F} \cong T{\mathcal F}^\perp \subset TM$ with the exponential map of the Levi-Civita connection. Then there exists a neighborhood ${\mathcal U}(N_{\mathcal F})$ of the zero section such that $({\mathcal U}(N_{\mathcal F}),p) $ is a system of transversals. Hence:

\begin{lem}
  Every foliation admits a system of transversals. 
\end{lem}

To see why the previously defined pairs $({\mathcal U}(N_{\mathcal F}),p )$  deserve to be called ``system of transversals'', denote by ${\mathcal U}(N_{\mathcal F})_m$ the intersection of ${\mathcal U}(N_{\mathcal F})$ with the fiber over $m$ of the canonical projection $\pi:N_{\mathcal F}\to M$.
The restriction of $\pi$ to ${\mathcal U}(N_{\mathcal F})$ is still denoted by the same letter.
The local inverse theorem implies the following results:
\begin{enumerate}
  \item Upon shrinking ${\mathcal U}(N_{\mathcal F})$ if necessary, one can assume that for all  $ m \in {\mathcal U}(N_{\mathcal F})$, the image through $p$ of $ {\mathcal U}(N_{\mathcal F})_m$ is a submanifold of $M $ transverse  to $ {\mathcal F}_m$ at $m$ that we denote by $ {\mathcal T}^{p}_m$.
  \item Upon shrinking ${\mathcal U}(N_{\mathcal F})$ if necessary, one can assume that for all leaves $F$ of ${\mathcal F}$, the map $ p: \pi^{-1}(F) \to M  $ is a local diffeomorphism onto a neighborhood $ {\mathcal U}(F)$ of $F$ in $M$. We call that local diffeomorphism a \emph{local neighborhood diffeomorphism at the leaf $F$}.
  \item Upon shrinking ${\mathcal U}(N_{\mathcal F})$ if necessary, one can assume that for each leaf $F$ of ${\mathcal F} $, the disjoint union $\sqcup_{m \in F} {\mathcal T}^{p}_m$ is a foliation $\mathcal T^F$ of ${\mathcal U}(F)$ transverse to ${\mathcal F}$ (in a neighborhood of $F$).
\end{enumerate}
A system of transversals $({\mathcal U}(N_{\mathcal F}),p ) $ that satisfies these conditions shall be called a \emph{good system of transversals}. 

For $ ( {\mathcal U} (N_{\mathcal F}) , p ) $ a good system of transversals, we now transport, for each leaf $ F$  of $ {\mathcal F}$, the foliation $ {\mathcal  F}$ through the local neighborhood diffeomorphism at the leaf $ F$. This defines a foliation on $ {\mathcal U} (N_{\mathcal F}) $ that we call the \emph{monodromy foliation} and denote  by ${\mathcal M}_{\mathcal F} $.

Notice that:
\begin{enumerate}
\item When $M$ has dimension $d$ and ${\mathcal F}$ rank $r$, the monodromy foliation has rank $r$, but on a manifold
of dimension $2d-r$.
\item By construction, for each leaf $F$ of ${\mathcal F} $, the leaves of ${\mathcal M}_{\mathcal F}$ through points above $F$ are contained in $\pi^{-1}(F)$, i.e. the monodromy foliation is tangent to all the $\pi^{-1}(F)$.
\end{enumerate}
Indeed, since the restriction $\left.{\mathcal M}_{\mathcal F}\right|_{\pi^{-1}(F)}$ of ${\mathcal M}_{\mathcal F} $ to $\pi^{-1}(F)$ is diffeomorphic, as a foliated manifold, to a neighborhood of $F$ in $M$, the following holds true by construction:

\begin{prop}
\label{prop:monodromies}
Let ${\mathcal F}$ be a foliation on a manifold $M$, let $F$ be a leaf of $ {\mathcal F}$ and let $m\in F$.
For every good system of transversals $({\mathcal U}(N_{\mathcal F}),p)$, the local neighborhood diffeomorphism at the leaf $F$ intertwines
\begin{enumerate}
  \item the monodromy of $ {\mathcal F}$ at $m$ with respect to the transverse foliation $\mathcal T^{F}$, and 
  \item the monodromy of the restriction of $\left.{\mathcal M}_{\mathcal F}\right.$ to ${\pi^{-1}(F)}$ at the point $i(m)$ with respect to the transverse foliation given by the fibers of ${\mathcal U}(N_{\mathcal F})_{|_F} \to F$. 
\end{enumerate}
\end{prop}

Recall that every flat connection $\nabla$ on a vector bundle $E \to N$ gives rise to a regular foliation ${\mathcal F}^\nabla $ on the total space of the vector bundle $E$: the leaf through $e \in E$ is by definition the subset of all points in $E$ that can be reached from $e$ by parallel transport over paths in $M$. 
The leaves of this foliation have the dimension of $N$ and the zero section is a leaf. The same construction applies when ${\mathcal F}$ is a foliation on $M$ and $E \to M$ is a vector bundle equipped with a flat foliated connection. When applied to the Bott connection (the canonical $T{\mathcal F}$-action on $N_{\mathcal F}$),
this leads to a foliation on the normal bundle $ N_{\mathcal F} \to M$ that we call the \emph{Bott foliation}. We can now express the following notions.

\begin{defn}
	Let ${\mathcal F}$ be a regular foliation on a manifold $ M$.
	\begin{enumerate}
  	\item We say that \emph{the monodromy around a given leaf $F$ is linearizable} when for one (equivalently all) local transverse submanifolds $ {\mathcal T}_m$ to $ F$ at one (equivalently all) point $m \in M$, the monodromy map $ \pi_1(F) \to {\rm Diff}_m ({\mathcal T}_m)$ is conjugate to a linear representation of $\pi_1(F)$. 
  	\item We say that \emph{all the monodromies are simultaneously linerizable} when the Bott foliation and the monodromy foliation (computed with the help of any (equivalently all) good system of transversal) are diffeomorphic in a neighborhood of the zero section. 
	\end{enumerate}
\end{defn}

The next proposition justifies the appellation ``all monodromies are simultaneously linearizable'' used in the previous definition:

\begin{prop}
	Let ${\mathcal F}$ be a regular foliation on a manifold $ M$. If all the monodromies are simultaneously linearizable, then the monodromy around each leaf of $ {\mathcal F}$ is linearizable.
More precisely, the monodromy of each leaf $F$ of $ {\mathcal F}$ linearizes to the \emph{holonomy} of the Bott connection on the restriction to $F$ of the normal bundle.
\end{prop}

The previous proposition follows immediately from Proposition \ref{prop:monodromies} together with the following obvious lemma:

\begin{lem}
\label{lem:linearcase}
Let $\nabla$ be a flat connection on a vector bundle $E \to N$. For the associated foliation on $E$, the monodromy of the zero section $N\subset E$ with respect to the transverse foliation given by the fibers of $E \to N$ is by linear endomorphisms. 
These linear endomorphisms coincide with the holonomy of $\nabla$.
\end{lem}

We now arrive at the main result of this section:

\begin{thm}
\label{thm:monodromies-main}
	Let ${\mathcal F}$ be a regular foliation on a manifold $ M$. The Atiyah class of the Lie algebroid pair $(TM,T{\mathcal F})$ vanishes if and only if all the monodromies are simultaneously linearizable.
More precisely, if the monodromy of each leaf $F$ of $ {\mathcal F}$ linearizes to the \emph{holonomy} of the Bott connection on the restriction to $F$ of the normal bundle.
\end{thm}

By construction, the Atiyah class of the Lie algebroid pair $(TM,T{\mathcal F})$ coincides the \emph{Molino class} defined in \cite{molino_proprietes_1973}.
The proof of Theorem~\ref{thm:monodromies-main} shall of course make use of the Lie algebroid pair $(TM,T{\mathcal F}) $. 
It is not relevant to assume that this Lie algebroid pair integrates to a Lie groupoid pair as there are many natural counter-examples, see \cite[Examples 12]{moerdijk_integrability_2006}.
For $M$ simply connected, the Lie algebroid $TM$ integrates to the pair Lie groupoid $M \times M \toto M$, and $T{\mathcal F}$ integrates to a closed Lie subgroupoid of it if and only if $M/{\mathcal F}$ is a manifold (in which case it integrates to the Lie groupoid $M \times_{\mathcal F} M \toto M $ of pairs of points in $M$ which are in the same leaf of ${\mathcal F} $, i.e. in  the same fiber of the natural projection $ M \mapsto M/{\mathcal F}$).

\begin{proof}
  Let us spell out what the construction, given in Section \ref{sec:local-Lie-groupoids}, of the quotient space $\L/\A$ gives when applied to the particular case of $\ll=TM,\aa=T{\mathcal F}$.  A Lie groupoid $\L$ that integrates $\ll=TM $ is the pair groupoid $ {\L} :=M \times M \toto M $. The Lie algebroid $\aa =T{\mathcal F}$ acts on $ {\L}$ from both sides, i.e.
  \begin{enumerate}
    \item it acts on the left by mapping $ u \in \Gamma (T{\mathcal F})$ to the vector field on $ {\L}$ whose value at $(m,m') \in M \times M$ is $  (u_m,0) \in T_{(m,m')} M \times M \simeq T_m M\times T_{m'}M $,
    \item it acts on the right by mapping $ u \in \Gamma (T{\mathcal F})$ to the vector field on $ {\L}$ whose value at $(m,m') \in M \times M$ is $  (0,-u_{m'}) \in T_{(m,m')} M \times M \simeq T_m M\times T_{m'}M $.
  \end{enumerate} 
  The quotient space $ {\L}/{\A} $ can be made sense of as follows: it is the quotient of a neighborhood $ {\mathcal U}$ of the diagonal  $ \Delta(M) \subset M \times M$ by the foliation given by the infinitesimal right action of $\aa =T{\mathcal F} $. By shrinking $ {\mathcal U}$ if necessary, we can assume this quotient to be a manifold, that we denote by $ {\L}/{\A} $.

  Now, let $ ( {\mathcal U} (N_{\mathcal F}) , p ) $ be a good system of transversals, with ${\mathcal U} (N_{\mathcal F}) \subset N_{\mathcal F} $. There is a natural map  
  from $ {\mathcal U} (N_{\mathcal F}) $ to $ M \times M$ mapping $x \in  N_{\mathcal F}|_m \cap {\mathcal U} (N_{\mathcal F})$ to $ (m,p(m))$. Upon shrinking ${\mathcal U} (N_{\mathcal F})$ if  necessary, we can assume that this map
  takes values in the open set $ {\mathcal U}$.   Upon shrinking ${\mathcal U} (N_{\mathcal F})$ once more if  necessary, we  can then assume that the composition 
  \begin{equation}
  \label{diiffeoPhi}
    \Phi
    : {\mathcal U} (N_{\mathcal F})
    \to {\mathcal U}
    \to {\L}/{\A}
  \end{equation}
  is a local diffeomorphism. By construction, the diffeomorphism $\Phi$ defined in (\ref{diiffeoPhi}) 
  intertwines 
  \begin{enumerate}
    \item the monodromy foliation of ${\mathcal U} (N_{\mathcal F})$, and
    \item the foliation on $\L/\A$ given by the left $\aa=T{\mathcal F}$-action.
  \end{enumerate}

  By Corollary \ref{thm:exponential-local}, the Atiyah class is zero if and only if there exists an $\aa$-equivariant diffeomorphism between $ {\L}/{\A}$ and $\ll/\aa$. Since the foliation induced by the $\aa$-action on $\ll/\aa$ is precisely the Bott foliation, we arrive at the desired condition in view of Lemma \ref{lem:linearcase}.
\end{proof}


\bibliographystyle{amsabbrvlinks} 
\bibliography{InvariantConnections}

\end{document}